\documentclass[11pt, reqno]{amsart}

\usepackage[a4paper,inner=2.0cm,outer=2cm,top=3cm,bottom=3cm,pdftex]{geometry}

\usepackage{color}
\usepackage{amsmath,amssymb,amsfonts,amsthm}
\usepackage{graphicx}
\usepackage{float}
\usepackage{enumerate}
  \usepackage[shortlabels]{enumitem}

 \newtheorem{theorem}{Theorem}[section]
 \newtheorem{corollary}[theorem]{Corollary}
 \newtheorem{lemma}[theorem]{Lemma}
 \newtheorem{proposition}[theorem]{Proposition}

 \newtheorem{definition}[theorem]{Definition}
 
 \theoremstyle{conjecture}

 \newtheorem{assumption}[theorem]{Assumption}

 \theoremstyle{remark}
 \newtheorem{remark}[theorem]{Remark}

 \newcommand{\eps}{\varepsilon}

\newcommand{\norm}[1]{\left\Vert#1\right\Vert}
\newcommand{\normm}[1]{{\left\vert\kern-0.25ex\left\vert\kern-0.25ex\left\vert #1 
    \right\vert\kern-0.25ex\right\vert\kern-0.25ex\right\vert}}
\newcommand{\abs}[1]{\left\vert#1\right\vert}
\newcommand{\set}[1]{\left\{\,#1\,\right\}}

\newcommand{\inner}[1]{\left(#1\right)}

\newcommand{\comi}[1]{\left<#1\right>}
\newcommand{\com}[1]{\left[#1\right]}

\newcommand{\lab}[1]{\label{#1}}

\begin{document}

\title[Well-posedness in Gevrey space  for the Prandtl equation]{Well-posedness in Gevrey Function space  for the Prandtl equations with non-degenerate critical points}%

\author{ Wei-Xi Li \and Tong Yang}

\date{}

\address{\newline 
 Wei-Xi Li,
School of Mathematics and Statistics, and Computational Science Hubei Key Laboratory,   Wuhan University,  430072 Wuhan, China
  }

\email{
wei-xi.li@whu.edu.cn}

\address{\newline 
 Tong Yang, Department of Mathematics, Jinan University, Guangzhou,   China $\&$	
	\newline\indent	
Department of Mathematics, City University of Hong Kong, Hong Kong
  }

\email{
matyang@cityu.edu.hk}

\begin{abstract}
In the paper, we study the well-posedness of the  Prandtl system 
without monotonicity and analyticity assumption. Precisely, 
for any index $\sigma\in[3/2, 2],$ we obtain the local in time  well-posedness in the space of   Gevrey class  $G^\sigma$ in the tangential variable  and Sobolev class in the normal variable so that the monotonicity condition on the
tangential velocity  is not needed to overcome
the loss of tangential derivative. This answers the open question raised in 
the paper of D. G\'{e}rard-Varet and N. Masmoudi
[{\it   Ann. Sci. \'{E}c. Norm. Sup\'{e}r}. (4) 48 (2015), no. 6, 1273-1325],
  in which the case $\sigma=7/4$ is solved.
\end{abstract}



\maketitle

\section{Introduction and  main results}

The Prandtl equations introduced by Prandtl  in 1904  describe 
the behavior of the incompressible flow near a rigid 
wall at high Reynolds number:
\begin{equation}\label{prandtl}
\left\{\begin{array}{l}\partial_t u^P + u^P \partial_x u^P + v^P\partial_yu^P -\partial_y^2u^P + \partial_x p
=0,\quad t>0,\quad x\in\mathbb R,\quad y>0, \\
\partial_xu^P +\partial_yv^P =0, \\
u^P|_{y=0} = v^P|_{y=0} =0 , \quad  \lim_{y\to+\infty} u =U(t,x), \\
u^P|_{t=0} =u^P_0 (x,y)\, ,
\end{array}\right.
\end{equation}
where      $u^P(t,x,y)$ and $v^P(t,x,y)$ represent the tangential and normal
velocities of the boundary layer, with $y$ being the scaled normal variable
to the boundary, while $U(t,x)$ and $p(t,x)$ are the values  on the
boundary of the tangential velocity and pressure of the outflow
satisfying  the Bernoulli law
\[
\partial_t U + U\partial_x U +\partial_x p=0.
\]
We refer to \cite{mamoudi,oleinik-3} for the mathematical derivation 
and background of this fundamental system in the field
of boundary layer.

By using the divergence free condition,
one can  represent
 $v$ in terms of $u$ so that the above system is reduced
 a scalar equation. Moreover,
note that the above $U$ and $p$ are known functions coming from the outflow
so that the Prandtl system is a degenerate parabolic  mix-type equation with
   loss
of derivative in the tangential direction $x$ because of  the term $v\partial_y u$.  In fact,
this is the main difficulty  in the  study of this boundary layer system.  

 Up to now, the well-posedness on 
the Prandtl system is achieved in various function spaces. Precisely,   when  the initial data satisfy the monotonic condition, that is, when the
tangetial velocity is monotonic with respect to $y$, in the classical work
by  Oleinik and her collaborators,  they  obtained the local-in-time
well-posedness by using  Crocco transformation. And this result together with some of her other works  were well presented in
the monograph \cite{oleinik-3}.   Recently,  Alexandre-Wang-Xu-Yang \cite{awxy}  and Masmoudi-Wong \cite{MW} independently obtained   the  well-posedness in the Sobolev space by the virtue of energy method  instead of the Crocco transformation,  where the key observation in their proofs  is the cancellation 
of the loss derivative terms.    On the other hand,
  for the initial data without the monotonicity assumption, it is natural to  perform estimate in the space of analytic functions, and in this context, the  well-posedness results were achieved by Sammartino and Caflisch, after the earlier work of Asano \cite{asa};  cf also \cite{cannone, KV} for the improvement.  The first result that does not require monotonicity and analyticity   was established  by    G\'erard-Varet and Masmoudi \cite{GM} in which 
 they  obtained the well-posedness in the Gevrey space $G^{7/4}$. In
fact, our paper is motivated by
\cite{GM} and we give an affirmative answer to an open question raised in
it. Also  in  the very recent work of Chen-Wang-Zhang \cite{cwz},
   the well-posedness  for the linearized Prandtl equation is studied
in Gevrey space $G^\sigma$  for any index $1\leq \sigma<2.$

 Recall 
that the Gevrey class, denoted by $G^\sigma, \sigma\geq 1$,  is an intermediate 
function space between analytic functions   and $C^\infty$ functions.  Note that  the Gevrey space $G^\sigma,\, \sigma>1$ contains compactly supported functions that are more physical, and this is the  main  difference from analytic  functions.    We also refer to \cite{lwx} for the smoothing effects in Gevrey space under the monotonicity assumption, and  the global weak solutions  by Xin-Zhang \cite{xin-zhang}.  
 On the other hand, without the monotonicity assumption on the
tangential velocity field,  the degeneracy may  cause  strong instability so
that the system is ill-posed  in Sobolev spaces, cf. \cite{e-2, GV-D, LY} and references therein.  

Without the assumption on monotonicity and analyticity, in the
recent interesting paper
 \cite{GM},  the authors established $G^{7/4}$ well-posedness for Prandtl equation with non-degenerate critical points with respect to the normal variable, and they also conjectured the result
should be valid for $G^2.$  In this paper, we will give an affirmative answer
to this conjecture. In fact, we show
 the well-posedness in all Gevrey space $G^\sigma$ with $\sigma\in[3/2, 2]$ and this includes the case studied in \cite{GM}. In addition,
   we  believe  the well-posedness  result can be extended, with
some new technique such as subelliptic estimates,  to $\sigma\in[1,3/2]$.     Finally, as the aforementioned works,  the present paper also
aims at giving insight on  the justification of inviscid limit for the Navier-Stokes equation with physical boundary, for this, we refer to \cite{GMM, GN, Ma} and the references therein for the recent progress. 

To have a clear presentation,  we will construct  a solutions $u^P$ that  is  a small perturbation around a shear flow,  that is, $u^P(t,x,y)=u^s(t, y)+u(t, x,y)$.   For this,  we suppose that the initial data $u^P_0$ in \eqref{prandtl} can be written as 
\begin{eqnarray*}
	u^P_0(x,y)=u_0^s(y)+u_0(x,y),
\end{eqnarray*}
with $u_0^s$ being independent of $x$ variable.   Then we reduce the original Prandtl equation \eqref{prandtl}   to  the following two time evolutional equations,    one of which is the equation for the shear flow $(u^s, 0)$ with $u^s$ solving 
\begin{equation}\label{eqshearflow}
\left\{
	\begin{aligned}
	&\partial_t   u^s-\partial_y^2 u^s=0,\\
	&u^s \big|_{y=0}=0,\quad\lim_{y\rightarrow +\infty} u^s=1,\\
	&u^s\big|_{t=0}=u_{0}^s.
\end{aligned}	
\right.
\end{equation} 
and the another reads
\begin{equation}\label{++repran}
\left\{
	\begin{aligned}
	&\partial_t     u+\inner{u^{s}+   u}\partial_x  u +  v \partial_y \inner{  u^s+u}-\partial_y^2    u=0,\\
	&   u \big|_{y=0}=0,\quad\lim_{y\rightarrow +\infty}    u=0,\\
	&   u\big|_{t=0}=  u_{0},
\end{aligned}	
\right.
\end{equation}
where 
\begin{eqnarray*}
  v=- \int_0^y\partial_x u(x,\tilde y)\,d\tilde y.\end{eqnarray*}
Note the equation \eqref{eqshearflow} is the heat equation and the well-posedness problem is well studied.  In this paper, we assume that the initial datum $u_0^s$ in \eqref{eqshearflow} admits non-degenerate critical points.  Precisely, we impose

\begin{assumption}[Assumption on the initial data $u_0^s$]
\label{maas} There exists a $y_0>0$ such that
 $  u_0^s\in C^6(\mathbb R_+)$  satisfies    the
following properties (see Figure 1):
\begin{enumerate}[(i)]
\item $\frac{du_0^s}{dy}  \inner{y_0}=0$ and  $\frac{d^2u_0^s}{dy^2}  \inner{y_0}\neq 0.$  Moreover, there exist  $0<\delta<y_0/2$ and a constant $c_0$ such that
\begin{eqnarray*}
\forall~y\in\big[ y_0-2\delta,y_0+2\delta\big],\quad \Big|\frac{d^2u_0^s}{dy^2}  \inner{y}\Big|	\geq c_0. 
\end{eqnarray*}
 
\item There exists a constant $0<c_1<1$   such that 
  \begin{eqnarray*}
	\forall~y \in  \big  [0,  y_0-\delta \big  ]\cup  \big [ y_0+ \delta ,~+\infty \big[ ,\quad c_1 \comi y^{-\alpha}\leq \Big|\frac{du_0^s(y)}{dy}\Big|\leq  c_1^{-1} \comi y^{-\alpha} 
\end{eqnarray*}
for some $\alpha>1,$
and that
\begin{eqnarray*}
	\forall~y\geq 0,\quad \abs{\frac{d^j u_0^s(y)}{dy^j}}\leq  c_1^{-1}  \comi y^{-\alpha-1} ~~{\rm for} ~ 2\leq j\leq 6.
\end{eqnarray*}
\item  The compatibility condition holds, that is, $u_0^s\big|_{y=0}=\partial_y^2 u_0^s\big|_{y=0}=0$ and  $
   u_0^s(y) \rightarrow 1
$ as $y\rightarrow +\infty.$
 \end{enumerate}
\end{assumption}

\begin{figure}[H]\label{fig1}
\begin{center}
 \includegraphics{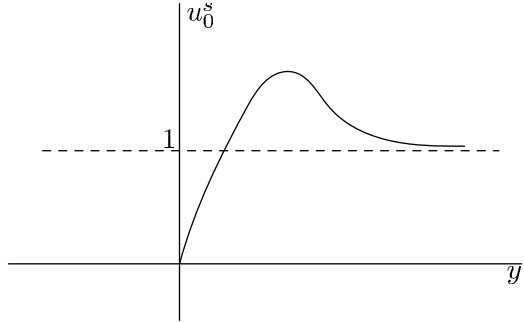}
 \end{center}
\caption{\small {An {example} of $u_0^s$}}
\label{fig region}
\end{figure}

\begin{remark}
\begin{enumerate}[(i)]
\item 
For  brevity of  presentation,  we only consider the case 
when the initial datum admits  one non-degenerate critical  point.  The
 result can be generalized to the case when there are several  non-degenerate critical points 
with slight modification.  
\item  {The initial datum $u_0^s$ here is not monotonic anymore. Note that in the work  \cite{awxy}  the monotonicity condition  is  required to overcome the loss of derivative in the $x$ variable. }  
\end{enumerate}
\end{remark}

\begin{proposition} [Well-posedness for the shear flow] \label{proshf} Let the initial data $u_0^s$ satisfy the conditions  in Assumption \ref{maas}.  Then
there exists a constant   $T_s>0$  such that the heat equation   \eqref{eqshearflow} admits a unique solution $u^s$ in $C\inner{[0,T_s];~C^6(\mathbb R_+)}$.
  In addition,    for any $t\in[0,T_s],$ we have, using the notation $\omega^s=\partial_y u^s,$
 \begin{eqnarray*}
 \forall~y\in\big[y_0-\frac{7}{4}\delta, y_0+\frac{7}{4}\delta\big],\quad \abs{\partial_y\omega^s(t,y)} \geq c_0/2,
 \end{eqnarray*}
 \begin{eqnarray*}
 \forall~y \in  \big  [0,  y_0-\frac{5 }{4}\delta \big  ]\cup  \big [ y_0+ \frac{5 }{4}\delta ,~+\infty \big[ ,\quad 2^{-1}c_1 \comi y^{-\alpha}\leq \abs{\omega^s(t,y)}\leq  2 c_1^{-1} \comi y^{-\alpha},
 \end{eqnarray*}
 and  
 \begin{eqnarray*}
 \forall~y\geq 0,\quad \abs{\partial_y^j \omega^s(t, y)}\leq   2c_1^{-1} \comi y^{-\alpha-1} ~{\rm for}~  1\leq j\leq 5. 
 \end{eqnarray*}
Recall $c_0,c_1, \delta$ are the constants given in  Assumption \ref{maas}.
\end{proposition}

Observe that the solution to \eqref{eqshearflow} has explicit representation by virtue of heat kernels. Then  the above proposition  follows from
 direct estimation. For  brevity, we omit its proof  and refer to  Lemma 2.1 in the second version of \cite{xz} for detailed discussion.  So it remains to solve \eqref{++repran}, which is the main part of the paper.  And we will solve the equation in the framework of  Gevery space in $x$ and Sobolev space in $y.$  To state the main result,  we first introduce the function spaces to be used.

\begin{definition}[Gevrey space in tangential variable]
\label{defgev} 
Let $\alpha $ be the number  given  in Assumption \ref{maas}, and let $\ell$ be a fixed number
satisfying that
\begin{equation}\label{ellalpha} 
	\ell >3/2, \quad \alpha\leq\ell< \alpha+\frac{1}{2}. 
\end{equation} 
 With each pair $(\rho,\sigma)$, $\rho>0,\sigma\geq 1, $ we associate a Banach space $X_{\rho,\sigma}$, equipped with the norm $\norm{\cdot}_{\rho,\sigma}$ that consists of all the smooth functions $f$ such that  $\norm{f}_{\rho,\sigma}<+\infty,$ where   
\begin{eqnarray}\label{trinorm}
\begin{split}
	\norm{f}_{\rho,\sigma} \stackrel{\rm def}{ =}&\sup_{m\geq 6}   \frac{\rho^{m-5}}{\big[\inner{m-6}!\big]^{\sigma}} \big\|\comi y^{\ell-1}\partial_x^{m} f\big\|_{L^2} +\sup_{m\geq 6}   \frac{\rho^{m-5}}{\big[\inner{m-6}!\big]^{\sigma}}\big\|\comi y^{\ell}\partial_x^{m } (\partial_y f)\big\|_{L^2}  \\
	&+\sup_{0\leq m\leq 5}  \big(\big\|\comi y^{\ell-1}\partial_x^{m} f\big\|_{L^2}+\big\|\comi y^{\ell}\partial_x^{m} (\partial_yf)\big\|_{L^2}\big) \\
	&+\sup_{\stackrel{1\leq j\leq 4}{i+j\geq 6}} \frac{\rho^{i+j-5}}{\big[\inner{i+j-6}!\big]^{\sigma}}\big\|\comi y^{\ell+1} \partial_x^{i} \partial_y^j (\partial_y f)\big\|_{L^2}  +\sup_{\stackrel{1\leq j\leq 4}{i+j\leq 5}}  \big\| \comi y^{\ell+1}\partial_x^{i} \partial_y^j (\partial_y f)\big\|_{L^2}.
\end{split}
\end{eqnarray}
  \end{definition}
 
\begin{remark}
For the classical Gevrey space $G^\sigma=\cup_{L>0}G^\sigma(L)$  in $x$ variable,
$f\in G^\sigma(L)$ if  the following estimates hold:
\begin{eqnarray*}
 \forall~m\geq 0,\quad	 \big\|\comi y^{\ell-1}\partial_x^{m} f\big\|_{L^2}+\big\|\comi y^{\ell }\partial_x^{m} (\partial_y f)\big\|_{L^2}  \leq   L^{m+1} (m!)^\sigma,
 	 \end{eqnarray*}
 	 and 
 	 \begin{eqnarray*}
 \forall~i\geq 0,~\forall~1\leq j\leq 4,	\quad  	 \big\|\comi y^{\ell+1} \partial_x^{i} \partial_y^j (\partial_y f)\big\|_{L^2} \leq   L^{i+j+1} [(i+j)!]^\sigma. 
 	 \end{eqnarray*}
 The  space $X_{\rho,\sigma}$ given in Definition \ref{defgev} is equivalent to the classical Gevrey space $G^\sigma$ in the following sense.  If  $f\in X_{\rho,\sigma}$ for some $\rho>0$ then we can find a constant $C$ such that 
 $\norm{f}_{\rho,\sigma}\leq C. $  
  Thus direct calculation shows  $f\in G^\sigma(L)$ 
if we choose  $$L= {1\over \rho}+\sup_{0\leq m\leq 5}  \big(\big\|\comi y^{\ell-1}\partial_x^{m} f\big\|_{L^2}+\big\|\comi y^{\ell}\partial_x^{m} (\partial_yf)\big\|_{L^2}\big)+\sup_{\stackrel{1\leq j\leq 4}{i+j\leq 5}} \big\| \comi y^{\ell+1}\partial_x^{i} \partial_y^j (\partial_y f)\big\|_{L^2}+C.$$
Conversely, if $f\in G^\sigma(L)$, then  $f\in X_{\rho,\sigma}$,  provided $\rho$ is chosen in such a way that 
\begin{eqnarray*}
\forall~m\geq 6,\quad	L^{m+1} (m!)^\sigma \leq  \rho ^{-(m-5)} \com{(m-6)!}^\sigma.
\end{eqnarray*}	
\end{remark}
 
 In view of  the definition $\norm{\cdot}_{\rho,\sigma},$  we see the the order of $y$ derivatives is at most $5.$   Then, if the  equation \eqref{++repran} is well-posed in $X_{\rho,\sigma}$,  the initial data $u_0$ should satisfy the following compatibility conditions, using the notation $\omega_0=\partial_yu_0,$ 
 
 \begin{equation}\label{comcon}
 \left\{
 \begin{aligned}
 &u_0|_{y=0}=\partial_y\omega_0|_{y=0}=0,\\
 &\partial_y^3\omega_0\big |_{y=0}=  \inner{\omega_0^s+\omega_0}\partial_x\omega_0\big |_{y=0}.	
	\end{aligned}
	\right.
	\end{equation}

 Now we state the main result in this paper as follows.

\begin{theorem}\label{mainthm}
	For $\sigma\in[3/2, 2]$, let the initial datum $u_0$ in \eqref{++repran} belong to $X_{2\rho_0,\sigma}$ for some $\rho_0>0$ and moreover
	\begin{eqnarray*}
		\norm{u_0}_{2\rho_0,\sigma}\leq \eta_0
	\end{eqnarray*} 
	for some $\eta_0>0.$  Suppose      that the   compatibility condition  \eqref{comcon} holds for $u_0$.
	 Then \eqref{++repran} admits a unique solution $u\in L^\infty\big([0,T];~X_{\rho,\sigma}\big)$ for some $T>0$ and some $0<\rho<2\rho_0,$ provided  $\eta_0$   is sufficiently small. 
\end{theorem}

\begin{remark}
(i) For  clear presentation, we consider the solution
as a perturbation around the shear flow. In fact,  the
 method can be applied to the general periodic case   {
studied in \cite{GM},  and we will  clarify further in Section \ref{sectiongeneral} why our result holds for the general  case without requiring the small perturbation around a shear flow.}  
\newline
(ii) As pointed  out in \cite{GM}, it is natural, inspired by \cite{GV-D},  to ask whether the $\sigma=2$ is the critical Gevrey index for the well-posedness for Prandtl equation.

\end{remark}

\noindent {\bf The methodologies}. 
At the end of the introduction, we will present the main methodologies
used in the proof.  

(i)  After applying $\partial_x^m$ to the equation \eqref{++repran} for the velocity, the main difficulty arises from the term 
\begin{eqnarray*}
(\partial_x^m v) \inner{\omega^s+\omega},
\end{eqnarray*}
which results in the lost of derivative in $x$ variable. Under Oleinik's   monotonicity assumption on the tangential velocity field, this can be  overcome  by using the cancellation  introduced by AWXY \cite{awxy} and Masmoudi-Wong \cite{MW}. In fact, this cancellation method  works at least in the domain where $u^s+u$ admits monotonicity.   Precisely, we  apply $\partial_x^m$ to the equation  
 for the vorticity $\omega=\partial_y u$
\begin{eqnarray*} 	
 \partial_t\omega + (u^s+u) \partial_x\omega  + v\partial_y(\omega^s+\omega) - \partial_y^2\omega=0,
\end{eqnarray*}
in which  the most difficult term  is $(\partial_x^m v) \inner{\partial_y\omega^s+\partial_y\omega}$.  To capture the
 cancellation, one can work on the  function, introduced in \cite{GM},
\begin{eqnarray*}
		f_m=\chi_1\partial_x^m \omega-\chi_1\frac{  \partial_y\omega^s+\partial_y \omega }{\omega^s+\omega }\partial_x^m u,\quad m\geq 1,
\end{eqnarray*}
where $\chi_1$ is a smooth function supported in the monotonic region.  

(ii) As for the domain near the critical points, we do not have the monotonicity anymore. One of the new  observations in this paper
 is that we can also apply the cancellation  to the  equation for the 
vorticity  and the equation for $\partial_y \omega$ 
\begin{eqnarray*}
	 \partial_t \inner{\partial_y \omega}+ u \partial _x\inner{\partial_y \omega}  + v\partial _y \inner{\partial_y \omega} - \partial_{y}^2\inner{\partial_y \omega}=- (\omega^s+\omega) \partial_x\omega+\inner{\partial_y\omega^s+\partial_y\omega}\partial_x u,
\end{eqnarray*}    
by using another auxilliary function
\begin{eqnarray*}
	h_m=	\chi_2 \partial_x^m\partial_y \omega-\chi_2\frac{  \partial_y^2\omega^s+\partial_y^2 \omega }{\partial_y\omega^s+\partial_y\omega }\partial_x^m \omega,
\end{eqnarray*}
where   $\chi_2$ is a cut-off function compactly supported in the region admitting the non-degenerate critical points. However, even
with this, we also have the loss of $x$ derivative because 
\begin{eqnarray*}
	g_{m+1}\stackrel{\rm def}{=} \partial_x^{m}\big[(\omega^s+\omega) \partial_x\omega-\inner{\partial_y\omega^s+\partial_y\omega}\partial_x u\big]
\end{eqnarray*}
appears in the equation for $h_m.$  Nonetheless, we can use again the cancellation method to the equations for the  velocity and the vorticity, to obtain a equation for $g_{m+1}.$ Precisely, we apply $\partial_x$ to the equations for velocity $u$ and for vorticity $\omega=\partial_yu$, and then multiply respectively the obtained equations by the factors $\partial_y\omega^s+\partial_y\omega$ and $ \omega^s+\omega$ respectively, and finally subtract one from another. We then
 obtain the equation for $g_1\stackrel{\rm def}{ =}(\omega^s+\omega) \partial_x\omega-\inner{\partial_y\omega^s+\partial_y\omega}\partial_x u$ as follows.
\begin{eqnarray*}
		 \Big(\partial_t    + \inner{u^{s}+u} \partial_x +v\partial_y  -\partial_y^2    \Big)
		 g_1=   2\inner{\partial_y^2\omega^s+\partial_y^2\omega} \partial_x \omega  -2\inner{\partial_y\omega^s+\partial_y\omega} \partial_x \partial_y\omega.
\end{eqnarray*} 
Note that the order of $x$ derivative for terms on right hand side is  equal to $1$ that is the same as in the representation of $g_1.$ The above equation  allows us to perform estimation on $g_{m+1}=\partial_x^m g_1$ in Gevrey norm by standard energy method.     

(iii) From the above procedure, we have the upper bound, by energy method, for  the auxilliary functions $f_m$ and $h_m.$ It remains to control the original $\partial_x^mu$ and $\partial_x^m\omega$ as well as the mixed derivatives, in terms of the auxilliary functions. This is clear when there is no cut-off functions $\chi_i$ involved, by virtue of the Hardy-type inequality (see \cite{MW} for instance under the   monotonicity assumption).    In  case considered
in this paper, we first follow the cancellation idea used in \cite{GM}, by taking $L^2$  inner product with   $\frac{\chi_2 \partial_x^m\omega}{\partial_y\omega^s+\partial_y\omega}$  on both sides of the equation for $\partial_x^m\omega$,  to obtain the estimate on $\chi_2\partial_x^m\omega$. Then by using the representation of $h_m$, we can  derive similar estimate on $\chi_2\partial_x^m\partial_y\omega$ from 
those on $h_m.$  Roughly speaking, this implies  $\chi_2\partial_x^m\partial_y\omega$ behaves similarly as the terms with $m$ order derivatives involved, rather than the $m+1$ order of mixed derivatives in Definition \ref{defgev}.     And this 
is the advantage of the new auxilliary function $h_m$ introduced in this
paper and this enables us to extend the well-posedness of the Prandtl
system from the Gevrey index $\sigma=7/4$ obtained
in \cite{GM} to $\sigma\in[3/2, 2].$

The rest of the paper is organized as follows.  Section  \ref{sec2}-\ref{sec5}  are devoted to the proof of  the uniform estimate  in Gevrey norm  for the approximate solutions to a regularized Prandtl equation.   In Section \ref{sec6},  we will give the proof of existence of the regularized Prandtl equation and in Section \ref{sec7} we will prove the main result of this paper.  {We explain in Section \ref{sectiongeneral} 
why the main result in this paper holds for the general  initial data rather than the small perturbations around a shear flow.}  The proofs
of some technical lemmas will be given in the Appendix.

\section{Regularized Prandtl equation and uniform estimates in Gevrey norm}\label{sec2}

In this section as well as Sections \ref{sec3}-\ref{sec6}, we will  study the initial-boundary problem  for the
following regularized Prandtl type equation of \eqref{++repran}
by  recalling $u^s$ given in Proposition \ref{proshf},       
\begin{equation}\label{regpran}
\left\{
	\begin{aligned}
	&\partial_t     u_\eps+\inner{u^{s}+   u_\eps}\partial_x  u_\eps +  v_\eps \partial_y \inner{  u^s+u_\eps}-\partial_y^2    u_\eps-\eps\partial_x^2 u_\eps=0,\\
	&   u_\eps \big|_{y=0}=0,\quad\lim_{y\rightarrow +\infty}    u_\eps=0,\\
	&   u_\eps\big|_{t=0}=  u_{0},
\end{aligned}	
\right.
\end{equation}
where 
$\eps>0$ is an arbitrarily  small number and $
  v_\eps=- \int_0^y\partial_x u_\eps(x,\tilde y)\,d\tilde y.$    We remark  the  regularized equation above shares the same compatibility condition \eqref{comcon} as the original one \eqref{++repran}. 
  
  The existence  of 
solutions to  \eqref{regpran} will be
 given in Section \ref{sec6},  where the life span $T^{*}_\eps$  may depend on the $\eps.$  Thus in order to obtain the solution to the original equation by letting $\eps\rightarrow 0,$ we  need  an uniform estimate,
for example,   in the Gevrey norm for $u_\eps,$  that will be stated in this section with the proof given in   Sections \ref{sec3}-\ref{sec5}.    To simplify the notations, we will use  the notations  $\omega_\eps=\partial_y u_\eps$  and  $\omega^s =\partial_y u^s$ from now on.
  
Throughout the paper, we will work on those solutions $u_\eps$  that   the properties listed in Proposition \ref{proshf} for $u^s$ can be
 preserved by $u^s+u_\eps.$ Precisely, we suppose that the solution  $u_\eps\in L^\infty\inner{[0, T];~X_{\rho_0,\sigma}}$   to \eqref{regpran}  has the following properties.  For any      $t\in[0, T]$  and  any $x\in\mathbb R,$   we have
 \begin{eqnarray}\label{condi}
\left\{
\begin{aligned}
  &\abs{\partial_y\omega^s(t,y)+\partial_y\omega_\eps(t,x,y)}\geq \frac{c_0}{4}, ~\,{\rm if}~~\,y\in\big[y_0-\frac{7}{4}\delta, y_0+\frac{7}{4}\delta\big], \\
  &4^{-1}c_1 \comi y^{-\alpha}\leq \abs{  \omega^s(t,y)+\omega_\eps(t,x,y)} \leq  4 c_1^{-1} \comi y^{-\alpha},~~\,{\rm if}~~\,y\in\big[0,  y_0-\frac{5 }{4}\delta \big  ]\cup  \big [ y_0+ \frac{5 }{4}\delta ,~+\infty \big[ ,\\
  & \abs{\partial_y\omega^s(t,y)+\partial_y\omega_\eps(t,x,y)}\leq 4c_1^{-1} \comi y^{-\alpha-1} ~~\, \textrm {for  } ~y\geq 0, \\
 &\sum_{1\leq j\leq 2}\Big(\norm{\comi y^{\ell-1}\partial_x^j u_\eps}_{L^\infty}  +\norm{\partial_x^{j-1} v_\eps}_{L^\infty}+  \norm{\comi y^{\ell}\partial_x^j \omega_\eps}_{L^\infty}\Big)+ \sum_{1\leq i, j \leq 2}\norm{\comi y^{\ell+1}\partial_x^i\partial_y^j\omega_\eps}_{L^\infty}\leq 1, 
 \end{aligned}
 \right.
\end{eqnarray}
where  $c_0, c_1$ and $\delta$ are the constants  given in  Assumption \ref{maas}.  
 
  According to the properties  \eqref{condi} above, we can divide the normal direction $y\geq 0$ into two parts, one is near the critical points of  $u^s+u_\eps$, and another one is away from the critical points where  $u^s+u_\eps$ admits
the monotonicity condition. That is, 
we can find two non-negative $C^\infty$ smooth functions $\chi_1$ and $\chi_2$
 depending only on $y$ such that
 \begin{equation}\label{chi1}
 0\leq \chi_1\leq1,\quad	\chi_1\equiv 1 ~\,\textrm{on}\, ~ \big]-\infty,~y_0-\frac{3}{2}\delta\big ] \cup \big[y_0+\frac{3}{2}\delta, ~+\infty\big[\, , \quad \chi_1\equiv 0 ~\,\textrm{on}\, ~ \big[y_0-\frac{5}{4}\delta, y_0+\frac{5}{4}\delta\big],
\end{equation}
and  
 \begin{equation}\label{cutofffu}
   0\leq \chi_2\leq1,\quad	\chi_2\equiv 1 ~\,\textrm{on}\, ~\big[y_0-\frac{3}{2}\delta, y_0+\frac{3}{2}\delta\big],\quad {\rm supp}\, \chi_2\subset\big[y_0-\frac{7}{4}\delta, y_0+\frac{7}{4}\delta\big].
\end{equation}
From the properties listed in \eqref{condi}, it follows that
  $\abs{\omega^s+\omega}>0$ on   supp\,$\chi_1$, and  $\abs{\partial_y\omega^s+\partial_y\omega}>0$ on   supp\,$\chi_2.$  Moreover,
\begin{equation}\label{suppch2}
	\chi_1'=\chi_1'\chi_2,\quad \chi_2'=\chi_2'\chi_1, ~~{\rm and } ~~\inner{1-\chi_2}=\inner{1-\chi_2}\chi_1,
\end{equation}
because $\chi_2\equiv 1 $ on  supp\,$\chi_1',$   $\chi_1\equiv 1 $ on  supp\,$\chi_2',$  and $\chi_1\equiv 1$  on  supp\,$(1-\chi_2).$ Here and  throughout the paper, $f'$ and $f''$ stand for the first and the
second order derivatives of  $f$.

\begin{definition}
Let    $\chi_1$ and $\chi_2$ given above and let  $u_\eps$  satisfy the properties  \eqref{condi}.    For $m\geq 1,$  we define  three  auxilliary functions $f_{m,\eps}, h_{m,\eps}$ and
$g_{m,\eps}$  according to the cancellation property:
 \begin{equation}\label{fungeps}
	f_{m,\eps}=\chi_1\partial_x^m \omega_\eps-\chi_1\frac{  \partial_y\omega^s+\partial_y \omega_\eps }{\omega^s+\omega_\eps }\partial_x^m u_\eps=\chi_1\inner{\omega^s+\omega_\eps}  \partial_y\inner{\frac{\partial_x^mu_\eps}{ \omega^s+\omega_\eps }},
\end{equation}
 \begin{equation}\label{funhm}
h_{m,\eps}=	\chi_2 \partial_x^m\partial_y \omega_\eps-\chi_2\frac{  \partial_y^2\omega^s+\partial_y^2 \omega_\eps }{\partial_y\omega^s+\partial_y\omega_\eps }\partial_x^m \omega_\eps,
\end{equation}
and
\begin{equation} 
\label{tildegm}
 g_{m,\eps}=
  \partial_x^{m-1}\Big(\inner { \omega^s+\omega_\eps}\partial_x \omega_\eps-\inner { \partial_y\omega^s+\partial_y\omega_\eps}\partial_x u_\eps\Big).
  \end{equation}
\end{definition}

\begin{definition}  
\label{gevspace}	
Let  $X_{\rho,\sigma}$ be given in Definition \ref{defgev}, equipped with the norm $\norm{\cdot}_{\rho,\sigma}$	 defined by \eqref{trinorm}.   Let    $\chi_1,\chi_2$ be given by \eqref{chi1}-\eqref{cutofffu},   and let  $u_\eps$ satisfy the properties  listed in \eqref{condi}.      We will use the notation  $\abs{\cdot}_{\rho,\sigma}$ which is given by   
\begin{eqnarray*}
	\abs{u_\eps}_{\rho,\sigma}&=&\norm{u_\eps}_{\rho,\sigma}+\sup_{1\leq m\leq 5}     \Big(m\norm{  g_{m,\eps}}_{L^2}+\big\|\comi y^{\ell} f_{m,\eps}\big\|_{L^2}+\norm{h_{m,\eps}}_{L^2}+\norm{\chi_2\partial_y\partial_x^m\omega_\eps}_{L^2}\Big)\\
	&&+ \sup_{m\geq 6}   \frac{\rho^{m-5}}{\big[\inner{m-6}!\big]^{\sigma}} \Big(m\norm{  g_{m,\eps}}_{L^2}+\big\|\comi y^{\ell} f_{m,\eps}\big\|_{L^2}+\norm{h_{m,\eps}}_{L^2}+\norm{\chi_2\partial_y\partial_x^m\omega_\eps}_{L^2}\Big).
\end{eqnarray*}
 Similarly we can define $\abs{u_0}_{\rho,\sigma}.$
\end{definition}

\begin{remark}
\begin{enumerate}[(i)]
\item Observe there is an extra factor $m$ in front of  the term $\norm{  g_{m,\eps}}_{L^2}$ in the definition of the norm $\abs{\cdot}_{\rho,\sigma}.$
\item
Direct calculation shows that 
\begin{equation} \label{eqnordif}
\norm{u_\eps}_{\rho,\sigma}\leq \abs{u_\eps}_{\rho,\sigma}\leq C_{\rho,\rho^*}\big( \norm{u_\eps}_{\rho*,\sigma}+\norm{u_\eps}_{\rho*,\sigma}^2\big)
\end{equation}
 for any   $\rho<\rho^*,$ with $C_{\rho, \rho^*}$ being
a constant depending only on the difference $\rho^*-\rho$.  
\end{enumerate}
\end{remark}

\begin{theorem}[uniform estimates in Gevrey space]\label{uniestgev}
Let $3/2\leq\sigma\leq 2.$  Let  $u_\eps\in L^\infty\inner{[0, T];~X_{\rho_0,\sigma}}$ be a solution to \eqref{regpran}    such that  the properties listed in  \eqref{condi} hold. 
 Then  there exists   a constant   $C_*>1,$  independent of $\eps$ and the solution $u_\eps,$    such that    the estimate 
 \begin{equation}\label{weesun}
	\abs{u_\eps (t)}_{\rho,\sigma}^2\leq C_* \abs{u_0}_{\rho, \sigma}^2+C_* \int_{0}^t \inner{\abs{u_\eps(s)}_{\rho,\sigma}^2+\abs{u_\eps(s)}_{\rho,\sigma}^4} \,ds+C_* \int_{0}^t\frac{ \abs{u_\eps(s)}_{\tilde\rho,\sigma}^2}{\tilde \rho-\rho}\,ds 
\end{equation}	 
 holds for any pair $(\rho,\tilde\rho)$ with   $0<\rho<\tilde \rho<\rho_0$, and for any $t\in[0,\tilde T]$, where $[0,\tilde T]$ is the maximal interval of existence for $\abs{u_\eps}_{\tilde \rho,\sigma}<+\infty.$  
\end{theorem}

  The above theorem is the key part of the paper, and  its proof 
follows from the discussion in  Sections \ref{sec3} to \ref{sec5}.


    \section{Proof of Theorem \ref{uniestgev}:   uniform estimate on
 $ g_m$}\label{sec3}

  This section along with Sections \ref{sec4}-\ref{sec5} are devoted to proving Theorem \ref{uniestgev}, the uniform estimates for the approximate solutions $u^\eps.$   
 To simplify the notations, we will remove in the following discussion the subscript $\eps$ in $u_\eps, \omega_\eps $ if no confusion occurs.  Similarly, we write $f_m, h_m$ and $g_m$ for the auxilliary functions $f_{m,\eps}, h_{m,\eps}$ and $g_{m,\eps}$ defined in \eqref{fungeps}-\eqref{tildegm}.     Moreover, we will use the capital letter $C$ to denote some generic  constants, which may vary from line to line that depend  only on the constants $c_j, \delta,\rho_0$ and $\alpha$  in Assumption \ref{maas} as well as on the Sobolev embedding constants,    but are independent of $\eps$ and the order $m$ of derivatives.

 We begin with a uniform estimate on $g_m=g_{m,\eps}$ with $g_{m,\eps}$ defined by \eqref{tildegm}, that is, 
 \begin{equation}\label{gm}
 	g_{m}=
  \partial_x^{m-1}\Big(\inner { \omega^s+\omega}\partial_x \omega-\inner { \partial_y\omega^s+\partial_y\omega}\partial_x u\Big), \quad m\geq 1.
 \end{equation}
 The main result in this section can be stated as follows. 

\begin{proposition}\label{prpenmon}
Let $m\geq 6$   and let $u\in L^\infty\inner{[0, T];~X_{\rho_0,\sigma}}$  be a solution to \eqref{regpran} under the assumptions in Theorem \ref{uniestgev}. Then for any pair $\inner{\rho,\tilde\rho}$ with $0<\rho<\tilde\rho\leq \rho_0$ and  for any small $t\in[0,T],$ we have
\begin{equation*}
     \begin{aligned}
      &m^2\norm{g_m(t)}_{L^2} ^2  
 	 \leq   \frac{C\big[(m-6)!\big]^{2\sigma}}{\rho^{2(m-5)}}\abs{u_0}_{\rho,\sigma}^2+  C      \int_0^t   \inner{\eps \norm{\partial_x^{m +1 } u}_{L^2}^2+    \eps  \norm{\partial_x^{m+1 } \omega}_{L^2}^2}   \,ds \\
 	 & \qquad  +  C   m^2  \int_0^t \norm{  \partial_y f_{m-1}}_{L^2}^2 ds+\frac{C \big[\inner{m-6}!\big]^{2\sigma}}{\rho^{2(m-5)}}  \inner{  
  \int_0^t\inner{\abs{u(s)}_{\rho,\sigma}^2 + \abs{u(s)}_{\rho,\sigma}^4}\,ds    +  
  \int_0^t   \frac{\abs{u(s)}_{\tilde \rho,\sigma}^2}{\tilde \rho-\rho}\,ds}.
  \end{aligned}
\end{equation*}
\end{proposition}
 
 \subsection{Preliminaries}
 
 Before proving   Proposition \ref{prpenmon}, we first list some inequalities used throughout the paper. 
 
    \begin{lemma} [Some inequalities]\label{lemequa}
\begin{enumerate}[(i)]
 \item  Given any non-negative integers $p$ and $q$, we have 
\begin{eqnarray*}
p!q!\leq (p+q)!.
\end{eqnarray*}	
\item We have $\abs{\cdot}_{\rho,\sigma}\leq \abs{\cdot}_{\tilde\rho,\sigma}$ for $\rho\leq\tilde\rho.$
\item For any integer $k\geq 1$ and for any  pair $(\rho,\tilde \rho)$ with $0<\rho<\tilde\rho\leq 1,$ we have 
\begin{equation}
\label{factor}
k\inner{\frac{\rho}{\tilde\rho}}^k\leq 	\frac{1}{\tilde\rho} k\inner{\frac{\rho}{\tilde\rho}}^k\leq\frac{1}{\tilde\rho-\rho}.
\end{equation}
\item  Let $\chi_2$ be given in \eqref{cutofffu} and  let $\sigma\geq 1.$  Then for any $0<r\leq 1$,      we have
\begin{equation}\label{etan}
 \norm{\comi y^{\ell-1}\partial_x^m u}_{L^2}+\norm{\comi y^{\ell}\partial_x^m \omega}_{L^2} \leq 
	\left\{
	\begin{aligned}
	& \frac{   \big[\inner{m-6}!\big]^{ \sigma}}{r^{ (m-5)}}\abs{u}_{r,\sigma},\quad {\rm if}~m\geq 6,\\
	&\abs{u}_{r,\sigma},\quad {\rm if}~0\leq m\leq 5,
	\end{aligned}
	\right.
\end{equation} 
and
\begin{equation}
\label{chi2est} 
 m\norm{  g_{m }}_{L^2}+\norm{\comi y^{\ell} f_{m }}_{L^2}+\norm{h_{m }}_{L^2}+\norm{\chi_2\partial_y\partial_x^m\omega }_{L^2}\leq \left\{
	\begin{aligned}
	& \frac{    \big[\inner{m-6}!\big]^{ \sigma}}{r^{ (m-5)}}\abs{u}_{r,\sigma},~{\rm if}~m\geq 6,\\
	& \abs{u}_{r,\sigma},\quad {\rm if}~1\leq m\leq 5,
	\end{aligned}
	\right.
\end{equation}
and 
\begin{equation}\label{emix}
	\norm{\comi y^{\ell+1}\partial_x^i\partial_y^j \omega}_{L^2}\leq 
	\left\{
	\begin{aligned}
	& \frac{   \big[\inner{i+j-6}!\big]^{ \sigma}}{r^{ (i+j-5)}}\abs{u}_{r,\sigma},\quad {\rm if}~i+j\geq 6 ~{\rm and}~ 1\leq j\leq 4,\\
	& \abs{u}_{r,\sigma},\quad {\rm if}~0\leq i+j\leq 5~{\rm and}~1\leq j\leq 4.
	\end{aligned}
	\right.
\end{equation}
\item Let $\sigma\geq 1$ and let $m\geq 7.$  Then  for any $0<r\leq 1$ we have
\begin{equation}
\label{fm1}
\norm{\partial_x^{m-1} \partial_y \omega}_{L^2} \leq \norm{  \partial_y f_{m-1}}_{L^2} +     \frac{C\big[\inner{m-7}!\big]^{\sigma}}{r^{m-6} } \abs{u}_{r,\sigma}
  \leq \norm{  \partial_y f_{m-1}}_{L^2} +  C m^{-\sigma} \frac{\big[\inner{m-6}!\big]^{\sigma}}{r^{m-5} } \abs{u}_{r,\sigma}.
\end{equation}
 \end{enumerate}

 \end{lemma}
 
 \begin{proof}
 The first  statement $(i)$ is clear.  The second and the fourth statements $(ii)$ and $(iv)$  follow directly from the definition of $\abs{\cdot}_{\rho,\sigma}$ (See Definition \ref{gevspace}).
 As for  $(iii)$,  we have for any $k\geq1$ and    any  pair $(\rho,\tilde \rho)$ with $0<\rho<\tilde\rho\leq 1,$ 
 \begin{eqnarray*}
 \frac{1}{1-\frac{\rho}{\tilde\rho}}=\sum_{j=0}^\infty\inner{\frac{\rho}{\tilde\rho}}^j \geq  k \inner{\frac{\rho}{\tilde\rho}}^k,
 \end{eqnarray*}
 from which the desired inequalities follow. 
 
Now we prove $(v)$. In view of \eqref{suppch2} we see $\chi_1\equiv 1$ on supp\,$1-\chi_2$ so that
   \begin{eqnarray*}
 \norm{\partial_x^{m-1} \partial_y \omega}_{L^2}  &\leq&   \norm{\inner{1-\chi_2} \partial_y \inner{\chi_1\partial_x^{m-1}  \omega}}_{L^2}+  \norm{\chi_2\partial_y \partial_x^{m-1}   \omega}_{L^2}\\ 
   &\leq&  \norm{ \inner{1-\chi_2}  \partial_y f_{m-1}}_{L^2}+   \norm{  \inner{1-\chi_2} \partial_y \big(\chi_1\frac{\partial_y\omega^s+\partial_y\omega}{\omega^s+\omega}\partial_x^{m-1}u\big)}_{L^2}+  \norm{\chi_2\partial_y \partial_x^{m-1}   \omega}_{L^2}\\ 
   &\leq&  \norm{   \partial_y f_{m-1}}_{L^2}+ C  \norm{    \partial_x^{m-1}u}_{L^2} + C\norm{    \partial_x^{m-1}\omega}_{L^2}+  \norm{\chi_2\partial_y \partial_x^{m-1}   \omega}_{L^2}.
   \end{eqnarray*}
In the above,    the second inequality uses \eqref{fungeps}, the definition of $f_m$.  This along with    \eqref{etan} and \eqref{chi2est}  yield  
      \begin{eqnarray*}
  \norm{\partial_x^{m-1} \partial_y \omega}_{L^2} 
  \leq \norm{  \partial_y f_{m-1}}_{L^2} +     \frac{C\big[\inner{m-7}!\big]^{\sigma}}{r^{m-6} } \abs{u}_{r,\sigma}  \leq  \norm{  \partial_y f_{m-1}}_{L^2} +  C m^{-\sigma} \frac{\big[\inner{m-6}!\big]^{\sigma}}{r^{m-5} } \abs{u}_{r,\sigma}. \end{eqnarray*}
     The proof is then completed. 
 \end{proof}

Let $g_m$ be given in  \eqref{gm}, and we define its key
component $\tilde g_m$ by  setting 
  \begin{equation}\label{funofg}
  \tilde g_m = \inner { \omega^s+\omega} \partial_x^m \omega-\inner{   \partial_y \omega^s+ \partial_y \omega } \partial_x^m u,\quad m\geq 1.\end{equation}
The next lemma is concerned with  the difference $ g_m-  \tilde g_m$.
 
 \begin{lemma}\label{lemdiffe}
   Let $m\geq 6 $ and let  $1\leq \sigma\leq 2.$  We have
  	\begin{eqnarray*}
  		 \norm{ g_m-\tilde g_m}_{L^2}\leq  C\norm{  \partial_y f_{m-1}}_{L^2}+  C  m^{1-\sigma} \frac{\big[\inner{m-6}!\big]^{\sigma}}{\tilde \rho^{m-5} } \abs{u}_{\tilde \rho,\sigma}+C m^{2-2\sigma} \frac{\big[\inner{m-6}!\big]^{\sigma}}{\rho^{m-5} } \abs{u}_{\rho,\sigma} ^2.
  	\end{eqnarray*}
 \end{lemma}
 
 \begin{proof}
 First, direct calculation shows 
  \begin{equation}\label{differ}
  	 g_m-\tilde g_m=\sum_{j=1}^{m-1}{{m-1}\choose j} \inner{\partial_x^j \omega}\partial_x^{m-j}\omega-\sum_{j=1}^{m-1}{{m-1}\choose j} \inner{\partial_x^j \partial_y \omega}\partial_x^{m-j}u.
  \end{equation}	
Thus
\begin{eqnarray*}
  \norm{ g_m-\tilde g_m}_{L^2} 
   	  \leq    \sum_{j=1}^{m-1}\frac{(m-1)!}{j!(m-1-j)!} \norm{\inner{ \partial_x^j \omega}\partial_x^{m-j}\omega}_{L^2} +   \sum_{j=1}^{m-1}\frac{(m-1)!}{j!(m-1-j)!} \norm{\inner{\partial_x^j \partial_y \omega}\partial_x^{m-j}u}_{L^2}.
  \end{eqnarray*}
  We first handle the second term on the right side of the
above estimate, and write
  \begin{eqnarray*}
  	 \sum_{j=1}^{m-1}\frac{(m-1)!}{j!(m-1-j)!} \norm{\inner{\partial_x^j \partial_y \omega}\partial_x^{m-j}u}_{L^2}\leq R_1+R_2,
  \end{eqnarray*}
  where 
  \begin{eqnarray*}
  	R_1=\sum_{j=1}^{[m/2]}\frac{(m-1)!}{j!(m-1-j)!}  \norm{\partial_x^j \partial_y \omega}_{L^\infty}\norm{\partial_x^{m-j}u}_{L^2}
  \end{eqnarray*}
  with $[m/2]$ standing for the largest integer less than or equal to $m/2,$ 
  and 
   \begin{eqnarray*}
  	R_2=\sum_{j=[m/2]+1}^{m-1}\frac{(m-1)!}{j!(m-1-j)!}  \norm{\partial_x^j \partial_y \omega}_{L^2}\norm{\partial_x^{m-j}u}_{L^\infty}.
  \end{eqnarray*}
To estimate $R_2$, we use \eqref{etan} and   \eqref{emix}  along with the Sobolev inequality (see Lemma  \ref{sobine} in Appendix),  to compute 
\begin{eqnarray*}
	R_2 
	& \leq& \norm{\partial_x^{m-1} \partial_y \omega}_{L^2}\norm{\partial_x u}_{L^\infty}+ m\norm{\partial_x^{m-2} \partial_y \omega}_{L^2}\norm{\partial_x^2 u}_{L^\infty}+\sum_{j=m-4}^{m-3}\frac{(m-1)!}{j!(m-1-j)!}  \frac{\big[(j-5)!\big]^\sigma}{\rho^{j-4}} \abs{u}_{\rho,\sigma} ^2\\
	&&+\sum_{j=[m/2]+1}^{m-5}\frac{(m-1)!}{j!(m-1-j)!}  \frac{\big[(j-5)!\big]^\sigma}{\rho^{j-4}}\frac{\big[\inner{m-j-5}!\big]^\sigma}{ \rho^{m-j-4}} \abs{u}_{\rho,\sigma} ^2.
	\end{eqnarray*}
Moreover, by \eqref{emix} and the last inequality in \eqref{condi}, we have
\begin{eqnarray*}
	m\norm{\partial_x^{m-2} \partial_y \omega}_{L^2}\norm{\partial_x^2 u}_{L^\infty}\leq C m^{1-\sigma} \frac{\big[\inner{m-6}!\big]^{\sigma}}{\tilde\rho^{m-5} }\abs{u}_{\tilde\rho,\sigma}.
\end{eqnarray*}
Direct computation gives
\begin{eqnarray*}
	\sum_{j=m-4}^{m-3}\frac{(m-1)!}{j!(m-1-j)!}  \frac{\big[(j-5)!\big]^\sigma}{\rho^{j-4}} \abs{u}_{\rho,\sigma} ^2\leq C m^{2-2\sigma} \frac{\big[\inner{m-6}!\big]^{\sigma}}{\rho^{m-5} } \abs{u}_{\rho,\sigma} ^2,
\end{eqnarray*}
and, using  the statement $(i)$ in Lemma \ref{lemequa},
	\begin{eqnarray*}
	&&\sum_{j=[m/2]+1}^{m-5}\frac{(m-1)!}{j!(m-1-j)!}  \frac{\big[(j-5)!\big]^\sigma}{\rho^{j-4}}\frac{\big[\inner{m-j-5}!\big]^\sigma}{ \rho^{m-j-4}} \abs{u}_{\rho,\sigma} ^2 \\
	& \leq &    C \frac{1}{\rho^{m-5}}\abs{u}_{\rho,\sigma} ^2 \sum_{j=[m/2]+1}^{m-5}\frac{(m-1)! \big[(j-5)!\big]^{\sigma-1}\big[\inner{m-j-5}!\big]^{\sigma-1}}{j^5 (m-j)^4}     \\
	& \leq &    C \frac{1}{\rho^{m-5}}\abs{u}_{\rho,\sigma} ^2 \sum_{j=[m/2]+1}^{m-3}\frac{(m-6)! \big[(m-10)!\big]^{\sigma-1}m^5}{j^5 (m-j)^4}     \\
	&\leq&  C m^{4-4\sigma} \frac{\big[\inner{m-6}!\big]^{\sigma}}{\rho^{m-5} } \abs{u}_{\rho,\sigma} ^2.
\end{eqnarray*}
Finally, for any small $\kappa>0$, we use \eqref{fm1} and the last inequality in \eqref{condi} to obtain 
 \begin{eqnarray*}
  \norm{\partial_x^{m-1} \partial_y \omega}_{L^2}\norm{\partial_x u}_{L^\infty}  \leq    C \norm{  \partial_y f_{m-1}}_{L^2}+  C   m^{ -\sigma} \frac{\big[\inner{m-6}!\big]^{\sigma}}{\tilde \rho^{m-5} } \abs{u}_{\tilde \rho,\sigma}.  \end{eqnarray*}
  Combining the inequalities above we conclude 
   \begin{eqnarray*}
   	R_2\leq  C\norm{  \partial_y f_{m-1}}_{L^2}+  C m^{1-\sigma} \frac{\big[\inner{m-6}!\big]^{\sigma}}{\tilde \rho^{m-5} } \abs{u}_{\tilde \rho,\sigma}+C m^{2-2\sigma} \frac{\big[\inner{m-6}!\big]^{\sigma}}{\rho^{m-5} } \abs{u}_{\rho,\sigma} ^2.
   \end{eqnarray*}
   The estimation on $R_1$ is similar as above  with  simpler  so that
we omit it for brevity. Then we have
   \begin{eqnarray*}
   	 R_1\leq C m^{1-\sigma} \frac{\big[\inner{m-6}!\big]^{\sigma}}{\tilde \rho^{m-5} } \abs{u}_{\tilde \rho,\sigma}+C m^{2-2\sigma} \frac{\big[\inner{m-6}!\big]^{\sigma}}{\rho^{m-5} } \abs{u}_{\rho,\sigma} ^2.
   \end{eqnarray*}
  Thus the desired estimate follows and  the proof of Lemma \ref{lemdiffe} is completed. 
 \end{proof}

\subsection{Proof of Proposition \ref{prpenmon}}

The rest of this section is devoted to proving   Proposition \ref{prpenmon}  by  energy method, and the proof  is  inspired by the arguments used in \cite{GM}.  To do so,  we    first write  the equation for $ g_m$ as follows with
its   derivation given in  the Appendix (see Lemma \ref{lemgm} in the Appendix).
\begin{eqnarray*}
	&&\Big(\partial_t    + \inner{u^{s}+u} \partial_x +v\partial_y  -\partial_y^2   -\eps\partial_x^2 \Big)  g_m\\
	&=&-  \sum_{j=1}^{m-1}{{m-1}\choose j} \inner{\partial_x^j u}   g_{m-j+1} -  \sum_{j=1}^{m-1}{{m-1}\choose j} \inner{\partial_x^j v}\partial_y  g_{m-j} \\
		&&+2  \sum_{j=0}^{m-1}{{m-1}\choose j} \inner{\partial_x^j \partial_y^2\omega^s+\partial_x^j \partial_y^2\omega}\partial_x^{m-j} \omega+2  \eps \sum_{j=0}^{m-1}{{m-1}\choose j} \inner{\partial_x^{j+1} \partial_y\omega}\partial_x^{m-j+1} u\\
	&&-2  \sum_{j=0}^{m-1}{{m-1}\choose j} \inner{\partial_x^j \partial_y\omega^s+\partial_x^j \partial_y\omega}\partial_x^{m-j} \partial_y\omega-2  \eps \sum_{j=0}^{m-1}{{m-1}\choose j} \inner{\partial_x^{j+1} \omega}\partial_x^{m-j+1} \omega,
	 \end{eqnarray*}
Moreover, observe that $\partial_y\omega^s|_{y=0}=\partial_y\omega|_{y=0}=0$ and then
\begin{eqnarray*}
	\partial_y g_m|_{y=0}=0.
\end{eqnarray*} 
Thus multiplying both sides by $m^2g_m$ and then taking integration over $\mathbb R_+^2$,  we have
\begin{equation}
\label{entildgm}
 \frac{1}{2}m^2\norm{   g_m(t)}_{L^2}^2+m^2\int_0^t\norm{\partial_y g_m(s)}_{L^2}^2\,ds+\eps m^2 \int_0^t  \norm{ \partial_x   g_m}_{L^2}^2	\,ds  
	\leq  \frac{1}{2}m^2\norm{   g_m(0)}_{L^2}^2+\sum_{1\leq i\leq 6} \mathcal P_i
\end{equation} 
with 
\begin{eqnarray*}
\mathcal P_1  &=&-\int_0^t \bigg( \sum_{j=1}^{m-1}{{m-1}\choose j} \inner{\partial_x^j u} g_{m-j+1} ,~ m^2   g_m\bigg)_{L^2}	\,ds,\\
\mathcal P_2  &=&-\int_0^t \bigg(  \sum_{j=1}^{m-1}{{m-1}\choose j} \inner{\partial_x^j v}\partial_y  g_{m-j}  ,~ m^2   g_m\bigg)_{L^2}	\,ds,\\
 \mathcal P_3  &=&\int_0^t \bigg(2 \sum_{j=0}^{m-1}{{m-1}\choose j} \inner{\partial_x^j \partial_y^2\omega^s+\partial_x^j \partial_y^2\omega}\partial_x^{m-j} \omega,~ m^2 g_m\bigg)_{L^2}	\,ds,\\
\mathcal P_4  &=&-\int_0^t \bigg(2\sum_{j=0}^{m-1}{{m-1}\choose j} \inner{\partial_x^j \partial_y\omega^s+\partial_x^j \partial_y\omega}\partial_x^{m-j} \partial_y\omega,~ m^2  g_m\bigg)_{L^2}	\,ds,\\
\mathcal P_5  &=&\int_0^t \bigg(2 \eps \sum_{j=0}^{m-1}{{m-1}\choose j} \inner{\partial_x^{j+1} \partial_y\omega}\partial_x^{m-j+1} u,~ m^2   g_m\bigg)_{L^2}	\,ds,
\\
\mathcal P_6&=&-\int_0^t \bigg(2 \eps \sum_{j=0}^{m-1}{{m-1}\choose j} \inner{\partial_x^{j+1} \omega}\partial_x^{m-j+1} \omega,~ m^2 g_m\bigg)_{L^2}	\,ds.
\end{eqnarray*}
In the following lemmas, we will  estimate  $\mathcal P_i, 1\leq i\leq 6$  
respectively.

\begin{lemma}[Estimate on $\mathcal P_3$]\label{lemp3}
	 	Let $3/2\leq \sigma\leq 2.$ Then for any small $\kappa>0$ and  for any pair $(\rho,\tilde\rho)$ with $0<\rho<\tilde\rho,$  we have
	 	\begin{eqnarray*}
	 		\mathcal P_3&  \leq&   \kappa m^2 \int_0^t \norm{\partial_y g_m}_{L^2}^2 ds+ C\kappa^{-1} m^2 \int_0^t \norm{  \partial_y f_{m-1}}_{L^2}^2 ds    \\
		&&+\frac{C \kappa^{-1}\big[\inner{m-6}!\big]^{2\sigma}}{\rho^{2(m-5)}}    
 \inner{ \int_0^t\inner{\abs{u(s)}_{\rho,\sigma}^2 + \abs{u(s)}_{\rho,\sigma}^4}\,ds  +\int_0^t   \frac{\abs{u(s)}_{\tilde \rho,\sigma}^2}{\tilde \rho(s)-\rho}\,ds}.
	 	\end{eqnarray*}
	 \end{lemma}

	 \begin{proof}
	We write  
	\begin{eqnarray*}
		\mathcal P_3= 2\int_0^t \bigg( \sum_{j=0}^{m-1}{{m-1}\choose j} \inner{\partial_x^j \partial_y^2\omega^s+\partial_x^j \partial_y^2\omega}\partial_x^{m-j} \omega,~  m^2g_m\bigg)_{L^2}\,dt\leq J_1+J_2+J_3+J_4
		\end{eqnarray*}
		with
		\begin{eqnarray*}
		J_1&=& 2m  \int_0^t \norm{\partial_y^2\omega^s+\partial_y^2\omega}_{L ^\infty}\norm{\partial_x^{m} \omega}_{L^2}\inner{m \norm{    g_m}_{L^2}}\,dt,\\
		J_2&=& 2m(m-1)  \int_0^t \norm{ \partial_x\partial_y^2\omega}_{L ^\infty}\norm{\partial_x^{m-1} \omega}_{L^2}\inner{m \norm{    g_m}_{L^2}}\,dt,\\
		J_3&=& 2m \sum_{j=2}^{[m/2]}\frac{(m-1)!}{j!(m-1-j)!}\int_0^t   \norm{\partial_x^j \partial_y^2\omega}_{L^\infty
		}\norm{\partial_x^{m-j} \omega}_{L^2}\inner{m \norm{    g_m}_{L^2}}\,dt,\\
		J_4&=& 2\sum_{j=[m/2]+1}^{m-1}{{m-1}\choose j}\int_0^t \bigg(  \inner{ \partial_x^j \partial_y^2\omega}\partial_x^{m-j} \omega,~  m^2g_m\bigg)_{L^2}.
	\end{eqnarray*}
	
	\underline{\it Estimate on $J_1$ and $J_2$:} 
	 For $J_1$, we use the third estimate in \eqref{condi} as well as  \eqref{chi2est}  to obtain 
	\begin{eqnarray*}
		 J_1\leq     \frac{C \big[\inner{m-6}!\big]^{2\sigma}}{\rho^{2(m-5)}} \int_0^t    \abs{u(s)}_{  \tilde\rho,\sigma} ^2 \frac{m\rho^{2(m-5)}}{\tilde \rho^{2(m-5)}}  \,ds \leq    \frac{C\big[\inner{m-6}!\big]^{2\sigma}}{\rho^{2(m-5)}}    
\int_0^t   \frac{\abs{u(s)}_{\tilde \rho,\sigma}^2}{\tilde \rho(s)-\rho}\,ds,
	\end{eqnarray*}
where	the last inequality follows from  \eqref{factor}  in Lemma \ref{lemequa}. Similarly,
	\begin{eqnarray*}
		 J_2\leq     \frac{C \big[\inner{m-6}!\big]^{2\sigma}}{\rho^{2(m-5)}} \int_0^t    \abs{u(s)}_{  \tilde\rho,\sigma}^2  \frac{m^{2-\sigma}\rho^{2(m-5)}}{\tilde \rho^{2(m-5)}}  \,ds \leq    \frac{C\big[\inner{m-6}!\big]^{2\sigma}}{\rho^{2(m-5)}}    
\int_0^t   \frac{\abs{u(s)}_{\tilde \rho,\sigma}^2}{\tilde \rho(s)-\rho}\,ds,
	\end{eqnarray*}  
where	the last inequality follows from the fact that $\sigma\geq 3/2\geq 1.$

	\underline{\it Estimate on $J_3$:}   
  By using the statement  $ (iv)$ in  Lemma   \ref{lemequa} as well as the Sobolev inequality (see Lemma \ref{sobine} in the Appendix),  we have
\begin{eqnarray*}
	J_{3}&\leq & C  m \sum_{j=2}^{[m/2]}\frac{(m-1)!}{j!(m-1-j)!} \frac{\big[(j-2)!\big]^\sigma}{\rho^{j-1}}\frac{\big[\inner{m-j-6}!\big]^\sigma}{\rho^{m-j-5}} 
\frac{\big[\inner{m-6}!\big]^\sigma}{\rho^{m-5}}  \int_0^t \abs{u(s)}_{\rho,\sigma}^3  \,ds\\
	&\leq & C m \frac{\big[\inner{m-6}!\big]^{\sigma+1}}{\rho^{2(m-5)}}   \sum_{j=2}^{[m/2]}\frac{m^5\big[(j-2)!\big]^{\sigma-1} \big[\inner{m-j-6}!\big]^{\sigma-1}}{j^2(m-j )^5}     
\int_0^t \abs{u(s)}_{\rho,\sigma}^3  \,ds\\
	&\leq & C m \big[(m-8)!\big]^{\sigma-1}    \frac{\big[\inner{m-6}!\big]^{\sigma+1}}{\rho^{2(m-5)}}   \int_0^t \abs{u(s)}_{\rho,\sigma}^3  \,ds \sum_{j=2}^{[m/2]}\frac{ 1}{j^2 }     
\qquad\qquad ( \textrm {using  (i) in Lemma }~\ref{lemequa})\\
	&\leq &  \frac{C\big[\inner{m-6}!\big]^{2\sigma}}{\rho^{2(m-5)}}    
\int_0^t \abs{u(s)}_{\rho,\sigma}^3  \,ds,  
\end{eqnarray*}
where in the last inequality, we have used $ m^{-2(\sigma-1)}\leq m^{-1}$ because $\sigma\geq3/2.$ 

	\underline{\it Estimate on $J_4$:} By integration by parts, for any small $\kappa>0,$ we have
	\begin{eqnarray*}
		J_4&=&-2\sum_{j=[m/2]+1}^{m-1}{{m-1}\choose j}\int_0^t \bigg(  \inner{ \partial_x^j \partial_y\omega}\partial_x^{m-j} \omega,~  m^2\partial_y g_m\bigg)_{L^2}\\
		&&-2\sum_{j=[m/2]+1}^{m-1}{{m-1}\choose j}\int_0^t \bigg(  \inner{ \partial_x^j \partial_y\omega}\partial_x^{m-j} \partial_y\omega,~  m^2 g_m\bigg)_{L^2}\\
		&\leq & \kappa m^2\int_0^t \norm{\partial_y g_m}_{L^2}^2\,ds+    C\kappa^{-1} m^2 
\int_0^t   \norm{ \partial_x^{m-1}\partial_y\omega }_{L^2}^2 \norm{\partial_x \omega}_{L^\infty}^2\,ds\\
		&&+  C\kappa^{-1} m^2 
\int_0^t \bigg[\sum_{j=[m/2]+1}^{m-2}{{m-1}\choose j} \norm{ \partial_x^j \partial_y\omega }_{L^2} \norm{\partial_x^{m-j} \omega}_{L^\infty}\bigg]^2\,ds\\
		&&+2 m\sum_{j=[m/2]+1}^{m-1}{{m-1}\choose j}\int_0^t \norm{ \partial_x^j \partial_y\omega }_{L^2} \norm{\partial_x^{m-j}\partial_y\omega}_{L^\infty}\inner{m \norm{ g_m}_{L^2}}\,ds\\
		&\stackrel{\rm def}{=} & \kappa m^2\int_0^t \norm{\partial_y g_m}_{L^2}^2\,ds+  J_{4,1}+J_{4,2}+J_{4,3}.
	\end{eqnarray*}
Now we use the statements  $(iii)$  and $(iv)$ in Lemma \ref{lemequa} to get by  repeating the arguments  used for the terms $J_1$-$J_3$,  
\begin{equation*}
 J_{4,3}	\leq    \frac{C\big[\inner{m-6}!\big]^{2\sigma}}{\rho^{2(m-5)}}    
\inner{\int_0^t \abs{u(s)}_{\rho,\sigma}^3  \,ds+\int_0^t   \frac{\abs{u(s)}_{\tilde \rho,\sigma}^2}{\tilde \rho(s)-\rho}\,ds},
\end{equation*}
and
\begin{eqnarray*}
J_{4,2} 	 
&\leq&    \frac{C\kappa^{-1}\big[\inner{m-6}!\big]^{2\sigma}}{\rho^{2(m-5)}}    
\inner{\int_0^t \abs{u(s)}_{\rho,\sigma}^4  \,ds+\int_0^t   \frac{\abs{u(s)}_{\tilde \rho,\sigma}^2}{\tilde \rho(s)-\rho}\,ds}.
\end{eqnarray*}
It remains to treat $J_{4,1}$. To do so, we use \eqref{fm1} and the last inequality in \eqref{condi} to obtain by using $\sigma\geq3/2> 1,$
\begin{eqnarray*}
J_{4,1} 
         \leq
 C \kappa^{-1} m^2 \int_0^t \norm{  \partial_y f_{m-1}}_{L^2}^2 ds 
    +    \frac{C \kappa^{-1} \big[\inner{m-6}!\big]^{2\sigma}}{  \rho^{2(m-5)} } \int_0^t    \abs{u}_{ \rho,\sigma}^2   \,ds.
\end{eqnarray*}
This along with the estimates on $J_{4,2}$ and $J_{4,3}$ given
above imply that for any small $\kappa>0,$ 
\begin{eqnarray*}
		J_{4}&\leq &   \kappa m^2 \int_0^t \norm{\partial_y g_m}_{L^2}^2 ds+ C\kappa^{-1} m^2 \int_0^t \norm{  \partial_y f_{m-1}}_{L^2}^2 ds    \\
		&&+\frac{C \kappa^{-1}\big[\inner{m-6}!\big]^{2\sigma}}{\rho^{2(m-5)}}    
 \inner{ \int_0^t\inner{\abs{u(s)}_{\rho,\sigma}^2 + \abs{u(s)}_{\rho,\sigma}^4}\,ds  +\int_0^t   \frac{\abs{u(s)}_{\tilde \rho,\sigma}^2}{\tilde \rho(s)-\rho}\,ds}.
\end{eqnarray*}
Then  combining the estimates on the terms $J_1$-$J_4$, the desired estimate follows.  Thus the proof of the lemma is completed. 
 \end{proof}
 
 \begin{lemma}[Estimate on $\mathcal P_4$]\label{lemp5}
Let $3/2\leq \sigma\leq 2$. Then for any small $\kappa>0,$ and  for any pair $(\rho,\tilde\rho)$ with $0<\rho<\tilde\rho,$  we have
\begin{eqnarray*}
	\mathcal P_4&\leq&\kappa m^2\int_0^t   \norm{\partial_y g_m}_{L^2} ^2\,ds +C \kappa^{-1}m^2\int_0^t   \norm{\partial_y   f_{m-1} }_{L^2} ^2\,ds\\
	&&+   \frac{C\kappa^{-1}\big[ \inner{m-6}!\big]^{2\sigma} }{\rho^{2(m-5)}} \inner{\int_0^t\inner{\abs{u(s)}_{\rho,\sigma}^2 + \abs{u(s)}_{\rho,\sigma}^4}\,ds +\int_0^t  \frac{\abs{u(s)}_{\tilde\rho,\sigma}^2  }{\tilde \rho-\rho}\,ds}.
\end{eqnarray*}
\end{lemma}

\begin{proof}
Let $\chi_2$ be the function given in \eqref{cutofffu}.  We can 
decompose $\mathcal P_4$ by
\begin{eqnarray*}
	\mathcal P_4  &=&-\int_0^t \bigg(2  \sum_{j=0}^{m-1}{{m-1}\choose j} \inner{\partial_x^j \partial_y\omega^s+\partial_x^j \partial_y\omega}\partial_x^{m-j} \partial_y\omega,~ m^2 g_m\bigg)_{L^2}	\,ds,\\
	 &=&-2\sum_{j=0}^{m-1}{{m-1}\choose j} \int_0^t \bigg(    \chi_2 \inner{\partial_x^j \partial_y\omega^s+\partial_x^j \partial_y\omega}\partial_x^{m-j} \partial_y\omega,~  m^2    g_m\bigg)_{L^2}	\,ds\\
	 	 & &-2 \sum_{j=0}^{m-1}{{m-1}\choose j} \int_0^t \bigg(   \inner{1-\chi_2}\inner{\partial_x^j \partial_y\omega^s+\partial_x^j \partial_y\omega}\partial_x^{m-j} \partial_y\omega,~m^2   g_m\bigg)_{L^2}	\,ds\\
	&\stackrel{\rm def}{=} &S_1+S_2,
\end{eqnarray*}
\underline{\it Estimate on $S_1$:}     
Note that $S_1\leq S_{1,1}+S_{1,2}$ with
\begin{eqnarray*}
	S_{1,1}&=& 2m \sum_{j=0}^{[m/2]}\frac{ (m-1)!}{j!(m-1 -j)!}  \int_0^t \norm{ \partial_x^j \partial_y\omega^s+\partial_x^j \partial_y\omega}_{L ^\infty} \norm{\chi_2\partial_x^{m-j} \partial_y\omega} _{L^2}\Big(m\norm{    g_m}_{L^2}\Big)	\,ds,\\
	S_{1,2}&=& 2m\sum_{[m/2]+1}^{m-1}\frac{ (m-1)!}{j!(m-1 -j)!}  \int_0^t \norm{\chi_2\partial_x^j \partial_y\omega}_{L ^2} \norm{\partial_x^{m-j} \partial_y\omega} _{L^\infty}\Big(m\norm{  g_m}_{L^2}\Big)	\,ds.
\end{eqnarray*}
Moreover, we   use \eqref{chi2est}  and \eqref{emix}  in  Lemma \ref{lemequa} and the Sobolev inequality (see Lemma \ref{sobine} in the Appendix), by following the arguments used for the terms $J_1$-$J_3$ in Lemma \ref{lemp3}, to obtain 
\begin{eqnarray*}
 S_{1,1}  & \leq&  Cm\sum_{j=3}^{[m/2]}\frac{ (m-1)!}{j!(m-1 -j)!}   \frac{\big[(j-3)!\big]^\sigma}{\rho^{j-2}}\frac{\big[\inner{m-j-6}!\big]^\sigma}{\rho^{m-j-5}} \frac{\big[\inner{m-6}!\big]^\sigma}{ \rho^{m-5}}  
 \int_0^t  \abs{u(s)}_{\rho,\sigma}^3 \,ds \\
&+&  Cm\sum_{0\leq j\leq 2}\frac{ (m-1)!}{j!(m-1 -j)!}  
  \int_0^t  \frac{\big[\inner{m-j-6}!\big]^\sigma}{\tilde\rho^{m-j-5}}\frac{\big[\inner{m-6}!\big]^\sigma}{\tilde\rho^{ m-5 }}\norm{  \partial_x^j\partial_y\omega^s+  \partial_x^j\partial_y\omega}_{L ^\infty}  \abs{u(s)}_{\tilde\rho,\sigma}^2ds\\
&\leq&  \frac{C\big[\inner{m-6}!\big]^{2\sigma}}{\rho^{2(m-5)}}   \inner{ 
  \int_0^t  \abs{u(s)}_{ \rho,\sigma}^3 \,ds+    
  \int_0^t  \frac{ \abs{u(s)}_{\tilde\rho,\sigma}^2}{ \tilde\rho-\rho } \,ds}.
  \end{eqnarray*}
Similar estimate holds for  $S_{1,2}$. Thus, we conclude that
\begin{equation}
\label{esons1}
S_1	\leq \frac{C\big[\inner{m-6}!\big]^{2\sigma}}{\rho^{2(m-5)}}    \inner{ \int_0^t   \abs{u(s)}_{ \rho,\sigma}^4 \,ds+
  \int_0^t  \frac{ \abs{u(s)}_{\tilde\rho,\sigma}^3}{\tilde\rho-\rho} \,ds}.
\end{equation}
\underline{\it Estimate on $S_2$:}  Write $S_{2}=S_{2,1}+S_{2,2}$ with
\begin{eqnarray*}
		S_{2,1}&= &-2m^2\sum_{j=0}^{[m/2]}{{m-1}\choose j} \int_0^t \bigg(  \inner{1-\chi_2}\inner{\partial_x^j \partial_y\omega^s+\partial_x^j \partial_y\omega}\partial_x^{m-j} \partial_y\omega,~  g_m\bigg)_{L^2}	\,ds,\\
			S_{2,2}&= &-2m^2\sum_{j=[m/2]+1}^{m-1}{{m-1}\choose j} \int_0^t \bigg(  \inner{1-\chi_2}\inner{\partial_x^j \partial_y\omega^s+\partial_x^j \partial_y\omega}\partial_x^{m-j} \partial_y\omega,~  g_m\bigg)_{L^2}	\,ds.
\end{eqnarray*}
 Following the arguments for $J_1$-$J_3$ in Lemma \ref{lemp3}, we have 
 \begin{equation}
 \label{s21+}	
 	S_{2,2}\leq   \frac{C \big[ \inner{m-6}!\big]^{2\sigma} }{\rho^{2(m-5)}} \inner{\int_0^t \abs{u(s)}_{\rho,\sigma}^3\,ds+\int_0^t  \frac{\abs{u(s)}_{\tilde\rho,\sigma}^2  }{\tilde \rho-\rho}\,ds}.
 \end{equation}
As for $S_{2,1}$, integration by parts yields
\begin{equation}\label{upfos222}
\begin{aligned}
	S_{2,1}
	=& 2m^2 \int_0^t \bigg(  \inner{1-\chi_2}\inner{ \partial_y\omega^s+\partial_y\omega}\partial_x^{m }  \omega,~ \partial_y  g_m\bigg)_{L^2}	\,ds\\
	&+2m^2\sum_{j=1}^{[m/2]}{{m-1}\choose j} \int_0^t \bigg(  \inner{1-\chi_2}\inner{\partial_x^j \partial_y\omega^s+\partial_x^j \partial_y\omega}\partial_x^{m-j} \omega,~   \partial_y g_m\bigg)_{L^2}	\,ds\\
	&+2m^2\sum_{j=0}^{[m/2]}{{m-1}\choose j} \int_0^t \bigg( \Big[\partial_y\Big( \inner{1-\chi_2}\inner{\partial_x^j \partial_y\omega^s+\partial_x^j \partial_y\omega}\Big)\Big]\partial_x^{m-j }  \omega,~   g_m\bigg)_{L^2}	\,ds,
	\end{aligned}
	\end{equation}
	where the last two terms on the right side of \eqref{upfos222} are 
bounded above by 
	\begin{eqnarray*}
		 \kappa m^2 \int_0^t \norm{\partial_y g_m}_{L^2}^2 ds+\frac{C \kappa^{-1}\big[\inner{m-6}!\big]^{2\sigma}}{\rho^{2(m-5)}}    
 \inner{ \int_0^t\inner{\abs{u(s)}_{\rho,\sigma}^3 + \abs{u(s)}_{\rho,\sigma}^4}\,ds  +\int_0^t   \frac{\abs{u(s)}_{\tilde \rho,\sigma}^2}{\tilde \rho -\rho}\,ds},
	\end{eqnarray*}
	for any small $\kappa>0.$ This can derived from a similar calculation as in Lemma \ref{lemp3}, observing $3/2\leq \sigma\leq 2.$  It remains to treat the first term on the right side of \eqref{upfos222}, for this,  we claim that
	\begin{equation}
	\label{cla}
	\begin{aligned}
	&2m^2 \int_0^t \bigg(  \inner{1-\chi_2}\inner{ \partial_y\omega^s+\partial_y\omega}\partial_x^{m }  \omega,~ \partial_y  g_m\bigg)_{L^2}	\,ds\\
	\leq&  
	\kappa m^2 \int_0^t \norm{\partial_y g_m}_{L^2}^2 ds+C\kappa^{-1} m^2\int_0^t   \norm{\partial_y   f_{m-1} }_{L^2} ^2\,ds\\
	&+\frac{C \kappa^{-1}\big[\inner{m-6}!\big]^{2\sigma}}{\rho^{2(m-5)}}    
 \inner{ \int_0^t\inner{\abs{u(s)}_{\rho,\sigma}^2 + \abs{u(s)}_{\rho,\sigma}^4}\,ds  +\int_0^t   \frac{\abs{u(s)}_{\tilde \rho,\sigma}^2}{\tilde \rho -\rho}\,ds}.
 \end{aligned}
	\end{equation}
 To confirm this, we use   the fact that $\abs{\omega^s+\omega}\geq c_1/4$ on supp\,$(1-\chi_2)$ 
 to   write,  in view of \eqref{funofg},
\begin{eqnarray*}
	\partial_x^m \omega= \frac{\tilde g_m}{\omega^s+\omega } + \frac{\partial_y\omega^s+\partial_y\omega}{\omega^s+\omega} \partial_x^mu\quad  \textrm{for} ~y\in{\rm supp}\,(1-\chi_2).
\end{eqnarray*} 
 As a result,  for any $\kappa>0,$ we use \eqref{condi} to have
\begin{eqnarray*}
&&2m^2\int_0^t\Big(  \inner{1-\chi_2} \inner{\partial_y\omega^s+ \partial_y\omega}\partial_x^{m}  \omega,~    \partial_y g_m\Big)_{L^2} dt\\
&=& 2m^2\int_0^t\Big(  \inner{1-\chi_2} \inner{\partial_y\omega^s+ \partial_y\omega} \frac{\tilde g_m}{  \omega^s+\omega } ,~  \partial_y g_m\Big)_{L^2} dt \\
&&+ 2m^2\int_0^t\Big(  \inner{1-\chi_2} \inner{\partial_y\omega^s+ \partial_y\omega}  \frac{\partial_y\omega^s+\partial_y\omega}{\omega^s+\omega} \partial_x^mu ,~  \partial_y g_m\Big)_{L^2} dt\\
\\
&\leq & \kappa m^2\int_0^t   \norm{\partial_y   g_m }_{L^2} ^2\,ds  +C\kappa^{-1}m^2 \int_0^t   \norm{ \inner{ g_m- \tilde g_m}}_{L^2} ^2\,ds+C\kappa^{-1}m^2 \int_0^t   \norm{g_m}_{L^2} ^2\,ds\\
&&-2m^2\int_0^t\Big(\inner{\partial_x^mu}\partial_y\Big[  \inner{1-\chi_2} \inner{\partial_y\omega^s+ \partial_y\omega}  \frac{\partial_y\omega^s+\partial_y\omega}{\omega^s+\omega} \Big] ,~     g_m \Big)_{L^2} dt
\\&&-2m^2\int_0^t\Big(  \inner{1-\chi_2} \inner{\partial_y\omega^s+ \partial_y\omega}  \frac{\partial_y\omega^s+\partial_y\omega}{\omega^s+\omega} \partial_x^m\omega ,~    g_m\Big)_{L^2} dt,
\end{eqnarray*}
where the  the last three terms on the right of the above inequality are
 bounded  above by 
\begin{eqnarray*}
	  \frac{C \kappa^{-1}\big[\inner{m-6}!\big]^{2\sigma}}{\rho^{2(m-5)}}    
 \inner{ \int_0^t \abs{u(s)}_{\rho,\sigma}^2  \,ds  +\int_0^t   \frac{\abs{u(s)}_{\tilde \rho,\sigma}^2}{\tilde \rho -\rho}\,ds}
\end{eqnarray*} 
by using the statements $(iii)$ and $(iv)$ in Lemma \ref{lemequa}.    Furthermore, by Lemma \ref{lemdiffe} we have
\begin{equation*}
\begin{aligned}
	&m^2 \int_0^t   \norm{ \inner{ g_m- \tilde g_m}}_{L^2} ^2  \leq   C m^2\int_0^t   \norm{\partial_y   f_{m-1} }_{L^2} ^2\,ds+   \frac{C\big[ \inner{m-6}!\big]^{2\sigma} }{\rho^{2(m-5)}} \int_0^t \abs{u(s)}_{\tilde\rho,\sigma}^2  \frac{m^{2+2-2\sigma}\rho^{2(m-5)}}{\tilde \rho^{2(m-5)}}\,ds\\
	&\qquad\qquad\qquad\qquad\qquad\qquad\qquad+    \frac{C m^{2+4-4\sigma}\big[ \inner{m-6}!\big]^{2\sigma} }{\rho^{2(m-5)}} \int_0^t \abs{u(s)}_{\rho,\sigma}^4\,ds\\
	 &\leq  C m^2\int_0^t   \norm{\partial_y   f_{m-1} }_{L^2} ^2\,ds+   \frac{C\big[ \inner{m-6}!\big]^{2\sigma} }{\rho^{2(m-5)}} \inner{\int_0^t \abs{u(s)}_{\rho,\sigma}^4\,ds+\int_0^t  \frac{\abs{u(s)}_{\tilde\rho,\sigma}^2  }{\tilde \rho-\rho}\,ds}, 
	\end{aligned}    
\end{equation*}
where for the last inequality we  again use 
the fact that $\sigma\geq 3/2$ and \eqref{factor}.   Combining these 
estimates  gives \eqref{cla}. Consequently,  in view of \eqref{upfos222} we  conclude 
\begin{eqnarray*}
	S_{2,1}&\leq& \kappa m^2\int_0^t   \norm{\partial_y g_m}_{L^2} ^2\,ds +C \kappa^{-1}m^2\int_0^t   \norm{\partial_y   f_{m-1} }_{L^2} ^2\,ds\\
	&&+   \frac{C\kappa^{-1}\big[ \inner{m-6}!\big]^{2\sigma} }{\rho^{2(m-5)}} \inner{\int_0^t\inner{\abs{u(s)}_{\rho,\sigma}^2 + \abs{u(s)}_{\rho,\sigma}^4}\,ds +\int_0^t  \frac{\abs{u(s)}_{\tilde\rho,\sigma}^2  }{\tilde \rho-\rho}\,ds},
\end{eqnarray*}
which along with \eqref{s21+} yields 
\begin{eqnarray*}
	S_2&\leq&\kappa m^2\int_0^t   \norm{\partial_y g_m}_{L^2} ^2\,ds +C \kappa^{-1}m^2\int_0^t   \norm{\partial_y   f_{m-1} }_{L^2} ^2\,ds\\
	&&+   \frac{C\kappa^{-1}\big[ \inner{m-6}!\big]^{2\sigma} }{\rho^{2(m-5)}} \inner{\int_0^t\inner{\abs{u(s)}_{\rho,\sigma}^2 + \abs{u(s)}_{\rho,\sigma}^4}\,ds +\int_0^t  \frac{\abs{u(s)}_{\tilde\rho,\sigma}^2  }{\tilde \rho-\rho}\,ds}.
\end{eqnarray*} 
Combining the estimates on $S_1$ and $S_2$ yields the desired estimate,
and this  completes the proof of the lemma.   
	 \end{proof}

 	 \begin{lemma}[Estimate on $\mathcal P_1$ and $\mathcal P_2$]
	 	Let $3/2\leq \sigma\leq 2.$ Then for any small  $\kappa>0,$ and  for any pair $(\rho,\tilde\rho)$ with $0<\rho<\tilde\rho,$  we have
	 	\begin{eqnarray*}
	 		\mathcal P_1+ \mathcal P_2\leq       \kappa   m^2 \int_0^t  \norm{ \partial_y  g_m}_{L^2}^2	ds     + \frac{C\kappa^{-1}\big[ \inner{m-6}!\big]^{2\sigma} }{\rho^{2(m-5)}} \inner{\int_0^t\inner{ \abs{u(s)}_{\rho,\sigma}^3+\abs{u(s)}_{\rho,\sigma}^4  }ds +\int_0^t  \frac{\abs{u(s)}_{\tilde\rho,\sigma}^2  }{\tilde \rho-\rho}ds}.
	 	\end{eqnarray*}
	 \end{lemma}
	 
	 \begin{proof}
	 Following the argument in Lemma \ref{lemp3}, we can obtain that
	 \begin{eqnarray*}
	 	\mathcal P_1  &=&-\int_0^t \bigg( \sum_{j=1}^{m-1}{{m-1}\choose j} \inner{\partial_x^j u} g_{m-j+1} ,~ m^2   g_m\bigg)_{L^2}	\,ds\\
	 	&\leq & \frac{C\big[ \inner{m-6}!\big]^{2\sigma} }{\rho^{2(m-5)}} \inner{\int_0^t \abs{u(s)}_{\rho,\sigma}^3  \,ds +\int_0^t  \frac{\abs{u(s)}_{\tilde\rho,\sigma}^2  }{\tilde \rho-\rho}\,ds}.
	 \end{eqnarray*}
	It remains to treat $\mathcal P_2$.  Firstly,
 integration by parts gives,    	 
	\begin{eqnarray*}
	 		\mathcal P_2  &=&-\int_0^t \bigg( \sum_{j=1}^{m-1}{{m-1}\choose j} \inner{\partial_x^j v}\partial_y g_{m-j}  ,~ m^2    g_m\bigg)_{L^2}	\,ds\\
	 		&=&-\int_0^t \bigg(   \inner{\partial_x^{m-1} v}\partial_y g_{1}  ,~ m^2  g_m\bigg)_{L^2}	\,ds + m^2\int_0^t \bigg( \sum_{j=1}^{m-2}{{m-1}\choose j} \inner{\partial_x^j v} g_{m-j}  ,~  \partial_y     g_m \bigg)_{L^2}	\,ds \\
	 		&&-m^2\int_0^t \bigg( \sum_{j=1}^{m-2}{{m-1}\choose j} \inner{\partial_x^{j+1} u} g_{m-j}  ,~   g_m\bigg)_{L^2}	\,ds.\end{eqnarray*}
	 Moreover, by observing $m^{4-2\sigma}\leq m$ due to $\sigma\geq 3/2$,  for any small  $\kappa>0$ 	 we have 	
	 		\begin{eqnarray*}
	 		&&m^2\int_0^t \bigg( \sum_{j=1}^{m-2}{{m-1}\choose j} \inner{\partial_x^j v} g_{m-j}  ,~  \partial_y     g_m \bigg)_{L^2}	\,ds\\
	 		&\leq & \kappa   m^2 \int_0^t  \norm{\partial_y   g_m }_{L^2}^2	\,ds+ \kappa^{-1} m^2 \int_0^t \bigg[ \sum_{j=1}^{m-2}\frac{\inner{m-1}!}{ j!(m-j)!}  \norm{(\partial_x^{j} v)  g_{m-j} }_{L^2}  \bigg]^2	\,ds\\
	 		&\leq & \kappa   m^2 \int_0^t  \norm{\partial_y   g_m }_{L^2}^2	\,ds+ \frac{C\kappa^{-1}\big[ \inner{m-6}!\big]^{2\sigma} }{\rho^{2(m-5)}} \inner{\int_0^t \abs{u(s)}_{\rho,\sigma}^4  \,ds +\int_0^t  \frac{\abs{u(s)}_{\tilde\rho,\sigma}^2  }{\tilde \rho-\rho}\,ds}, 	\end{eqnarray*}
where for 	the last inequality we have
 again  used the argument   for Lemma \ref{lemp3}.   	
Similarly,  we have
		\begin{eqnarray*}
	 & &-\int_0^t \bigg(   \inner{\partial_x^{m-1} v}\partial_y g_{1}  ,~ m^2  g_m\bigg)_{L^2}	\,ds-m^2\int_0^t \bigg( \sum_{j=1}^{m-2}{{m-1}\choose j} \inner{\partial_x^{j+1} u} g_{m-j}  ,~   g_m\bigg)_{L^2}	\,ds\\
	 		&\leq &		\frac{C  \big[\inner{m-6}!\big]^{2\sigma}}{\rho^{2(m-5)}}     \int_0^t \inner{\abs{u(s)}_{\rho,\sigma}^3  \,ds+   \int_0^t   \frac{\abs{u(s)}_{  \tilde\rho,\sigma} ^2}{\tilde\rho-\rho} \,ds}.
	 	\end{eqnarray*}
Thus we obtain  
	 	\begin{eqnarray*}
	 		\mathcal P_2  \leq    \kappa   m^2 \int_0^t  \norm{ \partial_y  g_m}_{L^2}^2	\,ds    + \frac{C\kappa^{-1}\big[ \inner{m-6}!\big]^{2\sigma} }{\rho^{2(m-5)}} \inner{\int_0^t\inner{ \abs{u(s)}_{\rho,\sigma}^3+\abs{u(s)}_{\rho,\sigma}^4  }\,ds +\int_0^t  \frac{\abs{u(s)}_{\tilde\rho,\sigma}^2  }{\tilde \rho-\rho}\,ds}.
	 	\end{eqnarray*}	  
This along with the upper bound for $\mathcal P_1$ completes the proof of the lemma.
	 \end{proof}

 \begin{lemma}[Estimate on $\mathcal P_5$ and $\mathcal P_6$]\label{lemp6p5}
	 	Let $3/2\leq \sigma\leq 2.$ Then for any pair $(\rho,\tilde\rho)$ with $0<\rho<\tilde\rho$ and for any small $\kappa>0$,  we have
	 	\begin{eqnarray*}
	 	 \mathcal P_5+ \mathcal P_6 &\leq&    \kappa \eps m^2 \int_0^t \inner{ \norm{ \partial_x   g_m}_{L^2}^2+\norm{ \partial_y   g_m}_{L^2}^2}	\,ds+ \kappa^{-1} \eps    \int_0^t      \inner{\norm{\partial_x^{m+1} u}_{L^2} ^2+\norm{\partial_x^{m+1} \omega}_{L^2} ^2}\,ds\\
	 	 &&+  \frac{C\kappa^{-1}  \big[\inner{m-6}!\big]^{2\sigma}}{\rho^{2(m-5)}}    \inner{ \int_0^t \inner{\abs{u(s)}_{\rho,\sigma}^2+\abs{u(s)}_{  \rho,\sigma}^4 }  \,ds+   \int_0^t   \frac{\abs{u(s)}_{  \tilde\rho,\sigma} ^2}{\tilde\rho-\rho} \,ds}.
	 	\end{eqnarray*}
	  
	 \end{lemma}
	 
	 \begin{proof}
	 We only need to handle $\mathcal P_5,$ because the 
estimation on  $\mathcal P_6$ is similar so that we   omit it for brevity.  
	 Integrating by parts yields, for any $\kappa>0,$ 
	 \begin{eqnarray*}
	 	\mathcal P_5  &=&\int_0^t \bigg(2  \eps \sum_{j=0}^{m-1}{{m-1}\choose j} \inner{\partial_x^{j+1} \partial_y\omega}\partial_x^{m-j+1} u,~ m^2    g_m\bigg)_{L^2}	\,ds\\
	 	 &=&-2\eps m^2    \int_0^t \bigg(   \sum_{j=0}^{[m/2]}\frac{\inner{m-1}!}{ j!(m-1-j)!}   \inner{\partial_x^{j+1} \partial_y\omega}\partial_x^{m-j} u,~   \partial_x   g_m\bigg)_{L^2}	\,ds\\
	 	&&-2\eps m^2  \int_0^t \bigg(     \sum_{j=0}^{[m/2]}\frac{\inner{m-1}!}{ j!(m-1-j)!}   \inner{\partial_x^{j+2} \partial_y\omega}\partial_x^{m-j} u,~        g_m\bigg)_{L^2}	\,ds\\
	 	&&-2  \eps  m^2 \int_0^t \bigg( \sum_{j=[m/2]+1}^{m-1}\frac{\inner{m-1}!}{ j!(m-1-j)!} \inner{\partial_x^{j+1} \omega}\partial_x^{m-j+1} u,~  \partial_y g_m\bigg)_{L^2}	\,ds\\
	 	&&-2  \eps  m^2 \int_0^t \bigg( \sum_{j=[m/2]+1}^{m-1}\frac{\inner{m-1}!}{ j!(m-1-j)!}\inner{\partial_x^{j+1} \omega}\partial_x^{m-j+1} \omega,~    g_m\bigg)_{L^2}	\,ds\\
	 	 &\stackrel{\rm def}{=} & K_1+K_2+K_3+K_4.
	 	 \end{eqnarray*}
Moreover,	 	 following the arguments used in Lemma \ref{lemp3} for $J_3$, we see the    
	 \begin{eqnarray*}
	 	K_2+K_4\leq \frac{C \big[\inner{m-6}!\big]^{2\sigma}}{\rho^{2(m-5)}} \inner{ \int_0^t \abs{u(s)}_{  \rho,\sigma}^3   \,dt+   \int_0^t   \frac{\abs{u(s)}_{  \tilde\rho,\sigma} ^2}{\tilde\rho-\rho} \,dt}.
	 \end{eqnarray*}
	 As for $K_3$ we have, for any $\kappa>0,$
	 	 \begin{eqnarray*}
	 K_3 
	 	&\leq& 2  \kappa \eps m^2 \int_0^t  \norm{ \partial_y  g_m}_{L^2}^2	\,ds +\kappa^{-1}\eps m^2 \int_0^t  \norm{\partial_x^2u}_{L^\infty}^2\norm{\partial_x^m \omega}_{L^2} ^2	\,ds \\
	 	&&+\kappa^{-1}\eps m^2 \int_0^t \bigg[ \sum_{j=[m/2]+1}^{m-2}\frac{\inner{m-1}!}{ j!(m-1-j)!}  \norm{\partial_x^{m-j+1} u}_{L^\infty}  \norm{\partial_x^{j+1}  \omega}_{L^2} \bigg]^2	\,ds,
	 \end{eqnarray*}
	 where the last term was bounded  above by 
	 \begin{eqnarray*}
	 \frac{C\kappa^{-1} \big[\inner{m-6}!\big]^{2\sigma}}{\rho^{2(m-5)}} \inner{ \int_0^t \abs{u(s)}_{  \rho,\sigma}^4   \,dt+   \int_0^t   \frac{\abs{u(s)}_{  \tilde\rho,\sigma} ^2}{\tilde\rho-\rho} \,dt},	
	 \end{eqnarray*}
	 which can be derived similarly as the terms $J_1$-$J_4$ in Lemma \ref{lemp3}.  On the other hand,  for the second  term above, we use the  interpolation inequality to obtain, observing  the fact that $3/2\leq \sigma\leq 2$ as well as the last inequality in \eqref{condi},  
	 \begin{eqnarray*}
	 \kappa^{-1}\eps m^2 \int_0^t  \norm{\partial_x^2u}_{L^\infty}^2\norm{\partial_x^m \omega}_{L^2} ^2	\,ds
	  &\leq&  \kappa^{-1}  \eps      m^2  \int_0^t       \inner{ m^{-2}    \norm{\partial_x^{m+1} \omega}_{L^2} ^2 + m^2  \norm{\partial_x^{m-1} \omega}_{L^2} ^2	} ds\\ 
	 &\leq& \kappa^{-1} \eps   \int_0^t      \norm{\partial_x^{m+1} \omega}_{L^2} ^2\,ds+  \frac{C\kappa^{-1}  \big[\inner{m-6}!\big]^{2\sigma}}{\rho^{2(m-5)}}   \int_0^t   \frac{\abs{u(s)}_{  \tilde\rho,\sigma} ^2}{\tilde\rho-\rho} ds.
	\end{eqnarray*}
	Thus, combining the estimates above we obtain the upper bound for $K_3,$ that is, 
	\begin{eqnarray*}
		K_3&\leq&   2 \kappa \eps m^2 \int_0^t  \norm{ \partial_y  g_m}_{L^2}^2	\,ds +\kappa^{-1} \eps    \int_0^t      \norm{\partial_x^{m+1} \omega}_{L^2} ^2\,ds\\
		&&+  \frac{C\kappa^{-1}  \big[\inner{m-6}!\big]^{2\sigma}}{\rho^{2(m-5)}}    \inner{ \int_0^t \inner{\abs{u(s)}_{\rho,\sigma}^3+\abs{u(s)}_{  \rho,\sigma}^4 }  \,ds+   \int_0^t   \frac{\abs{u(s)}_{  \tilde\rho,\sigma} ^2}{\tilde\rho-\rho} \,ds}.
	\end{eqnarray*}
The estimation on $K_1$ is  similar, and we have 
\begin{eqnarray*}
	K_1& \leq&   \kappa \eps m^2 \int_0^t  \norm{ \partial_x  g_m}_{L^2}^2	\,ds +\kappa^{-1} \eps    \int_0^t      \norm{\partial_x^{m+1} u}_{L^2} ^2\,ds\\
		&&+  \frac{C\kappa^{-1}  \big[\inner{m-6}!\big]^{2\sigma}}{\rho^{2(m-5)}}    \inner{ \int_0^t \inner{\abs{u(s)}_{\rho,\sigma}^3+\abs{u(s)}_{  \rho,\sigma}^4 }  \,ds+   \int_0^t   \frac{\abs{u(s)}_{  \tilde\rho,\sigma} ^2}{\tilde\rho-\rho} \,ds}.
\end{eqnarray*}
Then the upper bound for $\mathcal P_5$ follows.
Similar argument works for $\mathcal P_6.$   Then the proof is then completed. 
	 \end{proof}

\begin{proof}
	[Completion of the proof of Proposition \ref{prpenmon}]  Combining  \eqref{entildgm} and the estimates in  lemmas \ref{lemp3}-\ref{lemp6p5}, we have, for any $\kappa>0,$
	\begin{eqnarray*}
	&& 	 \frac{1}{2}m^2\norm{   g_m(t)}_{L^2}^2+m^2\int_0^t\norm{\partial_y g_m(s)}_{L^2}^2\,ds+\eps m^2 \int_0^t  \norm{ \partial_x   g_m}_{L^2}^2	\,ds  \\
& 	\leq&    \frac{1}{2}m^2\norm{   g_m(0)}_{L^2}^2+4\kappa m^2\int_0^t\norm{\partial_y g_m(s)}_{L^2}^2\,ds +\kappa \eps m^2 \int_0^t  \norm{ \partial_x   g_m}_{L^2}^2	\,ds\\
&&+ \kappa ^{-1}\eps    \int_0^t      \inner{\norm{\partial_x^{m+1} u}_{L^2} ^2+\norm{\partial_x^{m+1} \omega}_{L^2} ^2}\,ds +C \kappa^{-1}m^2  \int_0^t \norm{  \partial_y f_{m-1}}_{L^2}^2 ds\\
&&+\frac{C\kappa^{-1}\big[\inner{m-6}!\big]^{2\sigma}}{\rho^{2(m-5)}}    
 \inner{ \int_0^t\inner{\abs{u(s)}_{\rho,\sigma}^2 + \abs{u(s)}_{\rho,\sigma}^4}ds  +\int_0^t   \frac{\abs{u(s)}_{\tilde \rho,\sigma}^2}{\tilde \rho-\rho}ds}.
 \end{eqnarray*}
	The second and third  terms on the right sides can be absorbed    provided  
  $\kappa\leq 1/4.$   Moreover,  recalling the definition of $\abs{\cdot}_{\rho,\sigma}$ (see Definition \ref{gevspace}), we have
  \begin{eqnarray*}
  m^2\norm{   g_m(0)}_{L^2}^2\leq 	 \frac{C\big[(m-6)!\big]^{2\sigma}}{\rho^{2(m-5)}}\abs{u_0}_{\rho,\sigma}^2.
  \end{eqnarray*}
    Thus   the desired estimate in  Proposition \ref{prpenmon} follows and 
 the proof    is   completed. 
\end{proof}

 \section{Proof of Theorem \ref{uniestgev}:  uniform estimates  away from the critical point}\label{sec4}
 
 In this section, we will perform estimates in the domain where $u^s+u$ admits monotonicity, and derive uniform upper bound for  $f_m$  appearing the definition of $\abs{\cdot}$ (see Definition \ref{gevspace}).  Recall $f_m$ is defined in \eqref{fungeps}, that is, 
 \begin{equation}\label{funoffs}
	f_m=\chi_1\partial_x^m \omega-\chi_1\frac{  \partial_y\omega^s+\partial_y \omega }{\omega^s+\omega }\partial_x^m u=\chi_1\inner{\omega^s+\omega}  \partial_y\inner{\frac{\partial_x^mu}{ \omega^s+\omega }},\quad m\geq 1,
\end{equation}
with $\chi_1$   given in \eqref{chi1}.  Moreover,  we denote   $\tilde f_m$ 
the main component of $f_m$ by 
\begin{equation}\label{+fungeps}
	\tilde f_m=\chi_1'\partial_x^m \omega-\chi_1'\frac{  \partial_y\omega^s+\partial_y \omega }{\omega^s+\omega }\partial_x^m u=\chi_1'\inner{\omega^s+\omega}  \partial_y\inner{\frac{\partial_x^mu}{ \omega^s+\omega }},\quad m\geq 1.
\end{equation}

 The main result in this section is the following proposition.
 
\begin{proposition} 
	\label{prpaway}
	Let $m\geq 6$   and  $u\in L^\infty\inner{[0, T];~X_{\rho_0,\sigma}}$  be a solution to \eqref{regpran} under the assumptions in Theorem \ref{uniestgev}. Then we have, for any $t\in[0,T],$ and   for any pair $\inner{\rho,\tilde\rho}$ with $0<\rho<\tilde\rho\leq \rho_0,$
\begin{eqnarray*}
  &&\norm{ \comi y^{\ell } f_m(t)}_{L^2 }^2+\norm{ \tilde f_m(t)}_{L^2 }^2+  \int_0^t  \norm{\comi y^{\ell} \partial_y f_m}_{L^2}^2 \,ds +\eps\int_0^t \inner{ \norm{\comi y^{\ell} \partial_x f_m}_{L^2}^2 +\norm{\partial_x \tilde f_m}_{L^2}^2}\,ds\\
   &\leq &  \frac{C\big[(m-6)!\big]^{2\sigma}}{\rho^{2(m-5)}}\abs{u_0}_{\rho,\sigma}^2+
 	 \frac{C\big[(m-6)!\big]^{2\sigma}}{\rho^{2(m-5)}} \inner{  \int_0^t   \inner{ \abs{u(s)}_{ \rho,\sigma}^2+\abs{u(s)}_{ \rho,\sigma}^4 }   \,ds+    \int_0^t  \frac{  \abs{u(s)}_{ \tilde\rho,\sigma}^2}{\tilde\rho-\rho}\,ds}. 
\end{eqnarray*}
\end{proposition}

Before presenting the proof   of the above proposition, we give an immediate corollary. 

\begin{corollary}\label{cordiff}
  Let $m\geq 6$   and  $u\in L^\infty\inner{[0, T];~X_{\rho_0,\sigma}}$  be a solution to \eqref{regpran} under the assumptions in Theorem \ref{uniestgev}. Then we have 
		\begin{eqnarray*}
     m^2\norm{g_m(t)}_{L^2} ^2  
 	 &\leq &  \frac{C\big[(m-6)!\big]^{2\sigma}}{\rho^{2(m-5)}}\abs{u_0}_{\rho,\sigma}^2+  C       \int_0^t   \inner{\eps\norm{\partial_x^{m +1 } u}_{L^2}^2+   \eps  \norm{\partial_x^{m+1 } \omega}_{L^2}^2}   \,ds \\
 	 && +  \frac{C \big[\inner{m-6}!\big]^{2\sigma}}{\rho^{2(m-5)}}    
\inner{  \int_0^t\inner{\abs{u(s)}_{\rho,\sigma}^2 + \abs{u(s)}_{\rho,\sigma}^4}\,ds    +   
  \int_0^t   \frac{\abs{u(s)}_{\tilde \rho,\sigma}^2}{\tilde \rho -\rho}\,ds},
\end{eqnarray*}
and
	\begin{eqnarray*}
  	 &&m^{2\sigma-1}	 \int_0^t \norm{ g_m(s)-\tilde g_m(s)}_{L^2}^2\,ds
  	 \\
  	 &\leq &  \frac{C\big[(m-6)!\big]^{2\sigma}}{\rho^{2(m-5)}}\abs{u_0}_{\rho,\sigma}^2+
 	 \frac{C\big[(m-6)!\big]^{2\sigma}}{\rho^{2(m-5)}} \inner{  \int_0^t   \inner{ \abs{u(s)}_{ \rho,\sigma}^2+\abs{u(s)}_{ \rho,\sigma}^4 }   \,ds+    \int_0^t  \frac{  \abs{u(s)}_{ \tilde\rho,\sigma}^2}{\tilde\rho-\rho}\,ds}.
  	\end{eqnarray*}
	 
\end{corollary}
\begin{proof}[Proof of the corollary]	
The first inequality follows from Proposition \ref{prpenmon} and Proposition \ref{prpaway}, and the second one holds because Lemma \ref{lemdiffe} and Proposition \ref{prpaway} as well as the fact that $\sigma\geq 3/2,$  since applying Proposition \ref{prpaway} for $m-1$ gives
\begin{eqnarray*}
	&& m^2  \int_0^t \norm{  \partial_y f_{m-1}}_{L^2}^2 ds\leq  m^{2\sigma-1}  \int_0^t \norm{  \partial_y f_{m-1}}_{L^2}^2 ds\\
	 &\leq & \frac{C\big[(m-6)!\big]^{2\sigma}}{\rho^{2(m-5)}}\abs{u_0}_{\rho,\sigma}^2+
 	 \frac{C\big[(m-6)!\big]^{2\sigma}}{\rho^{2(m-5)}} \inner{  \int_0^t   \inner{ \abs{u(s)}_{ \rho,\sigma}^2+\abs{u(s)}_{ \rho,\sigma}^4 }   \,ds+    \int_0^t  \frac{  \abs{u(s)}_{ \tilde\rho,\sigma}^2}{\tilde\rho-\rho}\,ds}.
\end{eqnarray*} 
The proof is completed.
\end{proof}

 The rest of this section is devote to proving    Proposition \ref{prpaway}
by the following lemmas  and the main tool  used  here is the cancellation property observed in \cite{MW}.

\begin{lemma}
The functions $f_m$ and $\tilde f_m$  defined in \eqref{funoffs}  and \eqref{+fungeps} satisfy the following equations:
\begin{equation}\label{eqgmeps}
 	\Big(\partial_t    +\inner{u^{s}+u}\partial_x+v\partial_y  -\partial_y^2  -\eps\partial_x^2 \Big)  f_m  =F_{m,\eps},
\end{equation}
and 
\begin{equation}\label{eqgmeps+++}
 	\Big(\partial_t    +\inner{u^{s}+u}\partial_x+v\partial_y  -\partial_y^2  -\eps\partial_x^2 \Big)  \tilde f_m  = \tilde F_{m,\eps},
\end{equation}
where 
\begin{eqnarray}
	F_{m,\eps}&=&  -\chi_1\sum_{k=1}^m {m\choose k}\inner{\partial_x^k u} \partial_x^{m-k+1}\omega-\chi_1 \sum_{k=1}^{m-1} {m\choose k}\inner{\partial_x^kv}  \partial_x^{m-k}\partial_y \omega \label{jm1}\\
&&+\chi_1 a \sum_{k=1}^m {m\choose k}\inner{\partial_x^k u} \partial_x^{m-k+1}u+\chi_1 a\sum_{k=1}^{m-1} {m\choose k}\inner{\partial_x^kv}   \partial_x^{m-k} \omega \label{jm2}\\
&&+\chi_1'v\partial_x^m\omega- 2\chi_1'\partial_x^m\partial_y\omega -\chi_1''\partial_x^m\omega-a \inner{\chi_1'v\partial_x^mu- 2\chi_1'\partial_x^m\omega -\chi_1''\partial_x^mu}  \label{jm3}\\
&&+\Big[ \partial_x\omega-\inner{\partial_x u}a    
-2a  \partial_ya  -2\eps\frac{\partial_x \omega}{\omega^s+\omega}\partial_x a \Big]\chi_1 \partial_x^m u  \label{jm4}\\
&&+ 2 \chi_1 \inner{\partial_ya }    \partial_x^m \omega+ 2 \chi_1' \inner{\partial_ya }    \partial_x^m u +2\eps \chi_1 \inner{\partial_xa }  \partial_x^{m+1} u  \label{jm5}
\end{eqnarray}
with
\begin{eqnarray*}
	a = \frac{ \partial_y\omega^s+ \partial_y \omega }{ \omega^s+\omega }, 
\end{eqnarray*} 
and the representation of $\tilde F_{m,\eps}$ is quite similar to $F_{m,\eps}$,  with the functions $\chi_1, \chi_1'$ and $\chi_1''$ in \eqref{jm1}-\eqref{jm5} replaced by $\chi_1', \chi_1''$ and $\chi_1'''$ respectively.  Moreover,
\begin{equation}\label{boundaycond}
	\partial_y f_m\big|_{y=0}=0.
\end{equation}
\end{lemma}

\begin{proof}
This proof is based on direct calculation that will
be sketched in   the Appendix (see Lemma \ref{applem}).	
\end{proof}

In the next two lemmas, we will derive the energy estimates on $f_m$ and $\tilde f_m,$ starting from the equations \eqref{eqgmeps} and \eqref{eqgmeps+++}. 

\begin{lemma}
	\label{lemlow} 
  	We have 	\begin{eqnarray*}
	  \frac{1}{2}\frac{d}{dt} \norm{\comi y^{\ell}f_m }_{L^2}^2+  \frac{1}{2} \norm{\comi y^{\ell}\partial_yf_m}_{L^2}^2 +\eps \norm{\comi y^{\ell}\partial_xf_m}_{L^2}^2  
	  \leq     \Big(\comi y^\ell F_{m,\ell}, \comi y^\ell f_m\Big)_{L^2}   + \frac{C\big[\inner{m-6}!\big]^{2\sigma}}{\rho^{2(m-5)}}    
 \abs{u}_{\rho,\sigma}^2.   
\end{eqnarray*}
The above estimate also holds  when $F_{m,\eps}$ and $f_m$ are replaced respectively by $\tilde F_{m,\eps}$ and $\tilde f_m.$ 
	\end{lemma}

\begin{proof}
We multiply both sides of \eqref{eqgmeps}  by  $\comi y^{2\ell}f_m$ and then take integration over $\mathbb R_+^2;$. Integrating by parts with
 the boundary condition \eqref{boundaycond} gives
\begin{eqnarray*}
	\Big(\comi y^\ell F_{m,\ell}, ~\comi y^\ell f_m\Big)_{L^2} &=&  \int_{\mathbb R_+^2} \Big(\inner{\partial_t+\inner{u^{s}+u}\partial_x+ v \partial_y -\partial_y^2-\eps\partial_x^2    } f_m\Big)\comi y^{2\ell}f_m\, dxdy \\
	& =&\frac{1}{2}\frac{d}{dt}\norm{\comi y^{\ell}f_m}_{L^2}^2+\norm{\comi y^{\ell}\partial_yf_m}_{L^2}^2+\eps\norm{\comi y^{\ell}\partial_xf_m}_{L^2}^2\\
	&&+\int_{\mathbb R_+^2}    \inner{  \partial_y\comi y^{2\ell}} \inner{\partial_yf_m} f_m\, dxdy-\int_{\mathbb R_+^2} v  \inner{  \partial_y\comi y^{\ell}}  \comi y^{\ell}    f_m^2\, dxdy.
	\end{eqnarray*}
Moreover, as for the last two terms on the right side, using the last inequality in \eqref{condi} as well as \eqref{chi2est}, we have
	\begin{eqnarray*}
	&& \Big|\int_{\mathbb R_+^2}    \inner{  \partial_y\comi y^{2\ell}} \inner{\partial_yf_m} f_m\, dxdy\Big|+\Big|\int_{\mathbb R_+^2} v  \inner{  \partial_y\comi y^{\ell}}  \comi y^{\ell}    f_m^2\, dxdy\Big| \\
	 & \leq &  \frac{1}{2}\norm{\comi y^{\ell}\partial_yf_m}_{L^2}^2 +   C\norm{\comi y^{\ell} f_m}^2   \leq   \frac{1}{2}\norm{\comi y^{\ell}\partial_yf_m}_{L^2}^2+\frac{C\big[\inner{m-6}!\big]^{2\sigma}}{\rho^{2(m-5)}} \abs{u}_{\rho,\sigma}^2.
\end{eqnarray*}
 Combining the above equalities gives the desired
estimate and then  completes the     proof of the lemma.
\end{proof}

\begin{lemma}\label{lemuppb}
Let $3/2\leq \sigma\leq 2.$ We have, for any $\kappa>0,$ and for any pair $(\rho,\tilde\rho)$ with $0<\rho<\tilde\rho,$
		\begin{eqnarray*}
		  \Big(\comi y^\ell F_{m,\ell}, ~\comi y^\ell f_m\Big)_{L^2} 
		\leq \kappa \eps \norm{\comi y^{\ell}\partial_x f_m}_{L^2}^2 +   \frac{C\kappa^{-1}\big[(m-6)!\big]^{2\sigma}}{\rho^{2(m-5)}}  \inner{ \abs{u}_{ \rho,\sigma}^2
		+\abs{u}_{ \rho,\sigma}^4+   \frac{ \abs{u}_{\tilde\rho,\sigma}^2}{\tilde \rho-\rho}}.
	\end{eqnarray*}
The above estimate also holds  with $F_{m,\eps}$ and $f_m$ replaced by $\tilde F_{m,\eps}$ and $\tilde f_m$ respectively.
\end{lemma}

 \begin{proof}  
We only need prove the first statement.  To do so,  we estimate
 term by term in the representation of $F_{m,\eps}$.
 
 {\it \underline{Estimate on  the terms in \eqref{jm1}-\eqref{jm3}} }:  We 
apply similar arguments as  for $J_1$-$J_3$ used in Lemma \ref{lemp3} to obtain that 
    \begin{equation*}
    \begin{aligned}
 	&\Big(\comi y^{\ell} \chi_1 \sum_{k=1}^{m-1} {m\choose k}\inner{\partial_x^kv}   \partial_x^{m-k}\partial_y \omega ,~\comi y^{\ell} f_m \Big)_{L^2}+
 	\Big(\comi y^{\ell}\chi_1\sum_{k=1}^m {m\choose k}\inner{\partial_x^k u}  \partial_x^{m-k+1}\omega, ~\comi y^{\ell} f_m \Big)_{L^2}\\
 	&\leq     \frac{C\big[(m-6)!\big]^{2\sigma}}{\rho^{2(m-5)}}  \inner{ \abs{u}_{ \rho,\sigma}^3 +   \frac{ \abs{u}_{\tilde\rho,\sigma}^2}{\tilde \rho-\rho}}.
 	\end{aligned}
 \end{equation*}
 This gives the upper bound for the terms in \eqref{jm1}.   Similarly, observe  
  $\abs{a(t,x,y)}\leq C\comi y^{-1}$ and thus 
 \begin{eqnarray*}
 	&&\Big(\comi y^{\ell}  \chi_1 a \sum_{k=1}^m {m\choose k}\inner{\partial_x^k u} \partial_x^{m-k+1}u+\comi y^{\ell} \chi_1 a \sum_{k=1}^{m-1} {m\choose k}\inner{\partial_x^kv}   \partial_x^{m-k} \omega,~\comi y^{\ell} f_m \Big)_{L^2}\\
 	&\leq &     \frac{C\big[(m-6)!\big]^{2\sigma}}{\rho^{2(m-5)}}  \inner{ \abs{u}_{ \rho,\sigma}^3 +   \frac{ \abs{u}_{\tilde\rho,\sigma}^2}{\tilde \rho-\rho}}.
 \end{eqnarray*}
 This gives the estimates on the terms in \eqref{jm2}. 
Furthermore,  observing $\chi_1'\partial_x^m\partial_y\omega=\chi_1'\chi_2\partial_x^m\partial_y\omega$ due to \eqref{suppch2} and thus using \eqref{chi2est}, we obtain
\begin{eqnarray*}
&&\bigg( \comi y^{ \ell}\Big[\chi_1'v\partial_x^m\omega- 2\chi_1'\partial_x^m\partial_y\omega -\chi_1''\partial_x^m\omega-a\inner{\chi_1'v\partial_x^mu- 2\chi_1'\partial_x^m\omega -\chi_1''\partial_x^mu} \Big],~ \comi y^{ \ell} f_m \bigg)_{L^2}    \\
	&\leq &      \frac{C\big[(m-6)!\big]^{2\sigma}}{\rho^{2(m-5)}}\abs{u}_{ \rho,\sigma}^2.
	\end{eqnarray*}
This gives the upper bound for the terms in  \eqref{jm3}.   

{\it \underline{Estimate on  the terms in \eqref{jm4}-\eqref{jm5}}: }   As for the term in \eqref{jm4},  we can verify that,   for any $y \in  {\rm supp}\,\chi_1,$
 \begin{eqnarray*}
 	\abs{\frac{\partial_x \omega}{\omega^s+\omega} (t,x,y)}\leq C\comi y^{\alpha}\abs{\partial_x \omega(t,x,y)}\leq C\comi y^{\ell}\abs{\partial_x \omega(t,x,y)}\leq C
 \end{eqnarray*}
 due to the fact that $\alpha\leq \ell$ and the last inequality in \eqref{condi}. Similarly  using \eqref{condi} gives, for any $y \in  {\rm supp}\,\chi_1,$ 
 \begin{eqnarray*}
 	 \abs{  a (t,x,y)}+ \abs{\partial_x  a (t,x,y)}+ \abs{\partial_x^2  a (t,x,y)} \leq C\comi y^{-1},
 	  \end{eqnarray*}
 	  and
 	  \begin{eqnarray*}
 	  	  \abs{\partial_y  a (t,x,y)}\leq C\comi y^{-1}\inner{1+\abs{u}_{\rho,\sigma}}.
 	  \end{eqnarray*}
 	  Hence,  we have
 	  \begin{eqnarray*}
	 \Big( \comi y^{ \ell} \big[ \partial_x\omega-\inner{\partial_x u}a    
  -2\eps\frac{\partial_x \omega}{\omega^s+\omega}\partial_x a\big]\chi_1   \partial_x^m u,~ \comi y^{ \ell} f_m\Big)_{L^2} 
	\leq  C \frac{\big[(m-6)!\big]^{2\sigma}}{\rho^{2(m-5)}}\abs{u}_{ \rho,\sigma}^2,
\end{eqnarray*}
and, for any $\kappa>0,$
 \begin{eqnarray*}
  && \Big(\comi y^{\ell}\Big[2\chi_1 \inner{\partial_ya }  \partial_x^m \omega +2 \chi_1' \inner{\partial_ya}    \partial_x^m u +2\eps \chi_1 \inner{\partial_xa}  \partial_x^{m+1} u\Big],~\comi y^{\ell} f_m \Big)_{L^2} \\
   &\leq&\kappa \eps \norm{\comi y^{\ell}\partial_x f_m}_{L^2}^2+ \frac{C\kappa^{-1}\big[(m-6)!\big]^{2\sigma}}{\rho^{2(m-5)}}\inner{\abs{u}_{ \rho,\sigma}^2+\abs{u}_{ \rho,\sigma}^3}.
\end{eqnarray*}
This gives the upper bound for  the terms in \eqref{jm4}-\eqref{jm5}.  The proof of Lemma \ref{lemuppb} is completed.    
  \end{proof}

\begin{proof}
	[Completion of the proof of Proposition \ref{prpaway}]  Combining  Lemma \ref{lemlow}  and Lemma \ref{lemuppb},  we have for any  $\kappa>0,$  	\begin{eqnarray*}
		&& \frac{1}{2}\frac{d}{dt} \norm{\comi y^{\ell}f_m}_{L^2}^2 +\frac{1}{2} \norm{\comi y^{\ell}\partial_y f_m}_{L^2}^2 +\eps \norm{\comi y^{\ell}\partial_x  f_m}_{L^2}^2  \\
		&\leq &     \kappa \eps \norm{\comi y^{\ell}\partial_x f_m}_{L^2}^2 +   \frac{C\kappa^{-1}\big[(m-6)!\big]^{2\sigma}}{\rho^{2(m-5)}}  \inner{ \abs{u}_{ \rho,\sigma}^2
		+\abs{u}_{ \rho,\sigma}^4+   \frac{ \abs{u}_{\tilde\rho,\sigma}^2}{\tilde \rho-\rho}}.
	\end{eqnarray*}
	Letting $\kappa$ be small sufficiently and then  taking integration over $[0,t]$ yields  the estimate on $f_m$ as stated  in    Proposition \ref{prpaway} because
	\begin{eqnarray*}
	 \norm{ \comi y^{\ell } f_m(0)}_{L^2 }^2\leq \frac{C\big[(m-6)!\big]^{2\sigma}}{\rho^{2(m-5)}}\abs{u_0}_{\rho,\sigma}^2.
	\end{eqnarray*}
 The estimation on $\tilde f_m$ is the  same as that of $f_m$.  The proof is thus then completed. 
\end{proof}

 \section{Proof of Theorem \ref{uniestgev}:  uniform estimates near the critical point}
 
Here we will perform the estimation, by virtue of the cut-off function $\chi_2$  introduced in\eqref{cutofffu},  in the domain that contains the non-degenerate critical point.   Precisely, in this part
 we will work on the terms $h_m$ and $\chi_2\partial_y\partial_x^m \omega$,  recalling 
\begin{equation}\label{repreofhm}
h_m=	\chi_2 \partial_x^m\partial_y \omega-\chi_2\frac{  \partial_y^2\omega^s+\partial_y^2 \omega }{\partial_y\omega^s+\partial_y\omega }\partial_x^m \omega,\quad m\geq 1.
\end{equation}
The main result can be stated as follows.

\begin{proposition}
\label{hm}	
Let $m\geq 6$ and let $u\in L^\infty\inner{[0, T];~X_{\rho_0,\sigma}}$  be the solution to \eqref{regpran} under the assumptions in Theorem \ref{uniestgev}.  Then we have, for any $t\in[0,T],$ and   for any pair $\inner{\rho,\tilde\rho}$ with $0<\rho<\tilde\rho<\rho_0,$
\begin{eqnarray*}
  && \norm{h_m}_{L^2}^2 +\norm{\chi_2\partial_y\partial_x^m \omega}_{L^2}^2\\
   &\leq &  \frac{C\big[(m-6)!\big]^{2\sigma}}{\rho^{2(m-5)}}\abs{u_0}_{\rho,\sigma}^2+
 	 \frac{C\big[(m-6)!\big]^{2\sigma}}{\rho^{2(m-5)}} \inner{  \int_0^t   \inner{ \abs{u(s)}_{ \rho,\sigma}^2+\abs{u(s)}_{ \rho,\sigma}^4 }   \,ds+    \int_0^t  \frac{  \abs{u(s)}_{ \tilde\rho,\sigma}^2}{\tilde\rho-\rho}\,ds}.  
\end{eqnarray*}

\end{proposition}

We will prove the above proposition through  the following subsections. As a preliminary 
we first estimate $\chi_2 \partial_x^m\omega$  in Subsection \ref{subsec51}.  The estimation on $h_m$ and  $\chi_2 \partial_y\partial_x^m\omega$  is given in Subsection   \ref{sec52}.

\subsection{Uniform upper bound for $\chi_2 \partial_x^m\omega$}\label{subsec51}
Here we estimate $\chi_2 \partial_x^m\omega,$   following the same cancellation method used in \cite{GM}. The main result can be stated as follows.

\begin{proposition}
	\label{prpnear}
Let $\chi_2$ be the cut-off function given in \eqref{cutofffu}, and let $u\in L^\infty\inner{[0, T];~X_{\rho_0,\sigma}}$  be the solution to \eqref{regpran} under the assumptions in Theorem \ref{uniestgev}. We have, for any    $t\in[0,T],$ and for any   $(\rho,\tilde\rho) $ with $0<\rho<\tilde\rho,$
	\begin{eqnarray*}
 &&\norm{ \chi_2 \partial_x^m   \omega(t)}_{L^2}^2+ \int_0^t \norm{\chi_2 \partial_y \partial_x^m   \omega}_{L^2}^2+\eps \int_0^t \norm{\chi_2   \partial_x^{m+1}\omega}_{L^2}^2 \\
	&\leq &\frac{C\big[(m-6)!\big]^{2\sigma}}{\rho^{2(m-5)}}\abs{u_0}_{\rho,\sigma}^2+
 	 \frac{C\big[(m-6)!\big]^{2\sigma}}{\rho^{2(m-5)}} \inner{  \int_0^t   \inner{ \abs{u(s)}_{ \rho,\sigma}^2+\abs{u(s)}_{ \rho,\sigma}^4 }   \,ds+    \int_0^t  \frac{  \abs{u(s)}_{ \tilde\rho,\sigma}^2}{\tilde\rho-\rho}\,ds}. 
\end{eqnarray*} \end{proposition}

 The proof follows from the same strategy  
as in \cite{GM}.    The key part is to estimate the term
\begin{eqnarray*}
	\abs{\inner{\chi_2'\partial_x^m v, ~\chi_2 \partial_x^m u}_{L^2}}.
\end{eqnarray*}   
Before presenting the proof of  Proposition \ref{prpnear},  we first recall the upper bound for the term above, established  in \cite{GM}  by virtue of a crucial representations  of $\partial_x^m u$  in terms of  $\hat g_m$  (see \cite[Lemma 3]{GM}), with $\hat g_m$ defined by 
\begin{equation}\label{+funof+g}
  \hat g_m =\Big(\psi\inner { \omega^s+\omega}+ 1-\psi  \Big)\inner{\partial_x^m \omega-\frac{   \partial_y \omega^s+ \partial_y \omega }{  \omega^s+\omega }\partial_x^m u}=\Big(\psi+\frac{ 1-\psi }{  \omega^s+\omega} \Big)\inner{  \omega^s+\omega}^2\partial_y\Big(\frac{\partial_x^mu}{\omega^s+\omega}\Big),  \end{equation}
where $m\geq 1$ and $\psi(y)\in C_0^\infty(\mathbb R)$ is a given function such that $\psi\equiv 1$ in $[0, y_0+2\delta].$   Precisely,   by implicit function 
theorem,  if the level set $ \big\{(x,y);~\omega^s+\omega=0\big\}$ of  $\omega^s+\omega$ is non-empty and it is  a curve in $\mathbb R_+^2$  denoted  by $y=\gamma(x).$  Then $\partial_x^m u$ can be represented as
  \begin{equation*}
	\partial_x^m u(t,x,y)=
	 \inner{\omega^s(t,x,y)+\omega(t,x,y)} \int_0^y \Big(\psi+\frac{ 1-\psi }{  \omega^s+\omega} \Big)^{-1}\frac{\hat g_m}{\inner{\omega^s+   \omega}^2 }\,dy,  \end{equation*}
for $y<\gamma(x),$  and for $y>\gamma(x)$ 
\begin{equation*}
	\partial_x^m u(t,x,y)=
		 \inner{\omega^s(t,x,y)+\omega(t,x,y)}\bigg[\int_{y_0+2\delta}^y \Big(\psi+\frac{ 1-\psi }{  \omega^s+\omega} \Big)^{-1}\frac{\hat g_m}{\inner{\omega^s+   \omega}^2 }\,dy+\beta(t,x)\bigg] 
\end{equation*}
with  $\beta(t,x)= \partial_x^m u(t,x,y_0+ 2\delta)/\big(\omega^s(t,y_0+2\delta)+\omega(t,x,y_0+2\delta)\big).$  By virtue of the above representations we can derive that, cf. \cite[Lemma 6]{GM},
\begin{equation*}
 \abs{\inner{\chi_2'\partial_x^m v, ~\chi_2 \partial_x^m u}_{L^2}} \leq C\norm{ \hat g_m}_{L^2\inner{\mathbb R_x\times \big\{0\leq y\leq y_0+2\delta \big\}}}\norm{\partial_x^{m+1}\omega}_{L^2}+ C  \norm{\partial_x^{m}\omega}_{L^2},
 \end{equation*}
 and thus 
  \begin{equation}\label{eskes}
 \abs{\inner{\chi_2'\partial_x^m v, ~\chi_2 \partial_x^m u}_{L^2}} \leq C\norm{ \tilde g_m}_{L^2}\norm{\partial_x^{m+1}\omega}_{L^2}+ C  \norm{\partial_x^{m}\omega}_{L^2},
 \end{equation}
 since $\hat g_m=\tilde g_m$ for $y\in[0, y_0+2\delta].$

The rest is for the proof of  Proposition \ref{prpnear}.  We first have 
the equation for $\chi_2\partial_x^m\omega$:
\begin{equation*} 
\begin{aligned}
 & \Big(\partial_t    +\inner{u^{s}+u}\partial_x +v\partial_y  -\partial_y^2 -\eps\partial_x^2 \Big) \chi_2 \partial_x^m   \omega \\
 =& - \chi_2 \sum_{k=1}^m {m\choose k}\inner{\partial_x^k u} \partial_x^{m-k+1} \omega   - \chi_2  \inner{\partial_y\omega^{s}+\partial_y\omega}\partial_x^m v  - \chi_2\sum_{k=1}^{m-1} {m\choose k}\inner{\partial_x^kv} \partial_x^{m-k}\partial_y \omega \\
 	&+  \chi_2 'v\partial_x^m \omega-  \chi_2'' \partial_x^m \omega- 2\chi_2' \partial_x^m \partial_y \omega .	
   \end{aligned}
\end{equation*}
This can be derived directly  from the equation of the vorticity $\omega.$  In view of \eqref{condi}, we see $\abs{ \partial_y\omega^{s}+\partial_y\omega}\geq c_0/4$ on  supp\,$\chi_2$, and  without loss of generality, we can assume $\abs{ \partial_y\omega^{s}+\partial_y\omega}=-\inner{ \partial_y\omega^{s}+\partial_y\omega}$  on  supp\,$\chi_2$.  This  enables us to take $L^2$ inner product on both sides of  the above equation with the function 
\begin{eqnarray*}
	-\frac{\chi_2\partial_x^m\omega}{\partial_y\omega^{s}+\partial_y\omega}.
\end{eqnarray*}
This gives
 \begin{equation}\label{lines}
\begin{aligned}
&	\inner{\Big(\partial_t    +\inner{u^{s}+u}\partial_x +v\partial_y  -\partial_y^2 -\eps\partial_x^2 \Big) \chi_2 \partial_x^m   \omega,~-\frac{\chi_2\partial_x^m\omega}{\partial_y\omega^{s}+\partial_y\omega}}_{L^2} \\
	=& \inner{\chi_2  \partial_x^m v,~ \chi_2\partial_x^m\omega}_{L^2}-\inner{ \chi_2 \sum_{k=1}^m {m\choose k}\inner{\partial_x^k u} \partial_x^{m-k+1} \omega,~-\frac{\chi_2\partial_x^m\omega}{\partial_y\omega^{s}+\partial_y\omega}}_{L^2}\\
	&-\inner{\chi_2\sum_{k=1}^{m-1} {m\choose k}\inner{\partial_x^kv} \partial_x^{m-k}\partial_y \omega,~-\frac{\chi_2\partial_x^m\omega}{\partial_y\omega^{s}+\partial_y\omega}}_{L^2} \\
 	&+ \inner{ \chi_2 'v\partial_x^m \omega-  \chi_2'' \partial_x^m \omega- 2\chi_2' \partial_x^m \partial_y \omega,~-\frac{\chi_2\partial_x^m\omega}{\partial_y\omega^{s}+\partial_y\omega}}_{L^2}.
\end{aligned}
\end{equation}
As for the last three terms on the right side of the above equation,  we follow the argument used in Lemma \ref{lemuppb} to get 
\begin{equation}\label{lases}
\begin{aligned}
&	 \bigg|   \int_0^t\inner{ \chi_2 \sum_{k=1}^m {m\choose k}\inner{\partial_x^k u} \partial_x^{m-k+1} \omega,~-\frac{\chi_2\partial_x^m\omega}{\partial_y\omega^{s}+\partial_y\omega}}_{L^2}\,ds\bigg|  \\
&\quad+ \bigg|  \int_0^t\inner{\chi_2\sum_{k=1}^{m-1} {m\choose k}\inner{\partial_x^kv} \partial_x^{m-k}\partial_y \omega,~-\frac{\chi_2\partial_x^m\omega}{\partial_y\omega^{s}+\partial_y\omega}}_{L^2}\,ds\bigg| \\
 	&\qquad+   	\bigg| \int_0^t\inner{ \chi_2 'v\partial_x^m \omega-  \chi_2'' \partial_x^m \omega- 2\chi_2' \partial_x^m \partial_y \omega,~-\frac{\chi_2\partial_x^m\omega}{\partial_y\omega^{s}+\partial_y\omega}}_{L^2}\,ds\bigg|\\
 	\leq &\,  \frac{C\big[(m-6)!\big]^{2\sigma}}{\rho^{2(m-5)}}\inner{  \int_0^t   \abs{u(s)}_{\rho,\sigma}^3\,ds+\int_0^t  \frac{ \abs{u(s)}_{\tilde\rho,\sigma}^2}{\tilde\rho-\rho}\,ds}.
\end{aligned}
\end{equation}

In the following  two lemmas, we will  estimate  the    term on the left hand
side of    \eqref{lines} and the first term  on the right side respectively.

\begin{lemma} 
	\label{lemlowbound} 
	We have
	\begin{eqnarray*}
		&&\norm{ \chi_2 \partial_x^m   \omega(t)}_{L^2}^2+ \int_0^t \norm{\chi_2 \partial_y \partial_x^m   \omega}_{L^2}^2+\eps \int_0^t \norm{\chi_2   \partial_x^{m+1}\omega}_{L^2}^2\\
		&\leq &  \frac{C\big[(m-6)!\big]^{2\sigma}}{\rho^{2(m-5)}}   \abs{ u_0}_{\rho,\sigma}^2 +\frac{C\big[(m-6)!\big]^{2\sigma}}{\rho^{2(m-5)}}  \int_0^t   \abs{u(s)}_{\rho,\sigma}^2\,ds \\
		&&+C \bigg|\int_0^t \inner{\Big(\partial_t    +\inner{u^{s}+u}\partial_x +v\partial_y  -\partial_y^2 -\eps\partial_x^2 \Big) \chi_2 \partial_x^m   \omega,~-\frac{\chi_2\partial_x^m\omega}{\partial_y\omega^{s}+\partial_y\omega}}_{L^2}\,ds\bigg|.
	\end{eqnarray*}
\end{lemma}

\begin{proof} Direct computation shows
\begin{eqnarray*}
&&\inner{ \partial_t \chi_2 \partial_x^m   \omega,~-\frac{\chi_2\partial_x^m\omega}{\partial_y\omega^{s}+\partial_y\omega}}_{L^2}\\
&=&\frac{1}{2}\frac{d}{dt}\norm{ \inner{-\partial_y\omega^{s}-\partial_y\omega}^{-1/2} \chi_2 \partial_x^m   \omega}_{L^2}^2-{1\over 2}\inner{  \chi_2 \partial_x^m   \omega,~\frac{\partial_t\partial_y\omega^s+\partial_t\partial_y\omega}{\inner{\partial_y\omega^{s}+\partial_y\omega}^2}\chi_2\partial_x^m\omega}_{L^2}.
\end{eqnarray*}
Then integration by parts  gives 
\begin{eqnarray*}
&&\inner{\Big( \inner{u^{s}+u}\partial_x +v\partial_y    \Big) \chi_2 \partial_x^m   \omega,~-\frac{\chi_2\partial_x^m\omega}{\partial_y\omega^{s}+\partial_y\omega}}_{L^2}
\\
&=&-\frac{1}{2}\inner{  \chi_2 \partial_x^m   \omega,~\frac{\Big( \inner{u^{s}+u}\partial_x +v\partial_y    \Big) \partial_y\omega+v \partial_y^2\omega^s }{\inner{\partial_y\omega^{s}+\partial_y\omega}^2}\chi_2\partial_x^m\omega}_{L^2},
\end{eqnarray*}
\begin{eqnarray*}
&&\inner{ -\partial_y^2 \chi_2 \partial_x^m   \omega,~-\frac{\chi_2\partial_x^m\omega}{\partial_y\omega^{s}+\partial_y\omega}}_{L^2}
\\
&=&\norm{ \inner{-\partial_y\omega^{s}-\partial_y\omega}^{-1/2} \partial_y\inner{\chi_2 \partial_x^m   \omega}}_{L^2}^2 -\frac{1}{2}\inner{  \chi_2 \partial_x^m   \omega,~\frac{\partial_y^3\omega^s+\partial_y^3\omega }{\inner{\partial_y\omega^{s}+\partial_y\omega}^2}\chi_2\partial_x^m\omega}_{L^2}\\
&&+\inner{  \chi_2 \partial_x^m   \omega,~\frac{ \inner{\partial_y^2\omega^s+\partial_y^2\omega}^2 }{\inner{\partial_y\omega^{s}+\partial_y\omega}^3}\chi_2\partial_x^m\omega}_{L^2},
\end{eqnarray*}
and 
\begin{eqnarray*}
&&\inner{ -\eps\partial_x^2 \chi_2 \partial_x^m   \omega,~-\frac{\chi_2\partial_x^m\omega}{\partial_y\omega^{s}+\partial_y\omega}}_{L^2}
\\
&=&\eps \norm{ \inner{-\partial_x^2u^{s}-\partial_y\omega}^{-1/2}   \chi_2 \partial_x^{m+1}   \omega}_{L^2}^2 -\frac{1}{2}\inner{  \chi_2 \partial_x^m   \omega,~\frac{\eps\partial_x^2\partial_y \omega }{\inner{\partial_y\omega^{s}+\partial_y\omega}^2}\chi_2\partial_x^m\omega}_{L^2}\\
&&+\inner{  \chi_2 \partial_x^m   \omega,~\frac{\eps \inner{\partial_x\partial_y\omega}^2 }{\inner{\partial_y\omega^{s}+\partial_y\omega}^3}\chi_2\partial_x^m\omega}_{L^2}.
\end{eqnarray*}
 Moreover,  it follows from the equation of the vorticity that 
\begin{eqnarray*}
&&	\partial_t   \partial_y\omega+\Big(\inner{u^{s}+u}\partial_x+v\partial_y \Big) \partial_y \omega+v \partial_y^2 \omega^s-\partial_y^3 \omega-\eps\partial_x^2 \partial_y\omega\\
&=&-\inner{\omega^s+\omega}\partial_x\omega+\inner{\partial_y\omega^s+\partial_y\omega} \partial_xu.
\end{eqnarray*}
Hence, combining these estimates gives
	\begin{eqnarray*}
			&& \inner{\Big(\partial_t    +\inner{u^{s}+u}\partial_x +v\partial_y  -\partial_y^2 -\eps\partial_x^2 \Big) \chi_2 \partial_x^m   \omega,~-\frac{\chi_2\partial_x^m\omega}{\partial_y\omega^{s}+\partial_y\omega}}_{L^2}\\
			&=&\inner{\Big(\partial_t    +\inner{u^{s}+u}\partial_x +v\partial_y  -\partial_y^2 -\eps\partial_x^2 \Big) \chi_2 \partial_x^m   \omega,~-\frac{\chi_2\partial_x^m\omega}{\partial_y\omega^{s}+\partial_y\omega}}_{L^2}\\
			&=& \frac{1}{2}\frac{d}{dt}\norm{ \inner{-\partial_y\omega^{s}-\partial_y\omega}^{-1/2} \chi_2 \partial_x^m   \omega}_{L^2}^2+\norm{ \inner{-\partial_y\omega^{s}-\partial_y\omega}^{-1/2} \partial_y(\chi_2 \partial_x^m   \omega)}_{L^2}^2\\
			&&+\eps \norm{ \inner{-\partial_x^2u^{s}-\partial_y\omega}^{-1/2}    \chi_2 \partial_x^{m+1}   \omega}_{L^2}^2  -\inner{  \chi_2 \partial_x^m   \omega,~\frac{\partial_y^3\omega^s+\partial_y^3\omega+\eps\partial_x^2\partial_y \omega }{\inner{\partial_y\omega^{s}+\partial_y\omega}^2}\chi_2\partial_x^m\omega}_{L^2}\\
			&&-\frac{1}{2}\inner{  \chi_2 \partial_x^m   \omega,~\frac{\inner{\partial_xu}\inner{\partial_y\omega^s+\partial_y\omega}-\inner{\omega^s+\omega}\partial_x\omega}{\inner{\partial_y\omega^{s}+\partial_y\omega}^2}\chi_2\partial_x^m\omega}_{L^2}\\
			&&+\inner{  \chi_2 \partial_x^m   \omega,~\frac{ \inner{\partial_y^2\omega^s+\partial_y^2\omega}^2 }{\inner{\partial_y\omega^{s}+\partial_y\omega}^3}\chi_2\partial_x^m\omega}_{L^2}+\inner{  \chi_2 \partial_x^m   \omega,~\frac{\eps \inner{\partial_x\partial_y\omega}^2 }{\inner{\partial_y\omega^{s}+\partial_y\omega}^3}\chi_2\partial_x^m\omega}_{L^2},
			\end{eqnarray*}
			with the modulus of the last four terms on the right side  bounded  above by 
			\begin{eqnarray*}
				\frac{C\big[(m-6)!\big]^{2\sigma}}{\rho^{2(m-5)}} \abs{u}_{\rho,\sigma}^2 
			\end{eqnarray*}
			due to the inequalities in \eqref{condi}.  
			Thus by integrating both sides over $[0,t],$	 we have 
					\begin{eqnarray*}
			&&  \norm{  \inner{-\partial_y\omega^{s}-\partial_y\omega}^{-1/2}  \chi_2 \partial_x^m   \omega (t)}_{L^2}^2+ \int_0^t  \norm{ \inner{-\partial_y\omega^{s}-\partial_y\omega}^{-1/2} \chi_2 \partial_y\partial_x^m   \omega }_{L^2}^2 \, ds\\
			&&\qquad+\int_0^t\eps \norm{ \inner{-\partial_x^2u^{s}-\partial_y\omega}^{-1/2}    \chi_2 \partial_x^{m+1}   \omega}_{L^2}^2\,ds\\
			&\leq &\norm{ \inner{-\partial_y\omega^{s}(0)-\partial_y\omega(0)}^{-1/2} \chi_2 \partial_x^m   \omega (0)}_{L^2}^2+ \frac{C\big[(m-6)!\big]^{2\sigma}}{\rho^{2(m-5)}}\int_0^t  \abs{u(s)}_{\rho,\sigma}^2 \, ds\\
			&&+ 2 \bigg|\int_0^t \inner{\Big(\partial_t    +\inner{u^{s}+u}\partial_x +v\partial_y  -\partial_y^2 -\eps\partial_x^2 \Big) \chi_2 \partial_x^m   \omega,~-\frac{\chi_2\partial_x^m\omega}{\partial_y\omega^{s}+\partial_y\omega}}_{L^2}\,ds\bigg|.
			\end{eqnarray*}	
Observe that $ \inner{-\partial_y\omega^{s}-\partial_y\omega}^{-1/2} \geq \sqrt{c_0}/2$ on supp\, $\chi_2$ and that the first term on the right side is  bounded
above by  
\begin{eqnarray*}
	  \frac{C\big[(m-6)!\big]^{2\sigma}}{\rho^{2(m-5)}}   \abs{ u_0}_{\rho,\sigma}^2.
\end{eqnarray*}
 Then  the estimate in Lemma 	 \ref{lemlowbound} follows.   The proof is completed.
\end{proof}

\begin{lemma}\label{lem09}
	We have
	\begin{eqnarray*}
 &&\abs{\int_0^t \inner{\chi_2  \partial_x^m v,~ \chi_2\partial_x^m\omega}_{L^2}\,ds}
	 \\
	 &\leq&   \frac{C\big[(m-6)!\big]^{2\sigma}}{\rho^{2(m-5)}}\abs{u_0}_{\rho,\sigma}^2+
 	 \frac{C\big[(m-6)!\big]^{2\sigma}}{\rho^{2(m-5)}} \inner{  \int_0^t   \inner{ \abs{u(s)}_{ \rho,\sigma}^2+\abs{u(s)}_{ \rho,\sigma}^4 }   \,ds+    \int_0^t  \frac{  \abs{u(s)}_{ \tilde\rho,\sigma}^2}{\tilde\rho-\rho}\,ds}.
\end{eqnarray*}
 \end{lemma}

\begin{proof}
Integrating by parts gives
\begin{eqnarray*}
		   \inner{\chi_2  \partial_x^m v,~ \chi_2\partial_x^m\omega}_{L^2}
 =-2\inner{\chi_2 \partial_x^m v,~ \chi_2'\partial_x^mu}_{L^2}+\inner{\chi_2   \partial_x^{m+1} u,~ \chi_2\partial_x^mu}_{L^2}
		 =  -2\inner{\chi_2  \partial_x^{m} v,~ \chi_2'\partial_x^{m}u}_{L^2}.  \end{eqnarray*}
Moreover, in view of  \eqref{eskes},  we have
  \begin{eqnarray*}
	\abs{ \inner{\chi_2  \partial_x^{m} v,~ \chi_2'\partial_x^{m}u}_{L^2}}&\leq& C\norm{\tilde g_m}_{L^2 }\norm{\partial_x^{m+1}\omega}_{L^2}+ C  \norm{\partial_x^{m}\omega}_{L^2}^2\\
	&\leq& C\norm{ g_m}_{L^2 }\norm{\partial_x^{m+1}\omega}_{L^2}+C\norm{ g_m-\tilde g_m}_{L^2 }\norm{\partial_x^{m+1}\omega}_{L^2}+ C  \norm{\partial_x^{m}\omega}_{L^2}^2.
\end{eqnarray*}
Thus
\begin{equation}\label{cret}
\begin{aligned}
 & \abs{\int_0^t\inner{\chi_2  \partial_x^m v,~ \chi_2\partial_x^m\omega}_{L^2}\,ds}\\
   \leq &  C \int_0^t \norm{g_m}_{L^2 }\norm{\partial_x^{m+1}\omega}_{L^2}\,ds+ C\int_0^t  \norm{\partial_x^{m}\omega}_{L^2}^2\,ds +C \int_0^t \norm{g_m-\tilde g_m}_{L^2 }\norm{\partial_x^{m+1}\omega}_{L^2}\,ds.
   	\end{aligned}
	\end{equation} 
	On the other hand, since $\sigma\leq 2$,
 we can use \eqref{etan} and \eqref{chi2est} as well as  the statements $(ii)$-$(iii)$ in  Lemma \ref{lemequa} to get    
	\begin{eqnarray*}
		\int_0^t  \norm{ g_m}_{L^2}\norm{\partial_x^{m+1}\omega}_{L^2} \,ds &\leq &   \frac{ C\big[(m-6)!\big]^{2 \sigma}}{\rho^{ 2( m-5) }} \int_0^t  \abs{u}_{\rho,\sigma} \abs{u}_{\tilde\rho,\sigma}  \frac{m^{\sigma-1} \rho^{m-5}}{\tilde\rho^{m-4}}\,ds \\
	&\leq&    \frac{ C\big[(m-6)!\big]^{2 \sigma}}{\rho^{ 2( m-5) }} \int_0^t  \frac{  \abs{u(s)}_{\tilde\rho,\sigma}^2 }{\tilde\rho -\rho}\,ds,
	\end{eqnarray*} 
	and
	\begin{eqnarray*}
		\int_0^t  \norm{\partial_x^{m}\omega}_{L^2}^2\leq  \frac{  \big[(m-6)!\big]^{2 \sigma}}{\rho^{ 2( m-5) }} \int_0^t  \abs{u}_{\rho,\sigma} ^2  \,ds.
	\end{eqnarray*}
	Moreover, by  the second inequality in Corollary  \ref{cordiff}  we have
	 \begin{eqnarray*}
	&& \int_0^t \norm{ \tilde g_m- g_m}_{L^2}\norm{\partial_x^{m+1}\omega}_{L^2}\,ds\\
	&\leq & m^{2\sigma-1} \int_0^t \norm{ \tilde g_m- g_m}_{L^2}^2 \,ds+ m^{-2\sigma+1}\int_0^t \norm{\partial_x^{m+1}\omega}_{L^2}^2\,ds\\
	&\leq &\frac{C\big[(m-6)!\big]^{2\sigma}}{\rho^{2(m-5)}}\abs{u_0}_{\rho,\sigma}^2+
 	 \frac{C\big[(m-6)!\big]^{2\sigma}}{\rho^{2(m-5)}} \inner{  \int_0^t   \inner{ \abs{u(s)}_{ \rho,\sigma}^2+\abs{u(s)}_{ \rho,\sigma}^4 }   \,ds+    \int_0^t  \frac{  \abs{u(s)}_{ \tilde\rho,\sigma}^2}{\tilde\rho-\rho}\,ds},
	\end{eqnarray*}
where in the last inequality we have
used  the statement $(iii)$ in Lemma \ref{lemequa}.  Inserting these inequalities into \eqref{cret} gives the desired estimate. 
		Thus the proof of Lemma \ref{lem09} is completed. 
\end{proof}

\begin{proof}
[Completion of the proof of Proposition \ref{prpnear}]
In view of \eqref{lines}, we combine \eqref{lases} and  the estimates  in Lemmas \ref{lemlowbound}-\ref{lem09} to conclude   \begin{eqnarray*}
	&&\norm{ \chi_2 \partial_x^m   \omega(t)}_{L^2}^2+ \int_0^t \norm{\chi_2 \partial_y \partial_x^m   \omega}_{L^2}^2+\eps \int_0^t \norm{\chi_2   \partial_x^{m+1}\omega}_{L^2}^2 \\
	&\leq &\frac{C\big[(m-6)!\big]^{2\sigma}}{\rho^{2(m-5)}}\abs{u_0}_{\rho,\sigma}^2+
 	 \frac{C\big[(m-6)!\big]^{2\sigma}}{\rho^{2(m-5)}} \inner{  \int_0^t   \inner{ \abs{u(s)}_{ \rho,\sigma}^2+\abs{u(s)}_{ \rho,\sigma}^4 }   \,ds+    \int_0^t  \frac{  \abs{u(s)}_{ \tilde\rho,\sigma}^2}{\tilde\rho-\rho}\,ds}.
\end{eqnarray*}
This is just   the estimate in Proposition \ref{prpnear}. The proof is then completed.	
\end{proof}

\subsection{Proof of Proposition \ref{hm}. ( uniform estimate for $h_m$ and $\chi_2\partial_y\partial_x^m\omega$)} \label{sec52}
This part is devoted to proving Proposition \ref{hm}. 
We begin with the estimation on $h_m.$	Note that   $h_m$  solves the equation (see Lemma \ref{fme} in the Appendix  for its derivation): 
\begin{equation}\label{eqhme}
\begin{aligned}
	&\Big(\partial_t    +\inner{u^{s}+u}\partial_x +v\partial_y  -\partial_y^2  -\eps\partial_x^2 \Big) h_m\\
	=&P_{m}+2\inner{\partial_yb  }\partial_y (\chi_2\partial_x^m\omega)+2\eps\inner{\partial_xb  }\partial_x (\chi_2\partial_x^m\omega)\\
	&+ \chi_2 b   \sum_{k=1}^m {m\choose k}\inner{\partial_x^k u} \partial_x^{m-k+1} \omega+\chi_2b  \sum_{k=1}^{m-1} {m\choose k}\inner{\partial_x^kv} \partial_x^{m-k}\partial_y \omega \\ 
	&  -  \chi_2 ' b v\partial_x^m \omega+  b \chi_2''  \partial_x^m \omega+ 2b  \chi_2' \partial_x^m \partial_y \omega \\
	& -\chi_2\sum_{k=1}^m {m\choose k}\inner{\partial_x^k u} \partial_x^{m-k+1} \partial_y\omega-\chi_2\sum_{k=1}^{m-1} {m\choose k}\inner{\partial_x^kv} \partial_x^{m-k}\partial_y^2\omega\\
	& +  \chi_2 'v\partial_x^m \partial_y\omega-  \chi_2'' \partial_x^m \partial_y\omega- 2\chi_2' \partial_x^m \partial_y^2 \omega  -\chi_2g_{m+1}
,
\end{aligned}
\end{equation}	
where  $b =\frac{\partial_y^2\omega^{s}+\partial_y^2\omega}{\partial_y\omega^{s}+\partial_y\omega}$ and 
\begin{eqnarray*}
	P_{m}&=& \frac{2\Big(\inner{ \omega^s+ \omega}\partial_x \partial_y\omega-\inner{\partial_x u} \inner{\partial_y^2\omega^s+\partial_y^2\omega}\Big)\chi_2\partial_x^m\omega}{\partial_y\omega^s+\partial_y\omega}\\
	&&-\frac{  \Big(\omega\partial_x \omega-\inner{\partial_x u}  \inner{\partial_y\omega^s+\partial_y\omega}\Big)\inner{\partial_y^2\omega^s+\partial_y^2\omega}\chi_2\partial_x^m\omega}{\inner{\partial_y\omega^s+\partial_y\omega}^2}\\
	&&-  \frac{ 2 \Big(  \inner{\partial_y^3\omega^s+\partial_y^3\omega}  \inner{\partial_y^2\omega^s+\partial_y^2\omega} + \eps\inner{\partial_x\partial_y^2\omega}\partial_{x} \partial_y\omega\Big)\chi_2\partial_x^m\omega}{\inner{\partial_y\omega^s+\partial_y\omega}^2}\\
	&&+\frac{  2\Big(\inner{\partial_y^2\omega^s+\partial_y^2\omega}^2+ \eps\inner{ \partial_x\partial_y\omega}^2\Big)\inner{\partial_y^2\omega^s+\partial_y^2\omega}\chi_2\partial_x^m\omega}{\inner{\partial_y\omega^s+\partial_y\omega}^3}.
	\end{eqnarray*}
	Clearly,
 \begin{equation}\label{eqhm}
 \begin{aligned}
	&  \frac{1}{2} \norm{h_m(t)}_{L^2}^2+   \int_0^t   \norm{ \partial_y h_m(s)}_{L^2}^2\,ds+ \eps\int_0^t  \norm{ \partial_x h_m(s)}_{L^2}^2 \,ds   \\
		 \leq  & ~ \frac{1}{2}  \norm{h_m(0)}_{L^2}^2+\int_0^t\Big(\big(\partial_t    +\inner{u^{s}+u}\partial_x +v\partial_y  -\partial_y^2  -\eps\partial_x^2 \big) h_m,~h_m\Big)_{L^2}\\
		 \leq  & ~ \frac{1}{2} \frac{ \big[(m-6)!\big]^{2\sigma}}{\rho^{2(m-5)}}\abs{u_0}_{\rho,\sigma}^2+\int_0^t\Big(\big(\partial_t    +\inner{u^{s}+u}\partial_x +v\partial_y  -\partial_y^2  -\eps\partial_x^2 \big) h_m,~h_m\Big)_{L^2}.
		\end{aligned}
	\end{equation}
	It remains to estimate the  terms on the right
hand side of the last inequality that
will be given  in the following two lemmas.
	
	\begin{lemma}
	\label{lemea}
	We have, for any small $\kappa>0,$
	\begin{eqnarray*}
		&&\Big|\int_0^t\Big(P_{m}+2\inner{\partial_yb  }\partial_y (\chi_2\partial_x^m\omega)+2\eps\inner{\partial_xb  }\partial_x (\chi_2\partial_x^m\omega), ~h_m\Big)_{L^2}ds\Big|\\
		&&\quad+\Big|\int_0^t\Big(\chi_2 b   \sum_{k=1}^m {m\choose k}\inner{\partial_x^k u} \partial_x^{m-k+1} \omega+\chi_2b  \sum_{k=1}^{m-1} {m\choose k}\inner{\partial_x^kv} \partial_x^{m-k}\partial_y \omega, ~h_m\Big)_{L^2}ds\Big|\\
		&&\quad+\Big|\int_0^t\Big(-\chi_2\sum_{k=1}^m {m\choose k}\inner{\partial_x^k u} \partial_x^{m-k+1} \partial_y\omega-\chi_2\sum_{k=1}^{m-1} {m\choose k}\inner{\partial_x^kv} \partial_x^{m-k}\partial_y^2\omega, ~h_m\Big)_{L^2}ds\Big|\\
		&\leq&   \kappa\int_{0}^t  \norm{\partial_y h_{m}}_{L^2}^2+ \kappa\eps \int_{0}^t  \norm{\partial_x h_{m}}_{L^2}^2+\frac{C\kappa^{-1}\big[(m-6)!\big]^{2\sigma}}{\rho^{2(m-5)}} \inner{\int_{0}^t \inner{ \abs{u}_{\rho,\sigma}^2+\abs{u}_{\rho,\sigma}^3}ds+\int_{0}^t  \frac{\abs{u}_{\tilde\rho,\sigma}^2}{\tilde\rho-\rho}ds
}.
	\end{eqnarray*}
	\end{lemma}

\begin{proof}
Since the proof is  similar to the one for  Lemma \ref{lemuppb},    we omit it for  brevity.
\end{proof}

	\begin{lemma} \label{lemfinesre}  Let $\sigma\leq 2.$  We have,  for any small $\kappa>0,$
		\begin{eqnarray*}
		&&\Big|\int_0^t\Big(-  \chi_2 ' b v\partial_x^m \omega+  b \chi_2''  \partial_x^m \omega+ 2b  \chi_2' \partial_x^m \partial_y \omega, ~h_m\Big)_{L^2}ds\Big|\\
		&&\qquad +\Big|\int_0^t\Big( \chi_2 'v\partial_x^m \partial_y\omega-  \chi_2'' \partial_x^m \partial_y\omega- 2\chi_2' \partial_x^m \partial_y^2 \omega  -\chi_2g_{m+1}
, ~h_m\Big)_{L^2}ds\Big|\\
&\leq &    \kappa\int_{0}^t  \norm{\partial_y h_{m}}_{L^2}^2+C \kappa^{-1}\inner{\int_{0}^t  \norm{\partial_y f_{m}}_{L^2}^2+\frac{ \big[(m-6)!\big]^{2\sigma}}{\rho^{2(m-5)}} \int_{0}^t  \abs{u}_{\rho,\sigma}^2\,ds+ \int_0^t  \frac{  \abs{u(s)}_{\tilde\rho,\sigma}^2 }{\tilde\rho -\rho}\,ds}.
		\end{eqnarray*}
	\end{lemma}

\begin{proof}
It is clear that	
\begin{eqnarray*}
	\Big|\int_0^t\big(-  \chi_2 ' b v\partial_x^m \omega+  b \chi_2''  \partial_x^m \omega,~h_m\big)_{L^2}\Big| \leq \frac{C\big[(m-6)!\big]^{2\sigma}}{\rho^{2(m-5)}} \int_{0}^t  \abs{u}_{\rho,\sigma}^2\,ds.
\end{eqnarray*}
Moreover, integrating by parts yields, for any $\kappa>0,$ 
\begin{eqnarray*}
	\Big|\int_0^t\big(2b\chi_2' \partial_x^m \partial_y \omega,~h_m\big)_{L^2}\Big| &\leq&\Big|\int_0^t\big(\big[ \partial_y\inner{2b\chi_2' }\big]\partial_x^m  \omega,~h_m\big)_{L^2}\Big|+\Big|\int_0^t\big( 2b\chi_2' \partial_x^m  \omega,~\partial_y h_m\big)_{L^2}\Big|\\
	&\leq & \kappa\int_{0}^t  \norm{\partial_y h_{m}}_{L^2}^2\,ds+\frac{C\kappa^{-1}\big[(m-6)!\big]^{2\sigma}}{\rho^{2(m-5)}} \int_{0}^t  \abs{u}_{\rho,\sigma}^2\,ds.
\end{eqnarray*}
Similarly
\begin{eqnarray*}
	\Big|\int_0^t\big(\chi_2 'v\partial_x^m \partial_y\omega-  \chi_2'' \partial_x^m \partial_y\omega,~h_m\big)_{L^2}\Big| \leq  \kappa\int_{0}^t  \norm{\partial_y h_{m}}_{L^2}^2\,ds+\frac{C\kappa^{-1}\big[(m-6)!\big]^{2\sigma}}{\rho^{2(m-5)}} \int_{0}^t  \abs{u}_{\rho,\sigma}^2\,ds,
\end{eqnarray*}
and 
\begin{eqnarray*}
 	\Big|\int_0^t\big(2\chi_2'   \partial_x^m \partial_y^2 \omega,~h_m\big)_{L^2}\Big| 	&\leq &  \kappa\int_{0}^t  \norm{\partial_y h_{m}}_{L^2}^2\,ds +C\kappa^{-1}\inner{\int_{0}^t  \norm{\partial_y \partial_x^m\omega}_{L^2}^2\,ds +  \int_{0}^t  \norm{ h_{m}}_{L^2}^2\,ds} \\
 	&\leq& \kappa\int_{0}^t  \norm{\partial_y h_{m}}_{L^2}^2+C \kappa^{-1}\inner{\int_{0}^t  \norm{\partial_y f_{m}}_{L^2}^2+\frac{ \big[(m-6)!\big]^{2\sigma}}{\rho^{2(m-5)}} \int_{0}^t  \abs{u}_{\rho,\sigma}^2\,ds},
\end{eqnarray*}
where in the last inequality we have used
 the first estimate in \eqref{fm1}.  Finally,  we use  \eqref{chi2est} and the statements $(ii)$-$(iii)$ in  Lemma \ref{lemequa}, to get by noticing
 $\sigma\leq 2$  that
 \begin{eqnarray*}
	\Big|\int_0^t\big(\chi_2g_{m+1},~h_m\big)_{L^2}\Big|&\leq& \int_{0}^t \norm{g_{m+1}}_{L^2}\norm{h_{m}}_{L^2} 
 \leq   \frac{ C\big[(m-6)!\big]^{2 \sigma}}{\rho^{ 2( m-5) }} \int_0^t  \abs{u}_{\rho,\sigma} \abs{u}_{\tilde\rho,\sigma}  \frac{m^{\sigma-1} \rho^{m-5}}{\tilde\rho^{m-4}}\,ds \\
	&\leq&    \frac{ C\big[(m-6)!\big]^{2 \sigma}}{\rho^{ 2( m-5) }} \int_0^t  \frac{  \abs{u(s)}_{\tilde\rho,\sigma}^2 }{\tilde\rho -\rho}\,ds.
\end{eqnarray*}
 Combining these inequalities gives the estimate as desired. The proof is completed.
\end{proof}

\begin{proof}
	[Completion of the proof of Proposition \ref{hm}]  In view of \eqref{eqhm} and \eqref{eqhme}, we combine the estimates in  Lemma \ref{lemea}-\ref{lemfinesre} to obtain that by choosing $\kappa$ being sufficiently small,  
\begin{eqnarray*}
	&& \norm{h_m(t)}_{L^2}^2+   \int_0^t   \norm{ \partial_y h_m(s)}_{L^2}^2\,ds+ \eps\int_0^t  \norm{ \partial_x h_m(s)}_{L^2}^2 \,ds
	\\
	&\leq& C\int_{0}^t  \norm{\partial_y f_{m}}_{L^2}^2\,ds + \frac{C\big[(m-6)!\big]^{2\sigma}}{\rho^{2(m-5)}}\abs{u_0}_{\rho,\sigma}^2\\
	&&+
 	 \frac{C\big[(m-6)!\big]^{2\sigma}}{\rho^{2(m-5)}} \inner{  \int_0^t   \inner{ \abs{u(s)}_{ \rho,\sigma}^2+{u(s)}_{ \rho,\sigma}^3 }   \,ds+    \int_0^t  \frac{  \abs{u(s)}_{ \tilde\rho,\sigma}^2}{\tilde\rho-\rho}\,ds}\\
 	 &\leq&   \frac{C\big[(m-6)!\big]^{2\sigma}}{\rho^{2(m-5)}}\abs{u_0}_{\rho,\sigma}^2 +
 	 \frac{C\big[(m-6)!\big]^{2\sigma}}{\rho^{2(m-5)}} \inner{  \int_0^t   \inner{ \abs{u(s)}_{ \rho,\sigma}^2+\abs{u(s)}_{ \rho,\sigma}^4 }   \,ds+    \int_0^t  \frac{  \abs{u(s)}_{ \tilde\rho,\sigma}^2}{\tilde\rho-\rho}\,ds},
\end{eqnarray*}
where in the last inequality we have
used Proposition \ref{prpaway}.  This gives the upper bound  for $h_m.$
 Moreover,  in view of \eqref{repreofhm},  
\begin{eqnarray*}
	\norm{\chi_2 \partial_x^m\partial_y \omega(t)}_{L^2}^2 
	 \leq   \norm{h_m(t)}_{L^2}^2+C\norm{\chi_2 \partial_x^m \omega(t)}_{L^2}^2.
	\end{eqnarray*}
	Finally, we  use   Proposition \ref{prpnear} and the estimate on $h_m$ to get
	\begin{eqnarray*}
&&	\norm{\chi_2 \partial_x^m\partial_y \omega(t)}_{L^2}^2\\
& \leq &\frac{C\big[(m-6)!\big]^{2\sigma}}{\rho^{2(m-5)}}\abs{u_0}_{\rho,\sigma}^2+
 	 \frac{C\big[(m-6)!\big]^{2\sigma}}{\rho^{2(m-5)}} \inner{  \int_0^t   \inner{ \abs{u(s)}_{ \rho,\sigma}^2+\abs{u(s)}_{ \rho,\sigma}^4 }   \,ds+    \int_0^t  \frac{  \abs{u(s)}_{ \tilde\rho,\sigma}^2}{\tilde\rho-\rho}\,ds}.
\end{eqnarray*}
The upper bound for $\chi_2 \partial_x^m\partial_y \omega $ follows. 
Thus we complete the proof of Proposition \ref{hm}. 
\end{proof}

\section{Completeness of the proof of Theorem \ref{uniestgev}:  uniform estimates for $\norm{u}_{\rho,\sigma}$}  
\label{sec5}

To complete the proof of Theorem \ref{uniestgev}, it remains to estimate $\norm{u}_{\rho,\sigma}$ that  is given in Definition \ref{defgev}.    We will perform estimates on tangential derivatives and mixed derivatives of $u$ and $\omega$ respectively in the following two subsections.  In the last subsection we will give the proof of Theorem \ref{uniestgev} by combining all
the estimates obtained in the previous sections. 

\subsection{Estimate on tangential derivatives}  The main estimate
 in this subsection can be stated as follows.  

\begin{proposition}
\label{provel+}
Let $u\in L^\infty\inner{[0, T];~X_{\rho_0,\sigma}}$  be the solution to \eqref{regpran} under the assumptions in Theorem \ref{uniestgev}.  Then  for any $m\geq 6,$ for any $t\in[0,T],$ and   for any pair $\inner{\rho,\tilde\rho}$ with $0<\rho<\tilde\rho\leq \rho_0,$ we have
	\begin{eqnarray*}
		 &&\norm{ \comi y^{\ell-1} \partial_x^m u}_{L^2}^2+\norm{ \comi y^{\ell} \partial_x^m\omega}_{L^2}^2 +\eps \int_0^t \inner{\norm{\comi y^{\ell-1} \partial_x^{m+1}u}_{L^2}^2+\norm{ \comi y^{\ell}  \partial_x^{m+1}\omega}_{L^2}^2}\,ds \\
   &\leq & 	\frac{C\big[(m-6)!\big]^{2\sigma}}{\rho^{2(m-5)}}\abs{u_0}_{\rho,\sigma}^2+
 	 \frac{C\big[(m-6)!\big]^{2\sigma}}{\rho^{2(m-5)}} \inner{  \int_0^t   \inner{ \abs{u(s)}_{ \rho,\sigma}^2+\abs{u(s)}_{ \rho,\sigma}^4 }   \,ds+    \int_0^t  \frac{  \abs{u(s)}_{ \tilde\rho,\sigma}^2}{\tilde\rho-\rho}\,ds}.  
	\end{eqnarray*}
\end{proposition}

As a preliminary to prove the above proposition,  we need the following
\begin{lemma} \label{eqnorm}
Let $\chi_1,\chi_2$ be given in  \eqref{chi1} and \eqref{cutofffu}, and  let $u$ satisfy the condition \eqref{condi}. Then
	\begin{equation}\label{++ell1}
	  \norm{\chi_1\comi y^{\ell-1} \partial_x^m u}_{L^2}  +\norm{\chi_1\comi y^{\ell}\partial_x^m \omega}_{L^2} 
	 \leq   C \norm{\comi y^\ell f_m}_{L^2}+C\norm{\tilde f_m}_{L^2}+ C\norm{\chi_2\partial_x^m\omega}_{L^2},
\end{equation}
and 
\begin{equation}\label{higes}
\begin{aligned}
	 & \norm{\chi_1\comi y^{\ell-1} \partial_x^{m+1} u}_{L^2}  +\norm{\chi_1\comi y^{\ell}\partial_x^{m+1} \omega}_{L^2} \\
	 \leq &  C \norm{\comi y^\ell \partial_x f_m}_{L^2}+C\norm{ \partial_x \tilde f_m}_{L^2}+ C\norm{\chi_2\partial_x^{m+1}\omega}_{L^2}+C \norm{\comi y^{\ell-1}\partial_x^m u}_{L^2},
	 \end{aligned}
\end{equation}
with $f_m$ and $\tilde f_m$ defined in \eqref{funoffs} and \eqref{+fungeps}. 
\end{lemma}

\begin{proof} by using the fact that 
\begin{eqnarray*}
 \chi_1 \abs{\omega^s+\omega}\sim \chi_1 (1+y)^{-\alpha}
\end{eqnarray*}
due to \eqref{condi}, 
 integration by parts gives 
\begin{eqnarray*}
	&&\norm{\inner{1+ y}^{\ell-1}\chi_1\partial_x^m u}_{L^2}^2\\ 
	&\leq& C\int_{\mathbb R_x} \int_0^{+\infty}  \inner{1+ y}^{2\ell-2-2\alpha}\inner{\frac{\chi_1\partial_x^m u}{\omega^s+\omega}}^2dy dx\\
	&=&  \frac{C}{2\ell-1-2\alpha}\int_{\mathbb R_x} \int_0^{+\infty} \com{\partial_y \inner{1+ y}^{2\ell-1-2\alpha}}\inner{\frac{\chi_1\partial_x^m u}{\omega^s+\omega}}^2dy dx\\
	&=& - \frac{2C}{2\ell-1-2\alpha}\int_{\mathbb R_x} \int_0^{+\infty}  \inner{1+ y}^{2\ell-1-2\alpha} \inner{\frac{\chi_1\partial_x^m u}{\omega^s+\omega}} \com{\partial_y\inner{\frac{\chi_1\partial_x^m u}{\omega^s+\omega}}}dy dx\\
	&\leq&  C\int_{\mathbb R_x} \int_0^{+\infty}  \Big|\inner{1+ y}^{\ell-1}\chi_1\partial_x^m u\Big| \times\Big|\inner{1+ y}^\ell\inner{\omega^s+\omega}\partial_y\inner{\frac{\chi_1\partial_x^m u}{\omega^s+\omega}}\Big|dy dx\\
	&\leq&  C\norm{\inner{1+ y}^{\ell-1}\chi_1\partial_x^m u}_{L^2}\Big\|\comi y^\ell\inner{\omega^s+\omega}\partial_y\inner{\frac{\chi_1\partial_x^m u}{\omega^s+\omega}}\Big\|_{L^2}. 
\end{eqnarray*}	
This implies
\begin{eqnarray*}
	 \norm{\inner{1+ y}^{\ell-1}\chi_1\partial_x^m u}_{L^2}	 \leq   C \Big\|\comi y^\ell\inner{\omega^s+\omega}\partial_y\inner{\frac{\chi_1\partial_x^m u}{\omega^s+\omega}}\Big\|_{L^2}. 
\end{eqnarray*}
Moreover, for the term on the right side of the
above inequality, we have in view of  the definition of $f_m$ given in \eqref{funoffs} that
\begin{eqnarray*}
	\Big\|\comi y^\ell\inner{\omega^s+\omega}\partial_y\inner{\frac{\chi_1\partial_x^m u}{\omega^s+\omega}}\Big\|_{L^2} \leq  \norm{\comi y^\ell f_m}_{L^2}+\norm{ \chi_1' \comi y^\ell\partial_x^mu}_{L^2} \leq  \norm{\comi y^\ell f_m}_{L^2}+C\norm{ \chi_1' \partial_x^mu}_{L^2}.
\end{eqnarray*}
Thus   
\begin{eqnarray}\label{ellminus1}
	 \norm{\inner{1+ y}^{\ell-1}\chi_1\partial_x^m u}_{L^2}
	 \leq   C \norm{\comi y^\ell f_m}_{L^2}+C\norm{ \chi_1'\partial_x^mu}_{L^2}. 
\end{eqnarray}
Next we estimate the last term in \eqref{ellminus1}.  Observe that
 the condition \eqref{condi} implies
\begin{eqnarray*}
	\abs{	\frac{\partial_y\omega^s+\partial_y \omega}{\omega^s+\omega}}\geq \tilde c>0 \quad \textrm{on supp}\, \chi_1',
\end{eqnarray*}
for some constant $\tilde c$ depending only on the constants $c_0,c_1,\delta$ and $\alpha$ in Assumption \ref{maas}. 
Then 
\begin{eqnarray*}
	\norm{ \chi_1'\partial_x^mu}_{L^2}  \leq  C \Big\| \chi_1' \frac{\partial_y\omega^s+\partial_y \omega}{\omega^s+\omega} \partial_x^mu\Big\|_{L^2}  \leq  C \norm{\tilde f_m}_{L^2}+ C\norm{\chi_1'\partial_x^m\omega}_{L^2} \leq C \norm{\tilde f_m}_{L^2}+ C\norm{\chi_2\partial_x^m\omega}_{L^2},
\end{eqnarray*}
where for the second inequality we have used \eqref{+fungeps}, the definition of $\tilde f_m$,  and the last inequality follows from \eqref{suppch2}.  Now we combine the above estimates with \eqref{ellminus1} to obtain 
\begin{equation}\label{ell1}
	  \norm{\comi{y}^{\ell-1}\chi_1\partial_x^m u}_{L^2} 
	 \leq   C \norm{\comi y^\ell f_m}_{L^2}+C\norm{\tilde f_m}_{L^2}+ C\norm{\chi_2\partial_x^m\omega}_{L^2},
\end{equation}
that  yields the  upper bound  for the first term in \eqref{++ell1}.  On the other hand, note that
\begin{eqnarray*}
		\abs{	\frac{\partial_y\omega^s+\partial_y \omega}{\omega^s+\omega}}\leq C\comi y^{-1} \quad \textrm{on supp}\, \chi_1 
\end{eqnarray*}
 because of \eqref{condi},  and then
\begin{eqnarray*}
	\Big\| \comi y^\ell \chi_1 \frac{\partial_y\omega^s+\partial_y \omega}{\omega^s+\omega} \partial_x^mu\Big\|_{L^2} \leq C \norm{\comi{ y}^{\ell-1}\chi_1\partial_x^m u}_{L^2},
\end{eqnarray*}
that along with \eqref{funoffs} and \eqref{ell1} yield
\begin{eqnarray*}
	\norm{\chi_1 \comi y^\ell \partial_x^m\omega}_{L^2}
	 \leq     \norm{\comi y^\ell f_m}_{L^2}+C \norm{\comi{ y}^{\ell-1}\chi_1\partial_x^m u}_{L^2}\leq   C \norm{\comi y^\ell f_m}_{L^2}+C\norm{\tilde f_m}_{L^2}+ C\norm{\chi_2\partial_x^m\omega}_{L^2}.
\end{eqnarray*}
The  upper bound  for the second term in  \eqref{++ell1} follows. We have proven 
\eqref{++ell1}.

It remains to prove the second statement \eqref{higes}.  Note that
\begin{eqnarray*}
	\partial_xf_m=f_{m+1}-\chi_1 \Big[\partial_x\big( \frac{   \partial_y\omega^s+\partial_y \omega }{\omega^s+\omega }\big)\Big]\partial_x^m u,
\end{eqnarray*}
and moreover, direct calculation gives
\begin{eqnarray*}
	\norm{\chi_1 \comi y \partial_x\big( \frac{   \partial_y\omega^s+\partial_y \omega }{\omega^s+\omega }\big)}_{L^\infty}\leq C\norm{  \comi y^{ \alpha}   \partial_x \omega}_{L^\infty}+\norm{  \comi y^{1+\alpha}    \partial_x\partial_y \omega}_{L^\infty}\leq C 
\end{eqnarray*}
because of \eqref{condi} and the fact that $\alpha\leq\ell.$  Thus 
\begin{eqnarray*}
	\norm{\comi y^\ell f_{m+1}}_{L^2}\leq\norm{\comi y^\ell \partial_x f_{m}}_{L^2} +C  \norm{\comi y^{\ell-1}\partial_x^m u}_{L^2}.
	\end{eqnarray*}
Similar estimate holds for $ \norm{\tilde f_{m+1}}_{L^2}.$ 
This estimate and  \eqref{++ell1} give the second statement \eqref{higes} in Lemma \ref{eqnorm}. The proof is then completed. 
\end{proof}

\begin{proof}
[Proof of Proposition \ref{provel+}]

In view of Lemma \ref{eqnorm}  we have
\begin{eqnarray*}
  &&\norm{\chi_1\comi y^{\ell-1} \partial_x^m u}_{L^2}^2+\norm{\chi_1\comi y^{\ell} \partial_x^m\omega}_{L^2}^2 
 \\
 &&\qquad\qquad+\eps \int_0^t \inner{\norm{\chi_1   \comi y^{\ell-1}\partial_x^{m+1}u}_{L^2}^2+\norm{\chi_1   \comi y^{\ell }\partial_x^{m+1}\omega}_{L^2}^2}\,ds \\
 &\leq &\norm{\comi y^\ell f_m}_{L^2}^2+C\norm{\tilde f_m}_{L^2}^2+ C\norm{\chi_2\partial_x^m\omega}_{L^2}^2\\
 &&  + \eps C\int_0^t \inner{\norm{  \comi y^{\ell}\partial_xf_m}_{L^2}^2+\norm{ \partial_x\tilde f_m}_{L^2}^2 + \norm{\chi_2 \partial_x^{m+1}\omega}_{L^2}^2 + \norm{\comi y^{\ell-1} \partial_x^{m}u}_{L^2}^2 }\,ds.\end{eqnarray*}
 Moreover, the terms on the right side of the above
inequality  are bounded above by 
 \begin{eqnarray*}
  \frac{C\big[(m-6)!\big]^{2\sigma}}{\rho^{2(m-5)}}\abs{u_0}_{\rho,\sigma}^2+
 	 \frac{C\big[(m-6)!\big]^{2\sigma}}{\rho^{2(m-5)}} \inner{  \int_0^t   \inner{ \abs{u(s)}_{ \rho,\sigma}^2+\abs{u(s)}_{ \rho,\sigma}^4 }   \,ds+    \int_0^t  \frac{  \abs{u(s)}_{ \tilde\rho,\sigma}^2}{\tilde\rho-\rho}\,ds}
 	 \end{eqnarray*}
by using Proposition \ref{prpaway} and Proposition \ref{prpnear}.   Thus  
\begin{eqnarray*}  
&&\norm{\chi_1\comi y^{\ell-1} \partial_x^m u}_{L^2}^2+\norm{\chi_1\comi y^{\ell} \partial_x^m\omega}_{L^2}^2 
 \\
 &&\qquad\qquad+\eps \int_0^t \inner{\norm{\chi_1   \comi y^{\ell-1}\partial_x^{m+1}u}_{L^2}^2+\norm{\chi_1   \comi y^{\ell }\partial_x^{m+1}\omega}_{L^2}^2}\,ds \\
&\leq &	\frac{C\big[(m-6)!\big]^{2\sigma}}{\rho^{2(m-5)}}\abs{u_0}_{\rho,\sigma}^2+
 	 \frac{C\big[(m-6)!\big]^{2\sigma}}{\rho^{2(m-5)}} \inner{  \int_0^t   \inner{ \abs{u(s)}_{ \rho,\sigma}^2+\abs{u(s)}_{ \rho,\sigma}^4 }   \,ds+    \int_0^t  \frac{  \abs{u(s)}_{ \tilde\rho,\sigma}^2}{\tilde\rho-\rho}\,ds}.  	 	
\end{eqnarray*}
Note that  $1\leq\chi_1+\chi_2$ and $\comi y^{\ell} $ is equivalent to a constant  on  supp\,$\chi_2$. Then  combining the above inequality  and Proposition \ref{prpnear}, we obtain
\begin{equation}\label{chi1esers}
\begin{aligned}
&\norm{\chi_1\comi y^{\ell-1} \partial_x^m u}_{L^2}^2+\norm{ \comi y^{\ell} \partial_x^m\omega}_{L^2}^2 +\eps \int_0^t \inner{\norm{\chi_1   \comi y^{\ell-1}\partial_x^{m+1}u}_{L^2}^2+\norm{ \comi y^{\ell }\partial_x^{m+1}\omega}_{L^2}^2}\,ds \\
\leq &	\frac{C\big[(m-6)!\big]^{2\sigma}}{\rho^{2(m-5)}}\abs{u_0}_{\rho,\sigma}^2+
 	 \frac{C\big[(m-6)!\big]^{2\sigma}}{\rho^{2(m-5)}} \inner{  \int_0^t   \inner{ \abs{u(s)}_{ \rho,\sigma}^2+\abs{u(s)}_{ \rho,\sigma}^4 }   \,ds+    \int_0^t  \frac{  \abs{u(s)}_{ \tilde\rho,\sigma}^2}{\tilde\rho-\rho}\,ds}.  	 	
\end{aligned}
\end{equation}
Moreover,  by Poincar\'e inequality we have, for $ j=m$ and $j=m+1,$
\begin{eqnarray*}
	\norm{\chi_2\comi y^{\ell-1}\partial_x^j u}_{L^2}\leq C\norm{\chi_2\partial_x^j u}_{L^2}\leq C\norm{\partial_y\inner{\chi_2\partial_x^ju}}_{L^2} \leq   C\norm{  \chi_1 \partial_x^j u}_{L^2}+C\norm{   \partial_x^j \omega}_{L^2},  
\end{eqnarray*}
when in the last inequality we have used the fact
that $\chi_2'=\chi_1\chi_2'$ by \eqref{suppch2}. This and \eqref{chi1esers}      give
\begin{eqnarray*}
	&&\norm{\chi_2\comi y^{\ell-1} \partial_x^m u}_{L^2}^2 +\eps \int_0^t  \norm{\chi_2   \comi y^{\ell-1}\partial_x^{m+1}u}_{L^2}^2 \,ds \\
&\leq &	\frac{C\big[(m-6)!\big]^{2\sigma}}{\rho^{2(m-5)}}\abs{u_0}_{\rho,\sigma}^2+
 	 \frac{C\big[(m-6)!\big]^{2\sigma}}{\rho^{2(m-5)}} \inner{  \int_0^t   \inner{ \abs{u(s)}_{ \rho,\sigma}^2+\abs{u(s)}_{ \rho,\sigma}^4 }   \,ds+    \int_0^t  \frac{  \abs{u(s)}_{ \tilde\rho,\sigma}^2}{\tilde\rho-\rho}\,ds}.
\end{eqnarray*}
Consequently, we combine the above inequality and \eqref{chi1esers} to conclude, by using again the fact that  $1\leq\chi_1+\chi_2,$
\begin{eqnarray*}
  &&\norm{ \comi y^{\ell-1} \partial_x^m u}_{L^2}^2+\norm{ \comi y^{\ell} \partial_x^m\omega}_{L^2}^2 
 +\eps \int_0^t \inner{\norm{    \comi y^{\ell-1}\partial_x^{m+1}u}_{L^2}^2+\norm{   \comi y^{\ell }\partial_x^{m+1}\omega}_{L^2}^2}\,ds \\
   &\leq & 	\frac{C\big[(m-6)!\big]^{2\sigma}}{\rho^{2(m-5)}}\abs{u_0}_{\rho,\sigma}^2+
 	 \frac{C\big[(m-6)!\big]^{2\sigma}}{\rho^{2(m-5)}} \inner{  \int_0^t   \inner{ \abs{u(s)}_{ \rho,\sigma}^2+\abs{u(s)}_{ \rho,\sigma}^4 }   \,ds+    \int_0^t  \frac{  \abs{u(s)}_{ \tilde\rho,\sigma}^2}{\tilde\rho-\rho}\,ds}.  
\end{eqnarray*}
  Thus we get the desired estimate in Proposition \ref{provel+}
and this completes the proof. 
\end{proof}

 \subsection{Estimate on the mixed derivatives}   
For the mixed derivatives $\partial_x^i\partial_y^j\omega$ of vorticity, we have
  \begin{proposition}
 \label{norma}
 Let $u\in L^\infty\inner{[0, T];~X_{\rho_0,\sigma}}$  be a solution to \eqref{regpran} 
under the assumptions in Theorem \ref{uniestgev}. Then we have, for any pair $(i,j)$ with $1\leq j\leq 4$ and $ i+j\geq 6,$ for any $t\in[0,T],$ and   for any   $\rho>0,$
\begin{eqnarray*}
 \norm{ \comi{y}^{\ell+1} \partial_x^{i}\partial_y^{j}\omega(t)}_{L^2 }^2 
 \leq    \frac{C\big[(i+j-6)!\big]^{2\sigma}}{\rho^{2(i+j-5)}}\abs{u_0}_{\rho,\sigma}^2+
 	 \frac{C\big[(i+j-6)!\big]^{2\sigma}}{\rho^{2(i+j-5)}}  \int_0^t   \inner{ \abs{u(s)}_{ \rho,\sigma}^2+\abs{u(s)}_{ \rho,\sigma}^4 }   \,ds. 
\end{eqnarray*}
\end{proposition}

The proof of the above proposition can be obtained
by  similar  argument used  in the previous sections,  and the main difference  arises  from the boundary values since higher derivatives in $y$ are involved when we perform integration by parts.   So we first calculate $\partial_y^j\omega\big|_{y=0}.$ Firstly, we have     
\begin{eqnarray*}
\partial_y\omega^s(t,0)=\partial_y\omega(t,x,0)=0.	
\end{eqnarray*}
 Then  by the equation of vorticity, we  obtain that
\begin{eqnarray*}
	\partial_y^3\omega\big |_{y=0}= \partial_y \big(\partial_t    \omega+\inner{u^s+u}\partial_x\omega+v\partial_y \inner{\omega^s+\omega} -\eps\partial_x^2\omega \big)  \big |_{y=0} 
	 = \inner{\omega^s+\omega}\partial_x\omega\big |_{y=0},
\end{eqnarray*}
and    direct computation yields 
\begin{equation}
\label{bouval}
 \partial_y^5\omega\big |_{y=0} 
	 =- \inner{\partial_y^2\omega^s+\partial_y^2\omega} \partial_x\omega \big |_{y=0}+4\inner{\omega^s+\omega} \partial_x\partial_y^2\omega \big |_{y=0}-2\eps  \inner{\partial_x\omega}\partial_x^2\omega \big |_{y=0}.
 \end{equation}
 We  apply $\comi{y}^{\ell+1} \partial_x^{i}\partial_y^{j}$ to the equation   for  vorticity $\omega$ to  have  \begin{eqnarray}\label{equforhi}
	&&	\Big(\partial_t    +\inner{u^s+u}\partial_x+v\partial_y  -\partial_y^2-\eps\partial_x^2 \Big) \comi{y}^{\ell+1}\partial_x^{i}\partial_y^{j} \omega=A_{i,j}
	\end{eqnarray}
	with
	\begin{eqnarray*}
	A_{i, j} &=& -  \comi{y}^{\ell+1}\partial_x^{i}\partial_y^{j}\big(v\partial_y\omega^s\big)-\comi{y}^{\ell+1} \sum_{\stackrel {k+q\geq 1}{k\leq i, ~q\leq j}} {i\choose k}{ j \choose q}\inner{\partial_x^k\partial_y^q (u^s +u)} \partial_{x}^{i-k+1}\partial_y^{j-q}   \omega\\
	&&-\comi{y}^{\ell+1}\sum_{\stackrel {k+q\geq 1}{k\leq i, ~q\leq j}} {i\choose k}{ j \choose q } \inner{\partial_x^k\partial_y^q v}\partial_x^{i-k}\partial_y^{j-q+1}  \omega\\ 
	&&+v\inner{\partial_y\comi{y}^{\ell+1}}\partial_x^i\partial_y^j\omega-\inner{\partial_y^2\comi{y}^{\ell+1}}\partial_x^i\partial_y^j\omega-2\inner{\partial_y\comi{y}^{\ell+1}}\partial_x^i\partial_y^{j+1}\omega\\
	&\stackrel{\rm def}{=} & A_{i,j,1}+ A_{i,j,2}+ A_{i,j,3}+ A_{i,j,4}+ A_{i,j,5}+ A_{i,j,6}. 
	\end{eqnarray*}
	We will estimate the terms on both sides of \eqref{equforhi}
 in the following lemmas.

 	\begin{lemma}\label{lemboud}
	 Let $1\leq j\leq 4$ and $ i+j\geq 6.$ Then we have
	\begin{eqnarray*}
		&&\int_0^t \Big( \big(\partial_t    +\inner{u^s+u}\partial_x+v\partial_y  -\partial_y^2-\eps\partial_x^2 \big) \comi{y}^{\ell+1}\partial_x^{i}\partial_y^{j} \omega,~\comi{y}^{\ell+1}\partial_x^{i}\partial_y^{j} \omega\Big)_{L^2}\,ds\\
		&\geq&\frac{1}{2} \norm{\comi{y}^{\ell+1}\partial_x^{i}\partial_y^{j} \omega}_{L^2}^2+\frac{1}{2} \int_0^t \norm{ \comi{y}^{\ell+1}\partial_x^{i}\partial_y^{j+1} \omega }_{L^2}^2-\frac{1}{2} \norm{\comi{y}^{\ell+1}\partial_x^{i}\partial_y^{j} \omega(0)}_{L^2}^2\\
		&&-\frac{ C \com{\big((i+j)-6\big)!}^{2\sigma}}{\rho^{2\big( (i+j)-5\big)}} \int_0^t  \inner{\abs{u(s)}_{\rho,\sigma}^2 + \abs{u(s)}_{\rho,\sigma}^4}\,ds.
	\end{eqnarray*}	
	\end{lemma}
 
 \begin{proof} We only need to discuss the boundary terms when
we us integration by parts:
 \begin{eqnarray*}
 	&&\int_0^t \Big(-\partial_y^2 \inner{\comi{y}^{\ell+1}\partial_x^{i}\partial_y^{j} \omega},~\comi{y}^{\ell+1}\partial_x^{i}\partial_y^{j} \omega\Big)_{L^2}\,ds\\
 	&=&\int_0^t \norm{\partial_y  \big(\comi{y}^{\ell+1}\partial_x^{i}\partial_y^{j} \omega\big)}_{L^2}^2 \,ds+\int_0^t	\int_{\mathbb R}  \partial_y  \big(\comi{y}^{\ell+1}\partial_x^{i}\partial_y^{j} \omega\big) \Big|_{y=0}\partial_x^{i}\partial_y^{j} \omega(t,x,0)\,dxdt\\ 
 	&\geq&\int_0^t \norm{ \comi{y}^{\ell+1}\partial_x^{i}\partial_y^{j+1} \omega }_{L^2}^2 \,ds+\int_0^t	\int_{\mathbb R}  \Big(\partial_x^{i}\partial_y^{j} \omega(t,x,0)\Big)\partial_x^{i}\partial_y^{j+1} \omega(t,x,0)\,dxds\\
 	&&-\frac{ C \com{\big((i+j)-6\big)!}^{2\sigma}}{\rho^{2\big( (i+j)-5\big)}} \int_0^t  \abs{u(s)}_{\rho,\sigma}^2 \,ds,
 \end{eqnarray*} 
where we have used  \eqref{emix} in the last inequality. Note
that the boundary value is well-defined   in view of \eqref{bouval}.  
 Thus  the  estimate in Lemma \ref{lemboud}  follows by standard energy method if we can show that, for any $\kappa>0$,  $1\leq j\leq 4$ and $i\geq 0$ with $i+j\geq 6$  that
   \begin{equation}\label{inesv}
   \begin{aligned}
  &\Big| \int_0^t	\int_{\mathbb R} \Big(\partial_x^{i}\partial_y^{j} \omega(s,x,0)\Big) \partial_x^{i}\partial_y^{j+1} \omega(s,x,0)\,dxds \Big|\\
  \leq & \kappa   \int_0^t \norm{ \partial_x^i\partial_y^{j+1}\omega}_{L^2}^2\,ds 
	 + \frac{ C\kappa^{-1} \com{\big((i+j)-6\big)!}^{2\sigma}}{\rho^{2\big( (i+j)-5\big)}} \int_0^t  \abs{u(s)}_{\rho,\sigma}^4  \,ds.
	 \end{aligned}
   \end{equation}
  Note that \eqref{inesv} holds obviously  for   $j=1$ because $\partial_y\omega(t,x,0)=0.$ It remains to consider the cases when $j=2, 3, 4.$
  
  {\it \underline{The case when $j=4$}}: Recall, in view of \eqref{bouval},   
\begin{eqnarray*}
	\partial_y^5\omega\big |_{y=0}- \inner{\partial_y^2\omega^s+\partial_y^2\omega} \partial_x\omega \big |_{y=0}+4\inner{\omega^s+\omega} \partial_x\partial_y^2\omega \big |_{y=0}-2\eps  \inner{\partial_x\omega}\partial_x^2\omega \big |_{y=0}.  
\end{eqnarray*}
Then direct computation gives, using the argument in Lemma \ref{lemp3} as well as the Sobolev inequality (see Lemma \ref{sobine} in the Appendix),
\begin{eqnarray*}
	\norm{\partial_x^{i}\partial_y^{5} \omega(t,x,0)}_{ L_x^2}\leq  \frac{ C \com{\big((i+4)-6\big)!}^{\sigma}}{\rho^{ (i+4)-5}} \abs{u}_{\rho,\sigma}.
\end{eqnarray*}
Hence,  for any small $\kappa>0,$
\begin{eqnarray*}
	&& \Big| \int_0^t	\int_{\mathbb R} \Big(\partial_x^{i}\partial_y^{4} \omega(s,x,0) \Big)\partial_x^{i}\partial_y^{5} \omega(s,x,0)\,dxds \Big|\\
	&\leq& \kappa \int_0^t \norm{\partial_x^i\partial_y^4\omega(s,x,0)}_{ L_x^2}^2\,ds 
	 +\kappa^{-1}C \int_0^t\norm{\partial_x^{i}\partial_y^{5} \omega(s,x,0)}_{ L_x^2}^2 \,ds\\
	&\leq& \kappa C  \int_0^t \norm{ \partial_x^i\partial_y^5\omega}_{L^2}^2\,ds 
	 + \frac{ C\kappa^{-1} \com{\big((i+4)-6\big)!}^{2\sigma}}{\rho^{2\big( (i+4)-5\big)}} \int_0^t  \abs{u(s)}_{\rho,\sigma}^4  \,ds.
\end{eqnarray*}
  Thus we obtain  \eqref{inesv} for $j=4.$

  {\it \underline{The case when $2\leq j\leq 3$}}:
The estimation on 
\begin{eqnarray*}
	\Big| \int_0^t	\int_{\mathbb R} \Big(\partial_x^{i}\partial_y^{j} \omega(s,x,0) \Big)\partial_x^{i}\partial_y^{j+1} \omega(s,x,0)\,dxds \Big|
\end{eqnarray*}
for  $j=2$ and $j=3$ is simpler than the case $j=4,$ since only lower order derivatives are involved.  And thus for brevity we omit the details.   The proof is then completed. 
 \end{proof}

\begin{lemma}[Estimate on $A_{i,j,1}$]
	Under the same assumption as  in Proposition \ref{norma}, we have 	\begin{eqnarray*}
	 	\Big|\int_0^t \Big(-  \comi{y}^{\ell+1}\partial_x^{i}\partial_y^{j}\big(v\partial_y\omega^s\big),~\comi{y}^{\ell+1}\partial_x^{i}\partial_y^{j} \omega\Big)_{L^2}\,ds\Big|\leq 
	 \frac{  C \com{(i+j-6)!}^{2\sigma}}{\rho^{2(i+j-5)}} \int_0^t \inner{\abs{u(s)}_{\rho,\sigma}^2 +\abs{u(s)}_{\rho,\sigma}^3}\,ds.
	\end{eqnarray*}
\end{lemma}

\begin{proof}
	Using $\partial_y v=-\partial_xu,$ we have
	\begin{eqnarray*}
		&&\Big|\int_0^t \Big(-  \comi{y}^{\ell+1}\partial_x^{i}\partial_y^{j}\big(v\partial_y\omega^s\big),~\comi{y}^{\ell+1}\partial_x^{i}\partial_y^{j} \omega\Big)_{L^2}\,ds\Big|\\
		&\leq & \Big|\int_0^t \Big(   \comi{y}^{\ell+1}\big(\partial_x^{i}v\big) \partial_y^{j+1}\omega^s,~\comi{y}^{\ell+1}\partial_x^{i}\partial_y^{j} \omega\Big)_{L^2}\,ds\Big|\\
		&&+
		 \Big|\int_0^t \Big(   \sum_{q=1}^j\comi{y}^{\ell+1}\big(\partial_x^{i+1}\partial_y^{q-1}u\big) \partial_y^{j-q+1}\omega^s,~\comi{y}^{\ell+1}\partial_x^{i}\partial_y^{j} \omega\Big)_{L^2}\,ds\Big|.
	\end{eqnarray*}
	Moreover, note that for $1\leq j\leq 4$ we have $\comi y^{\ell+1}\comi y^{-\alpha-j-1}\in L^2(\mathbb R_+)$ because of $\ell<\alpha+1/2,$  and thus
	\begin{eqnarray*}
		 \norm{ \comi{y}^{\ell+1}\partial_y^{j+1}\omega^s}_{L^2(\mathbb R_+)}\leq C.
	\end{eqnarray*}
	Consequently, 
	\begin{eqnarray*}
		  \Big|\int_0^t \Big(   \comi{y}^{\ell+1}\big(\partial_x^{i}v\big) \partial_y^{j+1}\omega^s,~\comi{y}^{\ell+1}\partial_x^{i}\partial_y^{j} \omega\Big)_{L^2}\,ds\Big| 
		&\leq &\frac{  C \com{(i+1-6)!}^{\sigma}}{\rho^{i+1-5}}\frac{    \com{(i+j-6)!}^{\sigma}}{\rho^{i+j-5}} \int_0^t  \abs{u}_{\rho,\sigma}^2\,ds\\
		&\leq & \frac{C    \com{(i+j-6)!}^{2\sigma}}{\rho^{2(i+j-5)}} \int_0^t  \abs{u}_{\rho,\sigma}^2\,ds.
	\end{eqnarray*}
	Direct calculation also shows
	\begin{eqnarray*}
		\Big|\int_0^t \Big(   \sum_{q=1}^j\comi{y}^{\ell+1}\big(\partial_x^{i+1}\partial_y^{q-1}u\big) \partial_y^{j-q+1}\omega^s,~\comi{y}^{\ell+1}\partial_x^{i}\partial_y^{j} \omega\Big)_{L^2}\,ds\Big| \leq\frac{C    \com{(i+j-6)!}^{2\sigma}}{\rho^{2(i+j-5)}} \int_0^t  \abs{u}_{\rho,\sigma}^3\,ds.
	\end{eqnarray*}
	Then the desired estimate follows and we complete the proof.
\end{proof}

\begin{lemma}[Estimate on $A_{i,j,3}$ and $A_{i,j,6}$]
	Under the same assumption as Proposition \ref{norma}, we have, for any $\kappa>0,$
	\begin{eqnarray*}
	&&	\Big|\int_0^t \Big(A_{i,j,3},~\comi{y}^{\ell+1}\partial_x^{i}\partial_y^{j} \omega\Big)_{L^2}\,ds\Big|+\Big|\int_0^t\Big(A_{i,j,6},~\comi{y}^{\ell+1}\partial_x^{i}\partial_y^{j} \omega\Big)_{L^2}\,ds\Big|\\ 
	&\leq&  \kappa \int_0^t\norm{ \comi{y}^{\ell+1}\partial_x^{i}\partial_y^{j+1} \omega }_{L^2}^2\,ds+\frac{C\kappa^{-1}   \com{(i+j-6)!}^{2\sigma}}{\rho^{2(i+j-5)}} \int_0^t \inner{\abs{u(s)}_{\rho,\sigma}^2+\abs{u(s)}_{\rho,\sigma}^4}\,ds.
	\end{eqnarray*}
\end{lemma}

\begin{proof}
	We decompose  $A_{i,j,3}$ as follows by using  $\partial_y v=-\partial_x u,$
	\begin{eqnarray*}
	A_{m,j,3} &=&  -\comi{y}^{\ell+1}\sum_{ 1\leq k\leq i} {i\choose k}  \inner{\partial_x^k v} \partial_x^{i-k}\partial_y^{j+1}  \omega \\
	&&+\comi{y}^{\ell+1}\sum_{\stackrel { q\geq 1}{k\leq i, ~q\leq j}} {i\choose k}{ j \choose q } \inner{\partial_x^{k+1}\partial_y^{q-1} u} \partial_x^{i-k}\partial_y^{j-q+1}  \omega. 
	\end{eqnarray*}
Following the similar argument as in Lemma \ref{lemp3}, we have
\begin{equation}\label{aij1}
\begin{aligned}
	&\Big|\int_0^t \Big(\comi{y}^{\ell+1}\sum_{\stackrel { q\geq 1}{k\leq i, ~q\leq j}} {i\choose k}{ j \choose q } \inner{\partial_x^{k+1}\partial_y^{q-1} u} \partial_x^{i-k}\partial_y^{j-q+1}  \omega,~\comi{y}^{\ell+1}\partial_x^{i}\partial_y^{j} \omega\Big)_{L^2}\,ds\Big|	\\
	\leq & \frac{ C \com{(i+j-6)!}^{2\sigma}}{\rho^{2(i+j-5)}} \int_0^t  \abs{u}_{\rho,\sigma}^3\,ds.
	\end{aligned}
\end{equation}
Next, we will prove that, for any small $\kappa>0,$
\begin{equation}\label{aij2}
\begin{aligned}
	&\Big|\int_0^t \Big( -\comi{y}^{\ell+1}\sum_{ 1\leq k\leq i} {i\choose k}  \inner{\partial_x^k v} \partial_x^{i-k}\partial_y^{j+1}  \omega,~\comi{y}^{\ell+1}\partial_x^{i}\partial_y^{j} \omega\Big)_{L^2}\,ds\Big|	\\
	\leq &  \kappa \int_0^t\norm{ \comi{y}^{\ell+1}\partial_x^{i}\partial_y^{j+1} \omega}_{L^2}^2\,ds+\frac{C\kappa^{-1}   \com{(i+j-6)!}^{2\sigma}}{\rho^{2(i+j-5)}} \int_0^t \inner{\abs{u}_{\rho,\sigma}^3+\abs{u}_{\rho,\sigma}^4}ds.
	\end{aligned}
\end{equation}
To do so,  integration by parts gives
	\begin{equation}\label{eqlina}
	\begin{aligned}
&\Big|\int_0^t \Big( -\comi{y}^{\ell+1}\sum_{ 1\leq k\leq i} {i\choose k}  \inner{\partial_x^k v} \partial_x^{i-k}\partial_y^{j+1} \omega,~\comi{y}^{\ell+1}\partial_x^{i}\partial_y^{j} \omega\Big)_{L^2}\,ds\Big|\\
\leq& \Big|\int_0^t \Big(\comi{y}^{\ell+1}\sum_{ 1\leq k\leq i} {i\choose k}  \inner{\partial_x^k v} \partial_x^{i-k}\partial_y^{j} \omega,~ \comi{y}^{\ell+1}\partial_x^{i}\partial_y^{j+1} \omega \Big)_{L^2}\, ds\Big|\\
& +\Big|\int_0^t\Big(\comi{y}^{\ell+1}\sum_{ 1\leq k\leq i} {i\choose k}  \inner{\partial_x^{k+1} u} \partial_x^{i-k}\partial_y^{j} \omega,~\comi{y}^{\ell+1}\partial_x^{i}\partial_y^{j} \omega\Big)_{L^2}ds\Big|\\
&+\Big|\int_0^t\Big(\big(\partial_y\comi{y}^{2\ell+2}\big)\sum_{ 1\leq k\leq i} {i\choose k}  \inner{\partial_x^k v} \partial_x^{i-k}\partial_y^{j } \omega,~ \partial_x^{i}\partial_y^{j} \omega\Big)_{L^2}ds\Big|.
\end{aligned}
\end{equation} 
Moreover, as   in Lemma \ref{lemp3} and Lemma \ref{lemuppb}, we can prove that the first term on the right side of \eqref{eqlina} is bounded  above by
\begin{eqnarray*}
	\kappa \int_0^t \norm{ \comi{y}^{\ell+1}\partial_x^{i}\partial_y^{j+1} \omega}_{L^2}^2\,ds+\frac{C\kappa^{-1}  \com{(i+j-6)!}^{2\sigma}}{\rho^{2(i+j-5)}} \int_0^t \abs{u}_{\rho,\sigma}^4\,ds,
\end{eqnarray*}
and the last two terms are  bounded above by
\begin{eqnarray*}
	\frac{  C \com{(i+j-6)!}^{2\sigma}}{\rho^{2(i+j-5)}} \int_0^t \abs{u}_{\rho,\sigma}^3\,ds.
	\end{eqnarray*}
Thus combining the above estimate, \eqref{aij2} follows. This and \eqref{aij1} give
\begin{eqnarray*}
	 &&	\Big|\int_0^t \Big(A_{i,j,3},~\comi{y}^{\ell+1}\partial_x^{i}\partial_y^{j} \omega\Big)_{L^2}\,ds\Big| 	\\
	 &\leq&  \kappa \int_0^t\norm{ \comi{y}^{\ell+1}\partial_x^{i}\partial_y^{j+1} \omega }_{L^2}^2\,ds+\frac{C\kappa^{-1}   \com{(i+j-6)!}^{2\sigma}}{\rho^{2(i+j-5)}} \int_0^t \inner{\abs{u}_{\rho,\sigma}^3+\abs{u}_{\rho,\sigma}^4}\,ds.
\end{eqnarray*}
Similarly,
\begin{eqnarray*}
	 \Big|\int_0^t \Big(A_{i,j,6},~\comi{y}^{\ell+1}\partial_x^{i}\partial_y^{j} \omega\Big)_{L^2}\,ds\Big| 	 
	  \leq  \kappa \int_0^t\norm{ \comi{y}^{\ell+1}\partial_x^{i}\partial_y^{j+1} \omega }_{L^2}^2\,ds+\frac{C\kappa^{-1}   \com{(i+j-6)!}^{2\sigma}}{\rho^{2(i+j-5)}} \int_0^t  \abs{u}_{\rho,\sigma}^2\,ds.
\end{eqnarray*}
Thus the proof is completed. 	
	\end{proof}
	 
\begin{lemma} 	\label{lem67}
 Let $1\leq j\leq 4$ and $ i+j\geq 6.$ Then we have
 	\begin{eqnarray*}
		\Big|\int_0^t\Big(A_{i,j,2}+A_{i,j,4}+A_{i,j,5},~\comi{y}^{\ell+1}\partial_x^{i}\partial_y^{j} \omega\Big)_{L^2}\,ds\Big| 
	\leq  \frac{ C \com{(i+j-6)!}^{2\sigma}}{\rho^{2(i+j-5)}}  \int_0^t\inner{\abs{u(s)}_{\rho,\sigma}^2 +\abs{u(s)}_{\rho,\sigma}^3}\,ds.
	\end{eqnarray*}
\end{lemma}

	\begin{proof}
	   Since there is no $j+1$ order derivative in $y$  involved,
the proof is straightforward so that we omit the detail
for brevity. 
	\end{proof}
	
	\begin{proof}
	[Completion of the proof of Proposition \ref{norma}]
 Let $1\leq j\leq 4$ and $ i+j\geq 6.$  In view of \eqref{equforhi},
	we combine the estimates in Lemmas \ref{lemboud}-\ref{lem67} to conclude by  choosing $\kappa$  sufficiently small that
	\begin{eqnarray*}
		&&\frac{1}{2} \norm{\comi{y}^{\ell+1}\partial_x^{i}\partial_y^{j} \omega(t)}_{L^2}^2+\frac{1}{4} \int_0^t \norm{ \comi{y}^{\ell+1}\partial_x^{i}\partial_y^{j+1} \omega }_{L^2}^2\\
		&\leq &\frac{1}{2} \norm{\comi{y}^{\ell+1}\partial_x^{i}\partial_y^{j} \omega(0)}_{L^2}^2 +\frac{  C \com{(i+j-6)!}^{2\sigma}}{\rho^{2(i+j-5)}} \int_0^t \inner{\abs{u}_{\rho,\sigma}^2 +\abs{u}_{\rho,\sigma}^4}\,ds.
	\end{eqnarray*}	
	Note that
	\begin{eqnarray*}
		 \norm{\comi{y}^{\ell+1}\partial_x^{i}\partial_y^{j} \omega(0)}_{L^2}^2\leq \frac{  C \com{(i+j-6)!}^{2\sigma}}{\rho^{2(i+j-5)}} \abs{u_0}_{\rho,\sigma}^2.
	\end{eqnarray*}
	Then the desired estimate in Proposition \ref{norma} follows. The proof is then completed. 
	\end{proof}

\subsection{The proof of Theorem  \ref{uniestgev}}  
By Proposition \ref{provel+} and Proposition \ref{norma}, we have
\begin{eqnarray*}
	&&\bigg[\sup_{m\geq 6} \frac{\rho^{ m-5}} { \big[(m-6)!\big]^{ \sigma}}\inner{\norm{\comi y^{\ell-1} \partial_x^m u(t)}_{L^2}+\norm{\comi y^{\ell} \partial_x^m \omega(t)}_{L^2}}\bigg]^2\\
	&&\qquad\qquad+\bigg[\sup_{\stackrel{1\leq j\leq 4}{ i+j\geq 6}} \frac{\rho^{ i+j-5}} { \big[(i+j-6)!\big]^{ \sigma}} \norm{\comi y^{\ell+1} \partial_x^i\partial_y^j \omega(t)}_{L^2} \bigg]^2\\
	&\leq & C\abs{u_0}_{\rho,\sigma}^2+C
 	  \int_0^t   \inner{ \abs{u(s)}_{ \rho,\sigma}^2+\abs{u(s)}_{ \rho,\sigma}^4 }   \,ds+C
 	  \int_0^t   \frac{\abs{u(s)}_{\tilde \rho,\sigma}^2}{\tilde\rho-\rho}   \,ds.
\end{eqnarray*}
From Proposition \ref{prpaway} and  Proposition \ref{hm},
it follows that
\begin{eqnarray*}
	&&\bigg[\sup_{m\geq 6} \frac{\rho^{ m-5}} { \big[(m-6)!\big]^{ \sigma}}\inner{\norm{\comi y^{\ell} f_m(t)}_{L^2}+\norm{h_m(t)}_{L^2}+\norm{\chi_2\partial_y\partial_x^m\omega (t)}_{L^2}}\bigg]^2\\
	&\leq & C\abs{u_0}_{\rho,\sigma}^2+C
 	  \int_0^t   \inner{ \abs{u(s)}_{ \rho,\sigma}^2+\abs{u(s)}_{ \rho,\sigma}^4 }   \,ds+C
 	  \int_0^t   \frac{\abs{u(s)}_{\tilde \rho,\sigma}^2}{\tilde\rho-\rho}   \,ds.
\end{eqnarray*}
Moreover, we combine the first estimate in Corollary \ref{cordiff} and Proposition \ref{provel+} to have 
\begin{eqnarray*}
 \bigg[\sup_{m\geq 6} \frac{\rho^{ m-5} } { \big[(m-6)!\big]^{ \sigma}}\inner{m\norm{g_m(t)}_{L^2}}  \bigg]^2 \leq   C\abs{u_0}_{\rho,\sigma}^2+C
 	  \int_0^t   \inner{ \abs{u(s)}_{ \rho,\sigma}^2+\abs{u(s)}_{ \rho,\sigma}^4 }   \,ds+C
 	  \int_0^t   \frac{\abs{u(s)}_{\tilde \rho,\sigma}^2}{\tilde\rho-\rho}   \,ds.
\end{eqnarray*}
Finally, direct computation gives 
\begin{eqnarray*}
	&&\bigg[\sup_{m\leq 5}  \inner{\norm{\comi y^{\ell-1} \partial_x^m u(t)}_{L^2}+\norm{\comi y^{\ell} \partial_x^m \omega(t)}_{L^2}}\bigg]^2 +\bigg[\sup_{\stackrel{1\leq j\leq 4}{ i+j\leq 5}}  \norm{\comi y^{\ell+1} \partial_x^i\partial_y^j \omega(t)}_{L^2} \bigg]^2\\
	&&\qquad+\bigg[\sup_{1\leq m\leq 5}  \inner{m\norm{g_m(t)}_{L^2}+\norm{\comi y^{\ell} f_m(t)}_{L^2}+\norm{h_m(t)}_{L^2}+\norm{\chi_2\partial_y\partial_x^m\omega (t)}_{L^2}}\bigg]^2\\
	&\leq & C\abs{u_0}_{\rho,\sigma}^2+C
 	  \int_0^t   \inner{ \abs{u(s)}_{ \rho,\sigma}^2+\abs{u(s)}_{ \rho,\sigma}^4 }   \,ds+C
 	  \int_0^t   \frac{\abs{u(s)}_{\tilde \rho,\sigma}^2}{\tilde\rho-\rho}   \,ds.
\end{eqnarray*}
Combining these inequalities yields 
\begin{eqnarray*}
\abs{u(t)}_{\rho,\sigma}^2\leq 	C\abs{u_0}_{\rho,\sigma}^2+C
 	  \int_0^t   \inner{ \abs{u(s)}_{ \rho,\sigma}^2+\abs{u(s)}_{ \rho,\sigma}^4 }   \,ds+C
 	  \int_0^t   \frac{\abs{u(s)}_{\tilde \rho,\sigma}^2}{\tilde\rho-\rho}   \,ds.
\end{eqnarray*}
This completes the  proof of Theorem  \ref{uniestgev}.

\section{Existence  for the  regularized Prandtl equation}\label{sec6}

 In this section, we study the existence of the regularized Prandtl equation introduced in Section \ref{sec2}: 
 \begin{equation}\label{newreg}
 	\left\{
	\begin{aligned}
	&\partial_t     u_\eps +\inner{u^{s}+   u_\eps }\partial_x  u_\eps  +  v_\eps \partial_y \inner{  u^s+u _\eps}-\partial_y^2    u_\eps -\eps\partial_x^2 u_\eps =0,\\
	&   u_\eps \big|_{y=0}=0,\quad\lim_{y\rightarrow +\infty}    u_\eps =0,\\
	&   u_\eps\big|_{t=0}=  u_{0},
\end{aligned}	
\right.
 \end{equation}
 with $
  v_\eps=- \int_0^y\partial_x u_\eps(x,\tilde y)\,d\tilde y.$ This is a nonlinear parabolic equation.  The main result can be stated as follows. 
   
\begin{theorem}[Existence for the regularized Prandtl equation] \label{existhm}
Let   $\rho_0>0,\sigma\geq1$ be two given constants.  Suppose the initial datum $ u_0\in X_{2\rho_0,\sigma}$   satisfies the compatibility condition \eqref{comcon}.     Then there exists   $T_\eps^*>0$,  such that
the regularized Prandtl equation \eqref{newreg} admits a unique  solution $u_\eps\in L^\infty\big([0,T_\eps^*];X_{3\rho_0/2, \sigma}\big).$  
\end{theorem}

\begin{proof}[Sketch of the proof of Theorem \ref{existhm}] We will use iteration to prove the existence.  Since \eqref{newreg} is a parabolic equation,  then we can apply the standard energy estimate in Gevrey norms.

{\it Step (i)}. We first choose $u_j, j\geq 0, $ as follows.  Let $u_0$ be the initial datum in \eqref{newreg} and let $u_j$ be the solution to the linear parabolic equation 
\begin{eqnarray*}
\left\{
\begin{aligned}
	&\partial_t     u_j -\partial_y^2    u_j -\eps\partial_x^2 u_j =-\inner{u^{s}+   u_{j-1} }\partial_x  u_{j-1}  -  v_{j-1} \partial_y \inner{  u^s+u _{j-1}},\\
	&   u_j\big|_{y=0}=0,\quad\lim_{y\rightarrow +\infty}    u_j =0,\\
	&   u_j\big|_{t=0}=  u_{0},
	\end{aligned}
	\right.
\end{eqnarray*}
where $  v_j=- \int_0^y\partial_x u_j(x,\tilde y)\,d\tilde y.$ Note that the existence of solutions to the above linear initial-boundary problem   is guaranteed  
by using the heat kernel 
\begin{eqnarray*}
E_0 (t,x,y)=\frac{1}{4\pi t\sqrt{\eps}}e^{-x^2/(4t\eps)}e^{-y^2/(4t)}.
\end{eqnarray*}
Indeed,  define  two heat operators   $M_1$ and $M_2$ by
\begin{eqnarray*}
M_1(t)  f &=&\int_{\mathbb R} \, d\tilde x \int_{0}^{+\infty} d\tilde y\Big[E_0(t,x-\tilde x, y-\tilde y )-E_0(t,x-\tilde x, y+\tilde y )\Big]f(\tilde x,\tilde y),
\\
M_2  f&=&\int_0^t M_1(t-s)f(s)ds.  
\end{eqnarray*}
then we have 
\begin{eqnarray*}
	u_j=M_1u_0-M_2 \Big(\inner{u^{s}+   u_{j-1}}\partial_x  u_{j-1}  +  v_{j-1} \partial_y \inner{  u^s+u_{j-1} }\Big).
	\end{eqnarray*}
 {\it Step (ii)}. Now we consider the difference
 \begin{eqnarray*}
 \xi_0=u_0,\quad 	\xi_j=u_{j}-u_{j-1},\quad \zeta_j=v_{j}-v_{j-1},\quad j\geq 1.
 \end{eqnarray*}
 Then 
 \begin{equation}\label{eqxi1}
 	\partial_t     \xi_1 -\partial_y^2    \xi_1 -\eps\partial_x^2 \xi_1 = \partial_y^2 u_0+\eps\partial_x^2u_0-\inner{u^{s}+   u_{0} }\partial_x  u_{0}  -  v_{0} \partial_y \inner{  u^s+u _{0}},
 \end{equation}
 and for $j\geq 2$ we have
 \begin{equation}\label{eqxi}
 	\partial_t     \xi_j -\partial_y^2    \xi_j -\eps\partial_x^2 \xi_j =  -\inner{u^{s}+   u_{j-1} }\partial_x  \xi_{j-1}- \xi_j   \partial_x  u_{j-2}   -  \zeta_{j-1} \partial_y \inner{  u^s+u _{j-1}}-v_{j-2}\partial_y \xi_{j-1}.
 \end{equation}
  In view of equation \eqref{eqxi1},  the estimation on
  $\xi_1$ follows from the classical Gevrey regularity theorem for parabolic equation.  And we conclude, for some $T>0$ independent of $\eps,$
  \begin{equation}\label{initialest}
 	\sup_{t\in[0,T]}\norm{\xi_{1}(t)}_{3\rho_0/2,\sigma} \leq C\norm{u_0}_{2\rho_0,\sigma}.
 \end{equation}
 Note that the higher order derivatives are involved in the initial datum $u_0$ on the right side of \eqref{eqxi1}. This can be overcome by reducing the initial Gevrey radius $2\rho_0$ to a smaller one, saying  $3\rho_0/2$ for instance.

  Now we consider the case $j\geq 2.$  
  Applying $\comi y^{\ell-1}\partial_x^m $ to the above equation \eqref{eqxi} we have
 \begin{equation}\label{eqxi+}
 \begin{aligned}
 	&\big( \partial_t     -\partial_y^2     -\eps\partial_x^2 \big) \comi y^{\ell-1}\partial_x^m \xi_j =-2\inner{\partial_y \comi y^{\ell-1}}\partial_y\partial_x^m\xi_j-\inner{\partial_y^2\comi y^{\ell-1}}\partial_x^m\xi_j\\
 &	\qquad \qquad	\qquad  -\comi y^{\ell-1}\partial_x^m  \Big(\inner{u^{s}+   u_{j-1} }\partial_x  \xi_{j-1}+ \xi_j   \partial_x  u_{j-2}   +  \zeta_{j-1} \partial_y \inner{  u^s+u _{j-1}}+v_{j-2}\partial_y \xi_{j-1}\Big).
 	\end{aligned}
 \end{equation}
 Moreover for the terms on right side,  direct computation yields
 \begin{equation*}
 \begin{aligned}
&\Big|\Big(\comi y^{\ell-1}\partial_x^m  \Big(\inner{u^{s}+   u_{j-1} }\partial_x  \xi_{j-1}+ \xi_j   \partial_x  u_{j-2}   +  \zeta_{j-1} \partial_y \inner{  u^s+u _{j-1}}+v_{j-2}\partial_y \xi_{j-1}\Big),~\comi y^{\ell-1}\partial_x^m \xi_j \Big)_{L^2}\Big|\\
 	&\qquad\qquad+\Big|\Big(-2\inner{\partial_y \comi y^{\ell-1}}\partial_y\partial_x^m\xi_j-\inner{\partial_y^2\comi y^{\ell-1}}\partial_x^m\xi_j,~\comi y^{\ell-1}\partial_x^m \xi_j \Big)_{L^2}\Big|\\
 	&\leq \frac{\eps}{2}\norm{\comi y^{\ell-1}\partial_x^{m+1} \xi_j}_{L^2}^2+\frac{1}{2}\norm{\comi y^{\ell-1}\partial_y\partial_x^{m} \xi_j}_{L^2}^2+C \norm{\comi y^{\ell-1} \partial_x^{m} \xi_j}_{L^2}^2\\
 	&\quad+C\eps^{-1}\frac{\big[(m-6)!\big]^{2\sigma}}{\rho^{2(m-5)}} \inner{1+\norm{u_{j-1}}_{\rho,\sigma}^2+\norm{u_{j-2}}_{\rho,\sigma}^2}\norm{\xi_{j-1}}_{\rho,\sigma}^2.
 	\end{aligned}
 \end{equation*}
 Thus for $m\geq 6,$, we can apply energy method and the Gronwall inequality to \eqref{eqxi+} to obtain by noting that $\xi_j\big|_{t=0}=0$, 
 \begin{eqnarray*}
 	&& \norm{\comi y^{\ell-1} \partial_x^{m} \xi_j(t)}_{L^2}^2+  \int_0^t \norm{\comi y^{\ell-1}\partial_y \partial_x^{m} \xi_j}_{L^2}^2+\eps \int_0^t \norm{\comi y^{\ell-1}\partial_x^{m+1} \xi_j}_{L^2}^2\\
 	& \leq & C\eps^{-1}\frac{\big[(m-6)!\big]^{2\sigma}}{\rho^{2(m-5)}}\int_0^t  \inner{1+\norm{u_{j-1}(s)}_{\rho,\sigma}^2+\norm{u_{j-2}(s)}_{\rho,\sigma}^2}\norm{\xi_{j-1}(s)}_{\rho,\sigma}^2\,ds.
 \end{eqnarray*}
 The upper bound estimate for $m\leq 5$ is straightforward, and we have
 \begin{eqnarray*}
 	  \sup_{m\leq 5}\norm{\comi y^{\ell-1} \partial_x^{m} \xi_j(t)}_{L^2}^2  \leq C\eps^{-1} \int_0^t  \inner{1+\norm{u_{j-1}(s)}_{\rho,\sigma}^2+\norm{u_{j-2}(s)}_{\rho,\sigma}^2}\norm{\xi_{j-1}(s)}_{\rho,\sigma}^2\,ds.
 \end{eqnarray*}
 Similarly, for $m\geq 6,$ 
  \begin{eqnarray*}
 	&& \norm{\comi y^{\ell} \partial_x^{m} \partial_y \xi_j(t)}_{L^2}^2+  \int_0^t \norm{\comi y^{\ell}  \partial_x^{m}\partial_y^2 \xi_j}_{L^2}^2+\eps \int_0^t \norm{\comi y^{\ell}\partial_x^{m+1} \partial_y\xi_j}_{L^2}^2\\
 	& \leq & C\eps^{-1}\frac{\big[(m-6)!\big]^{2\sigma}}{\rho^{2(m-5)}}\int_0^t  \inner{1+\norm{u_{j-1}(s)}_{\rho,\sigma}^2+\norm{u_{j-2}(s)}_{\rho,\sigma}^2}\norm{\xi_{j-1}(s)}_{\rho,\sigma}^2\,ds,
 \end{eqnarray*}
 and for $1\leq q\leq 4$ and $p+q\geq 6$, 
repeating the argument used in Proposition \ref{norma} yields  
 \begin{equation*} 
 \begin{aligned}
 	& \norm{\comi y^{\ell} \partial_x^{p} \partial_y^q \partial_y \xi_j(t)}_{L^2}^2+  \int_0^t \norm{\comi y^{\ell}  \partial_x^{p} \partial_y^{q+1} \partial_y\xi_j}_{L^2}^2+\eps \int_0^t \norm{\comi y^{\ell}\partial_x^{p+1} \partial_y^q\partial_y\xi_j}_{L^2}^2\\
 	 \leq & C\eps^{-1}\frac{\big[(p+q-6)!\big]^{2\sigma}}{\rho^{2(p+q-5)}}\int_0^t  \inner{1+\norm{u_{j-1}(s)}_{\rho,\sigma}^2+\norm{u_{j-2}(s)}_{\rho,\sigma}^2}\norm{\xi_{j-1}(s)}_{\rho,\sigma}^2\,ds.
 	\end{aligned}
 \end{equation*}
 And the above two estimates for $m\leq 5$ and $p+q\leq 5$ are also straightforward.
 Combining the above inequalities we conclude 
  \begin{equation}\label{upxie}
 \forall~j\geq 2,\quad 	\norm{\xi_{j}(t)}_{\rho,\sigma}^2 \leq C\eps^{-1} \int_0^t  \inner{1+\norm{u_{j-1}(s)}_{\rho,\sigma}^2+\norm{u_{j-2}(s)}_{\rho,\sigma}^2}\norm{\xi_{j-1}(s)}_{\rho,\sigma}^2\,ds.
 \end{equation}
The above estimate and \eqref{initialest} enable us to use induction on $j$ to conclude that there exists a constant $M,$ depending only on $\norm{u_0}_{2\rho_0,\sigma},$ such that
\begin{equation*}
\forall~j\geq 0,\quad \norm{u_{j}(s)}_{3\rho_0/2,\sigma}\leq M,\quad {\rm and}~~
	 \sup_{0\leq t\leq T_\eps^*}	\norm{\xi_{j}(t)}_{3\rho_0/2,\sigma}^2\leq      C  2^{-j+1}  \norm{u_0}_{2\rho_0,\sigma}^2,
\end{equation*}
 provided   $T_\eps^*\leq \frac{\eps}{2C(2M^2+1)}$ with $C$ the constant in \eqref{upxie}.   This implies $u_j, j\geq 0,$ is a Cauchy sequence in the Banach space $L^\infty\inner{[0,T_\eps^*]; X_{3\rho_0/2,\sigma}},$ with $T_\eps^*$ depending only on $\eps$ but independent of $j$.  Thus the limit $u_\eps$ of the Cauchy sequence  $u_j$  in $L^\infty\inner{[0,T_\eps^*]; X_{3\rho_0/2,\sigma}}$  solves the initial-boundary problem \eqref{newreg}.  The proof is thus complete.
  \end{proof}

\section{Proof of the main result Theorem \ref{mainthm}}\label{sec7}

In this section, we will prove the main result Theorem \ref{mainthm}.  

\subsection{ Proof of Theorem \ref{mainthm}:   existence} Here we will adopt the idea of abstract Cauchy-Kovalevskaya theorem to prove the existence of solution to equation \eqref{++repran}, by virtue of the uniform estimate established in Theorem \ref{uniestgev}.
Let the initial data $u_0\in  X_{2\rho_0,\sigma}$ satisfy the  assumptions listed in Theorem  \ref{mainthm}. Then by Theorem \ref{existhm}, we can find a solution  $u_\eps\in L^\infty\inner{[0,   T^*_\eps];~X_{3\rho_0/2,\sigma}}$  to the regularized equation \eqref{regpran}.   In the following discussions we will remove the $\eps$-dependence of the lifespan and  derive an uniform upper bound for $u_\eps$

{\it Step (i).}  
We begin with the  construction of  two constants $R$ and $\lambda$, which depend only on the initial datum $u_0$ and the constants $C_*, c_j$ given respectively in Theorem \ref{uniestgev} and Assumption  \ref{maas}, as well as the constants in the  Sobolev imbedding inequalities. 
First, in view of \eqref{eqnordif}, we can find a constant $\hat C\geq 1,$ depending only on $\rho_0, $ such that      
\begin{equation}\label{intiu}
	\abs{u_0}_{\rho_0, \sigma}\leq  \hat C \inner{\norm{u_0}_{2\rho_0, \sigma}+\norm{u_0}_{2\rho_0, \sigma}^2}.
\end{equation}
   And by Sobolev inequalities and the definition of $\abs{\cdot}_{\rho,\sigma}$ (see Definition \ref{gevspace}), we deduce that, for any $t\geq 0$ and for any $(x,y)\in\mathbb R_+^2,$
\begin{equation}\label{uppuniform}
\begin{aligned}
 &	  \sum_{1\leq j\leq 2}\Big(\norm{\comi y^{\ell-1}\partial_x^j u_\eps(t)}_{L^\infty}  +\norm{\partial_x^{j-1} v_\eps(t)}_{L^\infty}+  \norm{\comi y^{\ell}\partial_x^j \omega_\eps(t)}_{L^\infty}\Big)\\
 	&\quad  + \sum_{1\leq i, j \leq 2}\norm{\comi y^{\ell+1}\partial_x^i\partial_y^j\omega_\eps(t)}_{L^\infty} +\norm{\comi y^{\ell}\omega_\eps(t)}_{L^\infty} +\norm{\comi y^{\ell+1} \partial_y\omega_\eps(t)}_{L^\infty} 
 \leq \tilde C  \abs{u(t)}_{\rho_0/2,\sigma},
 \end{aligned}
\end{equation}
with $\tilde C $ being a constant depending only on the  Sobolev imbedding constants but independent of $\eps.$   Let $C_*\geq 1$ be the constant given in Theorem \ref{uniestgev} and let $\hat C, \tilde C$ be the constants given in \eqref{intiu}-\eqref{uppuniform}.    Now we take two positive constants $R>0, \lambda>0$ such that 
\begin{equation}\label{Rtobe}
	R\geq 4C_* \hat C\inner{\norm{u_0}_{2\rho_0, \sigma}+\norm{u_0}_{2\rho_0, \sigma}^2},  ~\, {\rm and}~\sqrt 2 R  \tilde C \leq \frac{1}{4}\min\Big\{c_0, c_1\Big\},
\end{equation}
and 
\begin{equation}\label{lambdach}
	\frac{  \sqrt{5C_*+C_* R^2}}{\sqrt{\lambda}}= \frac{1}{2},\end{equation}
recalling    $c_0,c_1$ are the constants given in Assumption \ref{maas}.   We remark that the above $R$ indeed exist,  provided  
\begin{eqnarray*}
	 \norm{u_0}_{2\rho_0, \sigma}+\norm{u_0}_{2\rho_0, \sigma}^2<\frac{1}{16 \sqrt 2 \tilde C \hat C C_*}\min\Big\{c_0, c_1\Big\}.
\end{eqnarray*} 
In the following discussion we will  let $R$ and $\lambda$ be fixed 
so that \eqref{Rtobe} and \eqref{lambdach} hold.  

 {\it Step (ii)}.  
 {We define a function $T\rightarrow \normm{u_\eps}_{(\lambda,T)} $ by setting}
 \begin{equation}\label{trinormdef}
\normm{u_\eps}_{(\lambda,T)}\stackrel{\rm def}{ =}\sup_{\rho,t } \inner{ \frac{ \rho_0-\rho- \lambda t}{\rho_0-\rho}  }^{1/2}\abs{u_\eps(t)}_{\rho,\sigma},
\end{equation}
 where the supremum is taken over all pairs $(\rho, t)$ such that $ \rho>0,\, 0\leq t \leq  T$ and  $\rho+ \lambda t <\rho_0.$  
   {Note that the above function is well-defined  over the interval $[0, P_\eps[$ with $P_\eps$  given by 
  \begin{eqnarray*}
  P_\eps=\sup\set{T\in [0,  \rho_0/\lambda[\,;~ \normm{u_\eps}_{(\lambda,T)}<+\infty}.
  \end{eqnarray*}
  Note that $P_\eps\geq T_\eps^*$ because of  \eqref{eqnordif} and by recalling $[0,T_\eps^*]$ is the interval of the existence for $u_\eps\in X_{3\rho_0/2,\sigma}.$}
It is clear
 that
 \begin{eqnarray*}
 	T\rightarrow \normm{u_\eps}_{(\lambda,T)}
 \end{eqnarray*}
 is a increasing function of $T.$  Moreover, we have
 \begin{equation}\label{up+a}
 	 \normm{u_\eps}_{(\lambda, 0)}=\sup_{\rho\in]0, \rho_0[}\abs{u_\eps(0)}_{\rho,\sigma}\leq \abs{u_0}_{\rho_0,\sigma}\leq  \hat C \inner{\norm{u_0}_{2\rho_0, \sigma}+\norm{u_0}_{2\rho_0, \sigma}^2}<R,
 \end{equation}
where in the second inequality we
have used \eqref{intiu} and the last one follows from \eqref{Rtobe}.

{\it Step (iii)}.  In this step,   { recalling $R$ is given in Step $(i)$ and  $P_\eps$ is defined in the previous step, we will show that  
 \begin{equation}\label{stmai}
 	\forall~0\leq T<\min \Big\{\rho_0/(4\lambda), P_\eps  \Big\},\quad  \normm{u_\eps}_{(\lambda, T)}\leq R.
 \end{equation}}To confirm this,  suppose on the contrary to \eqref{stmai} that
 {$\normm{u_\eps}_{(\lambda, t_\eps)}>R$ for some $t_\eps< \min\set{\rho_0/(4\lambda), P_\eps} .$}
 Then 
in view of  \eqref{up+a}, we can find  some  {$T_\eps\in]0, t_\eps[\subset[0, \rho_0/(4\lambda)]$}  such that
 \begin{equation}\label{teps}
 \normm{u_\eps}_{(\lambda,T_\eps)}= R,
 \end{equation}
 since   $
 	T\rightarrow \normm{u_\eps}_{(\lambda,T)}
$
 is a increasing function of $T.$ 
 Thus,  {observing $T_\eps<\rho_0/(4\lambda),$}
\begin{eqnarray*}
\forall~t\in[0,T_\eps],\quad \frac{\sqrt 2}{2}\abs{u(t)}_{\rho_0/2,\sigma} \leq   \inner{ \frac{ \rho_0-\rho_0/2- \lambda t}{\rho_0-\rho_0/2}  }^{1/2}\abs{u(t)}_{\rho_0/2,\sigma}\leq \normm{u}_{(\lambda,T_\eps)}= R.
\end{eqnarray*}
    As a result,  for any $t\in[0,T_\eps],$
\begin{eqnarray*}
	 	 &	&  \sum_{1\leq j\leq 2}\Big(\norm{\comi y^{\ell-1}\partial_x^j u_\eps}_{L^\infty}  +\norm{\partial_x^{j-1} v_\eps}_{L^\infty}+  \norm{\comi y^{\ell}\partial_x^j \omega_\eps}_{L^\infty}\Big)\\
 	&&\quad  + \sum_{1\leq i, j \leq 2}\norm{\comi y^{\ell+1}\partial_x^i\partial_y^j\omega_\eps}_{L^\infty} +\norm{\comi y^{\ell}\omega_\eps}_{L^\infty} +\norm{\comi y^{\ell+1} \partial_y\omega_\eps}_{L^\infty} 
 \leq \frac{1}{4}\min\big\{c_0, c_1\big\}
\end{eqnarray*}
because  \eqref{uppuniform} and \eqref{Rtobe}, so that
 the property \eqref{condi} holds by $u_\eps$ for all $t\in[0, T_\eps]$ due to the fact that $\alpha\leq\ell.$    

 In the following argument, we let 
   $(\rho, t)$ be an arbitrary pair which is fixed at moment and satisfies that  $ \rho>0,\,  t\in[0, T_\eps]$ and  $\rho+ \lambda t<\rho_0.$  Then we have, in view of   \eqref{trinormdef},
   \begin{equation}\label{fiesonssd}
  \forall~0\leq s\leq t, \quad 	 \abs{u_\eps(s)}_{\rho,\sigma}\leq   \normm{u_\eps}_{(\lambda, T_\eps) }\inner{ \frac{ \rho_0-\rho- \lambda s}{\rho_0-\rho}  }^{-1/2}.
   \end{equation}
Furthermore,
we take in particular  such a $\tilde\rho(s)$ that 
\begin{eqnarray*}
\tilde \rho(s)=\frac{\rho_0+  \rho- \lambda  s}{2}.
\end{eqnarray*}
Then direct calculation shows 
that 
\begin{eqnarray}\label{rhos}
 \forall~0\leq s\leq t ,\qquad \rho< \tilde \rho(s)  \quad{\rm and}\quad \tilde \rho(s)+ \lambda  s <\rho_0,
\end{eqnarray}  
and
 \begin{eqnarray}\label{rhom}
 \forall~0\leq s\leq t ,\qquad \tilde\rho(s)-\rho=\frac{\rho_0-\rho - \lambda  s}{2}=  \rho_0-\tilde \rho(s)- \lambda s,\quad \rho_0-\tilde\rho(s)\leq \rho_0-\rho.
\end{eqnarray}
The   inequalities in \eqref{rhos} imply
\begin{equation}\label{tilrho}
	\forall~0\leq s\leq t ,\quad  \abs{u_\eps(s)}_{\tilde\rho(s),\sigma}\leq \normm{u_\eps}_{(\lambda, T_\eps)}\inner{\frac{\rho_0-\tilde\rho(s)- \lambda s}{\rho_0-\tilde\rho(s)}}^{-1/2}\leq \normm{u_\eps}_{(\lambda, T_\eps)}\inner{\frac{2(\rho_0- \rho)}{\rho_0- \rho- \lambda s}}^{ 1/2},
\end{equation} 
where the last inequality follows from \eqref{rhom}.

Now we apply  Theorem \ref{uniestgev} to the  pair $(\rho,\tilde\rho(s))$ given above to have  for any $t\in[0, T_\eps],$      
 \begin{equation*}
	\abs{u_\eps (t)}_{\rho,\sigma}^2\leq C_* \abs{u_0}_{\rho, \sigma}^2+C_* \int_{0}^t \inner{\abs{u_\eps(s)}_{\rho,\sigma}^2+\abs{u_\eps(s)}_{\rho,\sigma}^4} \,ds+C_* \int_{0}^t\frac{ \abs{u_\eps(s)}_{\tilde\rho(s),\sigma}^2}{\tilde \rho(s)-\rho}\,ds.
\end{equation*}	 
 Moreover, we insert    \eqref{fiesonssd}  and \eqref{tilrho}  into   the above inequality     to obtain, using  \eqref{rhom}  as well, 
\begin{eqnarray*}
\abs{ u_\eps(t)}_{ \rho,\sigma}^2& \leq  &C_* \abs{u_0 }_{\rho,\sigma}^2 +C_* \normm{u_\eps}_{(\lambda, T_\eps)}^2\int_0^t   \frac{\rho_0-\rho  }{ \rho_0-\rho - \lambda  s} \,ds+C_* \normm{u_\eps}_{(\lambda, T_\eps)}^4\int_0^t   \frac{\inner{\rho_0-\rho}^2}{\inner{\rho_0-\rho - \lambda  s}^2} \,ds\\
&&\qquad+C_*\normm{u_\eps}_{(\lambda, T_\eps)}^2  \int_0^t   \frac{2^2 \inner{\rho_0-\rho}}{\inner{\rho_0-\rho - \lambda  s}^2} \,ds\\
   &\leq & C_* \abs{u_0 }_{\rho,\sigma}^2 +\frac{\inner{5C_*+C_* R^2}\normm{u_\eps}_{(\lambda, T_\eps)}^2}{\lambda} \inner{\frac{\rho_0-\rho-\lambda t}{\rho_0-\rho}}^{-1},
\end{eqnarray*}
where in the last inequality we have used
 \eqref{teps} and the fact that 
\begin{eqnarray*}
	  \frac{\rho_0-\rho  }{ \rho_0-\rho - \lambda  s} \leq  \frac{\inner{\rho_0-\rho}^2}{\inner{\rho_0-\rho - \lambda  s}^2}\leq \frac{  \inner{\rho_0-\rho}}{\inner{\rho_0-\rho - \lambda  s}^2}.
\end{eqnarray*}
Then multiplying both sides by the fact $\inner{\rho_0-\rho-\lambda t}/\inner{\rho_0-\rho}$ implies, observing  $(\rho, t)$ is an arbitrary pair with $\rho>0$, $t\in[0, T^\eps]$  and $\rho+\lambda t<\rho_0$,   
\begin{eqnarray*}
	\normm{u_\eps}_{(\lambda, T_\eps)}  \leq   \sqrt{C_*} \sup_{\rho,t} \inner{\frac{\rho_0-\rho-\lambda t}{\rho_0-\rho}}^{1/2}\abs{u_0 }_{\rho,\sigma} +\frac{ \sqrt{5C_*+C_* R^2}}{\sqrt{\lambda}}\normm{u_\eps}_{(\lambda, T_\eps)} 
	  \leq  C_*  \abs{u_0 }_{\rho_0,\sigma} +\frac{1}{2}\normm{u_\eps}_{(\lambda, T_\eps)}.
\end{eqnarray*}
Here the last inequality holds because of \eqref{lambdach}  and the fact that $C_*\geq1$.  Then we conclude 
\begin{eqnarray*}
	\normm{u_\eps}_{(\lambda, T_\eps)}\leq  2C_* \abs{u_0 }_{\rho_0,\sigma}\leq 2C_* \hat C\inner{\norm{u_0}_{2\rho_0, \sigma}+\norm{u_0}_{2\rho_0, \sigma}^2}\leq R/2,
	 \end{eqnarray*}
where the second inequality follows from \eqref{intiu} and  in the last inequality we have used  \eqref{Rtobe}.       This contradicts  \eqref{teps} so that \eqref{stmai} holds.

{\it Step (iv)}.   {We conclude  that $P_\eps> \rho_0/(4\lambda)$,  otherwise, it follows from \eqref{stmai}  that for any $T\in [0, P_\eps[$ we have $\normm{u_\eps}_{(\lambda, T)}\leq R,$ which contradicts to the definition of $P_\eps.$} 
 {Consequently, we can rewrite \eqref{stmai} as  
\begin{eqnarray*}
\forall~0\leq T\leq  \rho_0/(4\lambda), \quad  \normm{u_\eps}_{(\lambda, T)}\leq R.	
\end{eqnarray*}
 Thus}
\begin{eqnarray*}
	\forall~t\in[0, \rho_0/(4\lambda)],\quad   \frac{\sqrt 2}{2}\abs{u_\eps(t)}_{\rho_0/2,\sigma} \leq   \inner{ \frac{ \rho_0-\rho_0/2- \lambda t}{\rho_0-\rho_0/2}  }^{1/2}\abs{u(t)}_{\rho_0/2,\sigma}\leq \normm{u_\eps}_{\inner{\lambda,\rho_0/(4\lambda)}}\leq R.
\end{eqnarray*}
This gives
\begin{eqnarray*}
\forall~\eps>0,\quad	 u_\eps\in L^\infty\inner{[0, \rho_0/(4\lambda)];~X_{\rho_0/2,\sigma}} ~ {\rm and }~\norm{u_\eps(t)}_{\rho_0/2,\sigma}\leq \abs{u_\eps(t)}_{\rho_0/2,\sigma} \leq \sqrt 2R.
\end{eqnarray*}
Now let  $\eps\rightarrow 0$ and we have, by compactness arguments,  the limit $u$ of $u_\eps$ solves the equation \eqref{++repran}.  We complete the existence part of Theorem \ref{mainthm}.

\subsection{Proof of Theorem \ref{mainthm}: uniqueness}  Let $u_1, u_2\in  L^\infty\inner{[0,\rho_0/(4\lambda)];~X_{\rho_0/2,\sigma}}$ be two solutions to the Prandtl equation \eqref{++repran}, and let $  v_j=- \int_0^y\partial_x u_j(x,\tilde y)\,d\tilde y.$  Then the differences  
	\begin{eqnarray*}
		  u\stackrel{\rm def}{=}u_1-u_2, \quad v\stackrel{\rm def}{=}v_1-v_2, 	\end{eqnarray*}
	satisfy the following initial boundary problem, using the notation $\omega=\partial_y u$ and $\omega_j=\partial_y u_j$ as before,
		\begin{equation}\label{eqdif}
\left\{
	\begin{aligned}
	&\partial_t     u+\inner{u^{s}+   u_1}\partial_x  u +  v_1 \partial_y u + u \partial_x u_2+  v \inner{  \omega^s+\omega_2}-\partial_y^2    u=0,\\
	&   u \big|_{y=0}=0,\quad\lim_{y\rightarrow +\infty}    u=0,\\
	&   u\big|_{t=0}=0.
\end{aligned}	
\right.
\end{equation}
Moreover, we have the equations for $\omega$ and $\partial_y\omega$:
\begin{equation}\label{eqvoti}
\partial_t     \omega+\inner{u^{s}+   u_1}\partial_x  \omega+  v_1 \partial_y \omega-\partial_y^2    \omega+ u \partial_x \omega_2+  v \inner{  \partial_y\omega^s+\partial_y\omega_2} =0,
\end{equation}
and
\begin{equation}\label{eqofpyo}
\begin{aligned}
	&\partial_t   (  \partial_y \omega)+\inner{u^{s}+   u_1}\partial_x(\partial_y  \omega)+  v_1 \partial_y (\partial_y  \omega)-\partial_y^2    (\partial_y  \omega)+  v \inner{  \partial_y^2\omega^s+\partial_y^2\omega_2} 
	\\
	=&-\Big[(\omega^s+\omega_1)\partial_x\omega- \inner{  \partial_y\omega^s+\partial_y\omega_2}\partial_xu\Big]+(\partial_x u_1)\partial_y\omega- \omega \partial_x \omega_2- u \partial_x \partial_y\omega_2.
	\end{aligned}
\end{equation}
Now we apply $\partial_x^m$ to the three equations above, and then we have, as in the previous sections,  several  terms have loss of $x$ derivative. Precisely,  
 $
	(\partial_x^m v) (\omega^s+\omega_2)
$
is involved in the equation for $\partial_x^m u,$   and $ (\partial_x^m v) (\partial_y\omega^s+\partial_y \omega_2)$
  in the equation for $\partial_x^m \omega,$  and meanwhile two terms  $ (\partial_x^m v) (\partial_y^2\omega^s+\partial_y^2 \omega_2)$ and $\partial_x^m\big[(\omega^s+\omega_1)\partial_x\omega- \inner{  \partial_y\omega^s+\partial_y\omega_2}\partial_xu\big]$
  in the equation for $\partial_x^m\partial_y\omega.$ To overcome the degeneracy, we just follow the same strategy as in Sections \ref{sec3}-\ref{sec5},  with $f_m, h_m$ and $g_m$ therein replaced respectively by  
\begin{eqnarray*}
f_m^*&=&\chi_1\partial_x^m \omega-\chi_1\frac{  \partial_y\omega^s+\partial_y \omega_2 }{\omega^s+\omega_2 }\partial_x^m u,\\
h_m^*&=&\chi_2\partial_x^m \partial_y\omega-\chi_2\frac{  \partial_y^2\omega^s+\partial_y ^2\omega_2 }{\partial_y\omega^s+\partial_y \omega_2 }\partial_x^m \omega,\\
g_m^*&=& \partial_x^{m-1}\Big[(\omega^s+\omega_1)\partial_x\omega- \inner{  \partial_y\omega^s+\partial_y\omega_2}\partial_xu\Big].
\end{eqnarray*}
 Then just repeating the argument in the  Sections \ref{sec3}-\ref{sec5}, with slight modification,  we can obtain, observing $u|_{t=0}=0,$ 
 \begin{equation}\label{laest}
 		\normm{u  (t)}_{\rho,\sigma}^2\leq   C \int_{0}^t\frac{\inner{1+\abs{u_1(s)}_{\rho/2,\sigma}^2+\abs{u_2(s)}_{\rho/2,\sigma}^2+\abs{u_1(s)}_{\rho/2,\sigma}^4+\abs{u_2(s)}_{\rho/2,\sigma}^4}  \normm{u(s) }_{\tilde\rho(s),\sigma}^2}{\tilde \rho(s)-\rho}\,ds,
 \end{equation}
 where  the definition of $\normm{u}_{\rho,\sigma}$ is  similar as $\abs{u}_{\rho,\sigma}$ by just replacing respectively  the summations \begin{eqnarray*}
 	\sup_{\stackrel{1\leq j\leq 4}{i+j\geq 6}}  \quad {\rm and}~~\sup_{\stackrel{1\leq j\leq 4}{i+j\leq 5}} 
 \end{eqnarray*}
 in Definition \ref{defgev}     by  
 \begin{eqnarray*}
 		\sup_{\stackrel{1\leq j\leq 2}{i+j\geq 6}}  \quad {\rm and}~~\sup_{\stackrel{1\leq j\leq 2}{i+j\leq 5}}.
 \end{eqnarray*}
 Now we emphasize the difference between \eqref{laest} and \eqref{weesun}.
Note that we work on $\normm{u}_{\rho,\sigma}$ instead of $\abs{u}_{\rho,\sigma}$ because 
  we lose $y$-derivative for the term  $ v \inner{  \partial_y^2\omega^s+\partial_y^2\omega_2} $ in \eqref{eqofpyo}.  So we have to reduce the order of $y$ derivatives from $4$ to $2.$  Moreover, observe that we also lose $x$ derivatives for $u_1$ and $u_2$ in equations \eqref{eqdif}-\eqref{eqofpyo} and this can be overcome by reducing the Gevrey radius $\rho$ to $\rho/2.$    Then by virtue of \eqref{laest}, we can follow the argument  used in  the existence part to conclude 
  \begin{eqnarray*}
  	\sup_{\rho,t} \inner{\frac{\rho_0/2-\rho-\lambda t}{\rho_0/2-\rho}}^{1/2}\normm{u(s)}_{\rho,\sigma}=0,
  \end{eqnarray*}
  where the supremum is taken over all pairs $(\rho, t)$ such that $ \rho>0$ and  $\rho+ \lambda t <\rho_0/2.$  
   And thus $u\equiv 0$ and the uniqueness  follows.

\section{  {
 General initial data}}\label{sectiongeneral}

  In this section, we will clarify why the above
 result holds for  general initial data without requiring the small perturbations around a shear flow.  Precisely, we consider the Prandtl equation in $\Omega\times \mathbb R_+$ with $\Omega$ the whole space $\mathbb R$ or the torus $\mathbb T,$ that is, 
  \begin{eqnarray}\label{orig9}
  	\left\{\begin{array}{l}\partial_t u^P + u^P \partial_x u^P + v^P\partial_yu -\partial_y^2u^P + \partial_x p
=0,\quad t>0,\quad x\in\Omega,\quad y>0, \\
\partial_xu^P +\partial_yv^P =0, \\
u^P|_{y=0} = v^P|_{y=0} =0 , \quad  \lim_{y\to+\infty} u =U(t,x), \\
u^P|_{t=0} =u^P_0 (x,y)\,.
\end{array}\right.
  \end{eqnarray} 
Without loss generality, we  suppose that $U\equiv 0$, and thus $\partial_xp\equiv 0$ by Bernoulli law.  

To investigate   the well-posedness in Gevrey class for the above Prandtl equation, there are  two main ingredients, one  is about the existence of approximate solution for the regularized Prandtl equation
\begin{eqnarray}\label{sys9}
		\left\{\begin{array}{l}\partial_t u^P + u^P \partial_x u^P + v^P\partial_yu -\partial_y^2u^P-\eps \partial_x^2u^P  
=0,\quad t>0,\quad x\in\Omega,\quad y>0, \\
\partial_xu^P +\partial_yv^P =0, \\
u^P|_{y=0} = v^P|_{y=0} =0 , \quad  \lim_{y\to+\infty} u =0, \\
u^P|_{t=0} =u^P_0 (x,y),\end{array}\right.
\end{eqnarray}
where and in the following discussion,  we  will use  $u^P$ and $v^P$ instead of $u^P_\eps$ and $v^P_\eps,$  by omitting  $\eps$ for simpler presentation.   And  another ingredient is  the uniform estimate for the approximate solution, which is the main concern of this paper, cf.  Sections  \ref{sec2}-\ref{sec5}.  We will explain why we do not need  the small perturbation in
obtaining the uniform estimate, and the requirement on the initial data is  only for  the construction of approximate solution.  

Suppose that the  initial-boundary problem \eqref{sys9} admits  a solution  $(u^P, v^P)$  in the interval $[0,T]$ satisfying  the properties listed 
below.  That is, given $y_0>0$, there are three large constants $C_1, C_2, C_3$ and a positive number $\delta_0\in [0,y_0/2]$ and two positive numbers $\ell, \alpha$  with $\ell>3/2$ and $\alpha+{1\over 2}>\ell,$  such that  for any      $t\in[0, T]$  and  any $x\in \Omega,$   we have,  using the notation $\omega=\partial_y u^P,$
 \begin{eqnarray}\label{ass9}
\left\{
\begin{aligned}
  &\abs{\partial_y\omega(t,x,y)}\geq \frac{1}{4 C_1}, ~\,{\rm if}~~\,y\in\big[y_0-\frac{7}{4}\delta_0, y_0+\frac{7}{4}\delta_0\big], \\
  &4^{-1}C_2^{-1}  \comi y^{-\alpha}\leq \abs{  \omega(t,x,y)} \leq  4 C_2\comi y^{-\alpha},~~\,{\rm if}~~\,y\in\big[0,  y_0-\frac{5 }{4}\delta_0 \big  ]\cup  \big [ y_0+ \frac{5 }{4}\delta_0 ,~+\infty \big[ ,\\
  & \abs{\partial_y\omega(t,x,y)}\leq 4C_2 \comi y^{-\alpha-1} ~~\, \textrm {for  } ~y\geq 0, \\
 &\sum_{1\leq j\leq 2}\Big(\norm{\comi y^{\ell-1}\partial_x^j u^P}_{L^\infty}  +\norm{\partial_x^{j-1} v^P}_{L^\infty}+  \norm{\comi y^{\ell}\partial_x^j \omega}_{L^\infty}\Big)+ \sum_{1\leq i, j \leq 2}\norm{\comi y^{\ell+1}\partial_x^i\partial_y^j\omega}_{L^\infty}\leq C_3. 
 \end{aligned}
 \right.
\end{eqnarray} 
Let $\inner{X_{\rho,\sigma},~\norm{\cdot} _{\rho,\sigma}}$ be the Gevrey space in 
the tangential variable $x\in\Omega$  introduced in Definition \ref{defgev},   with the $L^2$ norm therein  taken over $\Omega\times\mathbb R_+.$  Similarly, as in Definition \ref{gevspace}, we can define $\abs{u^P}_{\rho,\sigma}$ with the auxilliary functions therein replaced respectively  by the following new ones:
\begin{eqnarray*}
	f_m^P&=&\chi_1\partial_x^m\omega- \chi_1\frac{\partial_y\omega}{\omega}\partial_x^m u^P,\\
	h_m^P&=&\chi_2\partial_x^m\partial_y\omega-\chi_2 \frac{\partial_y^2\omega}{\partial_y\omega}\partial_x^m \omega,\\
	g_m^P&=&\partial_x^{m-1}\Big(\omega\partial_x\omega-(\partial_y\omega)\partial_xu^P\Big).
\end{eqnarray*}
Here $\chi_i, i=1,2,$ are given in \eqref{chi1} and \eqref{cutofffu}.

\begin{theorem}[uniform estimates in Gevrey space]\label{uniform9}
Let $3/2\leq\sigma\leq 2.$  Let the initial datum $u_0^P\in X_{2\rho_0,\sigma}$ and let  $u^P \in L^\infty\inner{[0, T];~X_{\rho_0,\sigma}}$ be a solution to \eqref{sys9}    such that  the properties listed in  \eqref{ass9} hold. 
 Then  there exists   a constant   $\tilde C_*>1,$  independent of $\eps,$    such that    the estimate 
 \begin{equation*}
	\abs{u^P (t)}_{\rho,\sigma}^2\leq \tilde C_* \abs{u^P_0}_{\rho, \sigma}^2+\tilde C_* \int_{0}^t \inner{\abs{u^P(s)}_{\rho,\sigma}^2+\abs{u^P(s)}_{\rho,\sigma}^4} \,ds+\tilde C_* \int_{0}^t\frac{ \abs{u^P(s)}_{\tilde\rho,\sigma}^2}{\tilde \rho-\rho}\,ds 
\end{equation*}	 
 holds for any pair $(\rho,\tilde\rho)$ with   $0<\rho<\tilde \rho<\rho_0$, and for any $t\in[0,\tilde T]$, where $[0,\tilde T]$ is the maximal interval of existence for $\abs{u^P(t)}_{\tilde \rho,\sigma}<+\infty.$  
\end{theorem}

\begin{proof}
The proof is similar to that of  Theorem \ref{uniestgev} by following the argument in   Sections  \ref{sec3}-\ref{sec5} 	 with slight modification.   So we omit it.  
\end{proof}
\begin{remark}
To prove the above theorem  we   only require that   the initial datum $u_0^P\in X_{2\rho_0,\sigma}$ and satisfies the conditions in \eqref{ass9}. Hence,
 we do not need the additional assumption that the initial datum is the small perturbation of a shear flow. 
\end{remark}

The remaining ingredient in the proof
 is to construct solution to \eqref{sys9} satisfying
 the properties listed in \eqref{ass9}. In fact,  this together with the uniform estimate given in   Theorem  \ref{uniform9}  enables us to repeat the
 argument in      Section \ref{sec7} to conclude the well-posedness in Gevrey space to the original Prandtl equation \eqref{orig9}.   For this,
it is not difficult to construct solution to \eqref{sys9} because
 it is a  parabolic initial-boundary problem. The key point 
 is to prove the properties listed in \eqref{ass9} are preserved in time
by supposing that they hold initially.   It is clear that these properties are indeed preserved with time for the shear flows since they satisfy the heat equation with initial-boundary conditions (see Proposition \ref{proshf}),   and thus so are for the solutions to Prandtl equation  by small perturbation.  
 For the general initial data rather than the small perturbation around a shear flow,   the existence of such approximate solutions that satisfy  \eqref{ass9} is proven by  G\'{e}rard-Varet and Masmoudi \cite[Section 4]{GM}  where they use  the maximum principle so that the assumptions in Theorem \ref{uniform9} hold.  This enables us to conclude that   the  result obtained by  G\'{e}rard-Varet and Masmoudi \cite{GM} still holds when the  Gevrey index $7/4$ therein is   replaced by any $\sigma\in [3/2,2],$  and there is 
 no additional assumption required.  

\appendix

 \section{Sobolev  inequality}\label{ineqa}
 
\begin{lemma}
\label{sobine}
For any $h\in H^2\inner{\mathbb R_+^2}\cap C^2\inner{\mathbb R_+^2},$ we have
\begin{eqnarray*}
 	\norm{h}_{L^\infty(\mathbb R_+^2)}\leq \sqrt{2}\inner{\norm{h}_{L^2}+ \norm{\partial_x h}_{L^2}+ \norm{\partial_yh}_{L^2}+ \norm{\partial_x\partial_y h}_{L^2}}.
 	 \end{eqnarray*} 	
\end{lemma}

\begin{proof}
We begin with the 1D Sobolev inequality:
\begin{equation}
	\label{Sobe}
	\norm{f}_{L^\infty(\Omega)}\leq \norm{f}_{L^2(\Omega)}+\norm{f'}_{L^2(\Omega)},\quad \Omega=\mathbb R_+ ~{\rm or}~~\mathbb R.
\end{equation}
To see this,  let $\omega=\mathbb R_+$ and let $r\in\mathbb R_+.$  By mean value Theorem, we can find a $\xi\in[r,r+1]$ such that 
 \begin{eqnarray*}
    \int_{r}^{r+1}  f(\tilde r) d\tilde r =f(\xi). 
 \end{eqnarray*}
 Moreover 
 \begin{eqnarray*}
 	\abs{f(r)-f(\xi)}=\abs{\int_\xi^r f'(\tilde r) \,d\tilde r} \leq \norm{f'}_{L^2(\mathbb R_+)}.
 \end{eqnarray*}
 Thus 
 \begin{eqnarray*}
 	\abs{f(r)}\leq \abs{f(\xi)}+\abs{f(r)-f(\xi)}\leq \norm{f}_{L^2(\mathbb R_+)}+\norm{f'}_{L^2(\mathbb R_+)},
 \end{eqnarray*}
 which implies,  taking the supremum over $r\in \mathbb R_+,$   
 \begin{eqnarray*}
 	\norm{f}_{L^\infty(\mathbb R_+)}\leq \norm{f}_{L^2(\mathbb R_+)}+\norm{f'}_{L^2(\mathbb R_+)}.
 \end{eqnarray*}
 Similarly, the above estimate also holds with $\mathbb R_+$ replaced by $\mathbb R.$  Then \eqref{Sobe} follows. 
 
Now we use \eqref{Sobe} to prove Lemma \ref{sobine}. For any $(x,y)\in\mathbb R_+^2,$ we have
 \begin{eqnarray*}
 	\abs{h(x,y)}&\leq& \inner{\int_{\mathbb R_+} \abs{h(x,\tilde y)}^2\,d\tilde y}^{1/2}+\inner{\int_{\mathbb R_+} \abs{(\partial_yh)(x,\tilde y)}^2\,d\tilde y}^{1/2}\\
 	&\leq& \inner{2\int_{\mathbb R_+} \inner{\int_{\mathbb R}\abs{h(\tilde x,\tilde y)}^2 \,d\tilde x+\int_{\mathbb R}\abs{(\partial_x h)(\tilde x,\tilde y)}^2 \,d\tilde x} d\tilde y}^{1/2}\\
 	&&+\inner{2\int_{\mathbb R_+} \inner{\int_{\mathbb R}\abs{(\partial_yh)(\tilde x,\tilde y)}^2 \,d\tilde x+\int_{\mathbb R}\abs{(\partial_x\partial_y h)(\tilde x,\tilde y)}^2 \,d\tilde x} d\tilde y}^{1/2}.
 	 \end{eqnarray*}
Taking the supremum over $(x,y)\in\mathbb R_+^2$, we obtain  the desired estimate in Lemma \ref{sobine}. 
 \end{proof}

\section{auxilliary functions}

\begin{lemma}[Equation for $f_m$]\label{applem} Let $f_m$ be given in \eqref{funoffs}. Then
we have 
\begin{eqnarray*}
&&	\Big(\partial_t    +\inner{u^{s}+u}\partial_x+v\partial_y  -\partial_y^2  -\eps\partial_x^2 \Big)  f_m \\
&=& -\chi_1\sum_{k=1}^m {m\choose k}\inner{\partial_x^k u} \partial_x^{m-k+1}\omega-\chi_1 \sum_{k=1}^{m-1} {m\choose k}\inner{\partial_x^kv}  \partial_x^{m-k}\partial_y \omega  \\
&&+\chi_1 a \sum_{k=1}^m {m\choose k}\inner{\partial_x^k u} \partial_x^{m-k+1}u+\chi_1 a\sum_{k=1}^{m-1} {m\choose k}\inner{\partial_x^kv}   \partial_x^{m-k} \omega  \\
&&+\chi_1'v\partial_x^m\omega- 2\chi_1'\partial_x^m\partial_y\omega -\chi_1''\partial_x^m\omega-a \inner{\chi_1'v\partial_x^mu- 2\chi_1'\partial_x^m\omega -\chi_1''\partial_x^mu}   \\
&&+\Big[ \partial_x\omega-\inner{\partial_x u}a    
-2a  \partial_ya  -2\eps\frac{\partial_x \omega}{\omega^s+\omega}\partial_x a \Big]\chi_1 \partial_x^m u   \\
&&+ 2 \chi_1 \inner{\partial_ya }    \partial_x^m \omega+ 2 \chi_1' \inner{\partial_ya }    \partial_x^m u +2\eps \chi_1 \inner{\partial_xa }  \partial_x^{m+1} u.
\end{eqnarray*}
where  $	a = \frac{ \partial_y\omega^s+ \partial_y \omega }{ \omega^s+\omega }$.
\end{lemma}

\begin{proof}
	
Observe that $ \partial_x  u+ \partial_yv=0$ and then it follows from the equation 
\begin{equation}\label{uequ}
	\partial_t   u+\inner{u^{s}+u}\partial_xu+v\inner{\omega^s+\omega}-\partial_y^2 u-\eps\partial_x^2 u=0, 
\end{equation}
that $\omega=\partial_y u$ satisfies 
\begin{equation} \label{eqomegaeps}
	\partial_t   \omega+\inner{u^{s}+u}\partial_x\omega+v \inner{\partial_y\omega^s+\partial_y  \omega}-\partial_y^2 \omega-\eps\partial_x^2 \omega=0.\end{equation}
We apply the operator $\partial_x^m$ to the two  equations above and then multiply the resulting equations by $\chi_1(y)$;  this gives
\begin{eqnarray}\label{+eq1}
	\begin{aligned}
&	\Big(\partial_t    +\inner{u^{s}+u}\partial_x    +v\partial_y-\partial_y^2  -\eps\partial_x^2 \Big) \chi_1\partial_x^m u+\chi_1\inner{\partial_x^m v} \inner{\omega^s+\omega}  \\
	=&-\chi_1\sum_{k=1}^m {m\choose k} \inner{\partial_x^k u } \partial_x^{m-k+1} u-\chi_1\sum_{k=1}^{m-1} {m\choose k}\inner{\partial_x^kv} \partial_x^{m-k}\omega\\
	&+\chi_1'v\partial_x^mu- 2\chi_1'\partial_x^m\omega -\chi_1''\partial_x^mu,
	\end{aligned}
\end{eqnarray}
and
\begin{eqnarray}\lab{+eq2}
	\begin{aligned}
&	\Big(\partial_t    +\inner{u^{s}+u}\partial_x+v\partial_y  -\partial_y^2  -\eps\partial_x^2 \Big)\chi_1 \partial_x^m \omega+\chi_1\inner{\partial_x^m v}  \inner{\partial_y\omega^s+ \partial_y \omega} \\
	=& -\chi_1\sum_{k=1}^m {m\choose k}\inner{\partial_x^k u} \partial_x^{m-k+1} \omega-\chi_1\sum_{k=1}^{m-1} {m\choose k}\inner{\partial_x^kv} \partial_x^{m-k} \partial_y\omega\\
	&+\chi_1'v\partial_x^m\omega- 2\chi_1'\partial_x^m\partial_y\omega -\chi_1''\partial_x^m\omega.
	\end{aligned}
\end{eqnarray}
Observe $\abs{\omega^s+\omega}>0$ on supp$\chi_1$  and then we can   multiply  both sides of  \eqref{+eq1}  by 
the factor 
\begin{eqnarray*}
	 a= \frac{  \partial_y\omega^s+ \partial_y\omega}{ \omega^s+\omega }, 
\end{eqnarray*} 
and then  subtract  the resulting equation by \eqref{+eq2}. Then 
the   function $f_m$, defined in \eqref{funoffs}   solves 
\begin{eqnarray*}
&&	\Big(\partial_t    +\inner{u^{s}+u}\partial_x+v\partial_y  -\partial_y^2  -\eps\partial_x^2 \Big)  f_m \\
&=& -\chi_1\sum_{k=1}^m {m\choose k}\inner{\partial_x^k u} \partial_x^{m-k+1}\omega-\chi_1 \sum_{k=1}^{m-1} {m\choose k}\inner{\partial_x^kv}  \partial_x^{m-k}\partial_y \omega \\
&&+\chi_1 a \sum_{k=1}^m {m\choose k}\inner{\partial_x^k u} \partial_x^{m-k+1}u+\chi_1 a \sum_{k=1}^{m-1} {m\choose k}\inner{\partial_x^kv}   \partial_x^{m-k} \omega\\
&&+\chi_1'v\partial_x^m\omega- 2\chi_1'\partial_x^m\partial_y\omega -\chi_1''\partial_x^m\omega-a \inner{\chi_1'v\partial_x^mu- 2\chi_1'\partial_x^m\omega -\chi_1''\partial_x^mu}\\
&&-\Big[\partial_t    a   +\inner{u^{s}+u}\partial_x  a  +v\partial_y  a -\partial_y^2 a_\eps  -\eps\partial_x^2 a \Big]\chi_1 \partial_x^m u\\
&&+ 2 \chi_1 \inner{\partial_ya}    \partial_x^m \omega+ 2 \chi_1' \inner{\partial_ya }    \partial_x^m u +2\eps \chi_1 \inner{\partial_xa }  \partial_x^{m+1} u .
\end{eqnarray*}
On the other hand we notice that, for any $y\in$ supp $\chi_1$,
\begin{eqnarray*}
\partial_t    a  +\inner{u^{s}+u}\partial_x  a  +v\partial_y  a -\partial_y^2 a   -\eps\partial_x^2 a  
 =  -\partial_x\omega+\inner{\partial_x u}a    
+2a  \partial_ya  +2\eps\frac{\partial_x \omega}{\omega^s+\omega}\partial_x a.
\end{eqnarray*}
Then combining the above equations  completes the proof.
\end{proof}

\begin{lemma}[Equation for $h_m$]\label{fme}
Let  $h_m$ be given in \eqref{repreofhm}. Then
we have 
\begin{eqnarray*}
	&&\Big(\partial_t    +\inner{u^{s}+u}\partial_x +v\partial_y  -\partial_y^2  -\eps\partial_x^2 \Big) h_m\\
	&=&P_{m}+2\inner{\partial_yb  }\partial_y (\chi_2\partial_x^m\omega)+2\eps\inner{\partial_xb }\partial_x (\chi_2\partial_x^m\omega)\\
	&&+ \chi_2 b   \sum_{k=1}^m {m\choose k}\inner{\partial_x^k u} \partial_x^{m-k+1} \omega+\chi_2b \sum_{k=1}^{m-1} {m\choose k}\inner{\partial_x^kv} \partial_x^{m-k}\partial_y \omega \\ 
	&&  -b \chi_2 '  v\partial_x^m \omega+  b \chi_2''  \partial_x^m \omega+ 2b  \chi_2' \partial_x^m \partial_y \omega \\
	&& -\chi_2\sum_{k=1}^m {m\choose k}\inner{\partial_x^k u} \partial_x^{m-k+1} \partial_y\omega-\chi_2\sum_{k=1}^{m-1} {m\choose k}\inner{\partial_x^kv} \partial_x^{m-k}\partial_y^2\omega\\
	&& +  \chi_2 'v\partial_x^m \partial_y\omega-  \chi_2'' \partial_x^m \partial_y\omega- 2\chi_2' \partial_x^m \partial_y^2 \omega  -\chi_2g_{m+1}
,
\end{eqnarray*}	
where $b=\frac{\partial_y^2\omega^{s}+\partial_y^2\omega}{\partial_y\omega^{s}+\partial_y\omega}$ and 
\begin{eqnarray*}
	P_{m}&=& \frac{2\Big(\inner{ \omega^s+ \omega}\partial_x \partial_y\omega-\inner{\partial_x u} \inner{\partial_y^2\omega^s+\partial_y^2\omega}\Big)\chi_2\partial_x^m\omega}{\partial_y\omega^s+\partial_y\omega}\\
	&&-\frac{  \Big(\omega\partial_x \omega-\inner{\partial_x u}  \inner{\partial_y\omega^s+\partial_y\omega}\Big)\inner{\partial_y^2\omega^s+\partial_y^2\omega}\chi_2\partial_x^m\omega}{\inner{\partial_y\omega^s+\partial_y\omega}^2}\\
	&&-  \frac{ 2 \Big(  \inner{\partial_y^3\omega^s+\partial_y^3\omega}  \inner{\partial_y^2\omega^s+\partial_y^2\omega} + \eps\inner{\partial_x\partial_y^2\omega}\partial_{x} \partial_y\omega\Big)\chi_2\partial_x^m\omega}{\inner{\partial_y\omega^s+\partial_y\omega}^2}\\
	&&+\frac{  2\Big(\inner{\partial_y^2\omega^s+\partial_y^2\omega}^2+ \eps\inner{ \partial_x\partial_y\omega}^2\Big)\inner{\partial_y^2\omega^s+\partial_y^2\omega}\chi_2\partial_x^m\omega}{\inner{\partial_y\omega^s+\partial_y\omega}^3}.
	\end{eqnarray*}
\end{lemma}

\begin{proof}
Observe $\omega=\partial_y u$ and $ \partial_y\omega$ solve the following equations:
\begin{equation}
\label{omegeq}
	\partial_t   \omega+\inner{u^{s}+u}\partial_x\omega+v\inner{\partial_y\omega^s+\partial_y\omega}-\partial_y^2 \omega-\eps\partial_x^2 \omega=0,
\end{equation}
and
\begin{equation}
\label{zetaeq}
\begin{aligned}
 \partial_t ( \partial_y\omega)+\inner{u^{s}+u}\partial_x( \partial_y\omega)+v\inner{\partial_y^2 \omega^s+\partial_y^2 \omega}-\partial_y^2 ( \partial_y\omega)-\eps\partial_x^2 ( \partial_y\omega)=-g_1,
\end{aligned}
\end{equation}
by recalling  $g_1= \inner{\omega^{s}+\omega}\partial_x\omega-\inner{\partial_y\omega^s+\partial_y\omega}\partial_xu.$
Now we perform $\chi_2\partial_x^m, m\geq 1,$ on both sides of \eqref{omegeq}-\eqref{zetaeq}, to obtain that
\begin{eqnarray*}
&&	\Big(\partial_t    +\inner{u^{s}+u}\partial_x +v\partial_y  -\partial_y^2  -\eps\partial_x^2 \Big)\chi_2 \partial_x^m \omega+ \chi_2\inner{\partial_y\omega^s+\partial_y \omega} \partial_x^m v\\
	&=& -\chi_2\sum_{k=1}^m {m\choose k}\inner{\partial_x^k u} \partial_x^{m-k+1} \omega-\chi_2\sum_{k=1}^{m-1} {m\choose k}\inner{\partial_x^kv} \partial_x^{m-k}\partial_y \omega\\
	&&+  \chi_2 'v\partial_x^m \omega-  \chi_2'' \partial_x^m \omega- 2\chi_2' \partial_x^m \partial_y \omega ,
\end{eqnarray*}
and
\begin{eqnarray*}
&&	\Big(\partial_t    +\inner{u^{s}+u}\partial_x +v\partial_y  -\partial_y^2  -\eps\partial_x^2 \Big) \chi_2 \partial_x^m\partial_y \omega+\chi_2\inner{\partial_y^2 \omega^s+\partial_y^2\omega} \partial_x^m v\\
	&=& -\chi_2\sum_{k=1}^m {m\choose k}\inner{\partial_x^k u} \partial_x^{m-k+1} \partial_y\omega-\chi_2\sum_{k=1}^{m-1} {m\choose k}\inner{\partial_x^kv} \partial_x^{m-k}\partial_y^2\omega\\
	&& +  \chi_2 'v\partial_x^m \partial_y\omega-  \chi_2'' \partial_x^m \partial_y\omega- 2\chi_2' \partial_x^m \partial_y^2 \omega  -\chi_2g_{m+1}.
\end{eqnarray*}
Now we multiply  the first equation by    $(\partial_y^2 \omega^s+\partial_y^2\omega)/ (\partial_y \omega^s+\partial_y\omega)$, and then subtract the obtained
 equation by  the second equation. This gives the equation for $h_m $: 
\begin{eqnarray*}
	&&\Big(\partial_t    +\inner{u^{s}+u}\partial_x +v\partial_y  -\partial_y^2  -\eps\partial_x^2 \Big) h_m\\
	&=&  \inner{-\partial_tb -\inner{u^s+u}\partial_xb -v\partial_yb+\partial_y^2b +\eps\partial_x^2b }\chi_2\partial_x^m\omega\\
	&&+2\inner{\partial_yb }\partial_y (\chi_2\partial_x^m\omega)+2\eps\inner{\partial_xb  }\partial_x (\chi_2\partial_x^m\omega)\\
	&&+ \chi_2 b   \sum_{k=1}^m {m\choose k}\inner{\partial_x^k u} \partial_x^{m-k+1} \omega+\chi_2b  \sum_{k=1}^{m-1} {m\choose k}\inner{\partial_x^kv} \partial_x^{m-k}\partial_y \omega \\ 
	&&  -b \chi_2 '  v\partial_x^m \omega+  b \chi_2''  \partial_x^m \omega+ 2b \chi_2' \partial_x^m \partial_y \omega \\
	&& -\chi_2\sum_{k=1}^m {m\choose k}\inner{\partial_x^k u} \partial_x^{m-k+1} \partial_y\omega-\chi_2\sum_{k=1}^{m-1} {m\choose k}\inner{\partial_x^kv} \partial_x^{m-k}\partial_y^2\omega\\
	&& +  \chi_2 'v\partial_x^m \partial_y\omega-  \chi_2'' \partial_x^m \partial_y\omega- 2\chi_2' \partial_x^m \partial_y^2 \omega  -\chi_2g_{m+1}.
\end{eqnarray*} 
Finally,  we use the equation \eqref{zetaeq} to compute 
\begin{eqnarray*}
	&&\com{-\partial_tb -\inner{u^s+u}\partial_xb -v\partial_yb +\partial_y^2b +\eps\partial_x^2b }\chi_2\partial_x^m\omega \\
	&=& \frac{2\Big(\inner{ \omega^s+ \omega}\partial_x \partial_y\omega-\inner{\partial_x u} \inner{\partial_y^2\omega^s+\partial_y^2\omega}\Big)\chi_2\partial_x^m\omega}{\partial_y\omega^s+\partial_y\omega}\\
	&&-\frac{  \Big(\omega\partial_x \omega-\inner{\partial_x u}  \inner{\partial_y\omega^s+\partial_y\omega}\Big)\inner{\partial_y^2\omega^s+\partial_y^2\omega}\chi_2\partial_x^m\omega}{\inner{\partial_y\omega^s+\partial_y\omega}^2}\\
	&&-  \frac{ 2 \Big(  \inner{\partial_y^3\omega^s+\partial_y^3\omega}  \inner{\partial_y^2\omega^s+\partial_y^2\omega} + \eps\inner{\partial_x\partial_y^2\omega}\partial_{x} \partial_y\omega\Big)\chi_2\partial_x^m\omega}{\inner{\partial_y\omega^s+\partial_y\omega}^2}\\
	&&+\frac{  2\Big(\inner{\partial_y^2\omega^s+\partial_y^2\omega}^2+ \eps\inner{ \partial_x\partial_y\omega}^2\Big)\inner{\partial_y^2\omega^s+\partial_y^2\omega}\chi_2\partial_x^m\omega}{\inner{\partial_y\omega^s+\partial_y\omega}^3}.
\end{eqnarray*}
Then combining the three equations above we obtain the desired equation of $h_m$. 
\end{proof}

\begin{lemma}[Equation for $ g_m$]
\label{lemgm}
Let $g_m$ be given in  \eqref{gm}. Then we have 
\begin{eqnarray*}
	&&\Big(\partial_t    + \inner{u^{s}+u} \partial_x +v\partial_y  -\partial_y^2 -\eps\partial_x^2   \Big) g_m\\
	&=&-\sum_{j=1}^{m-1}{{m-1}\choose j} \inner{\partial_x^j u}\  g_{m-j+1} -\sum_{j=1}^{m-1}{{m-1}\choose j} \inner{\partial_x^j v}\partial_y  g_{m-j} \\
 	&&+2\sum_{j=0}^{m-1}{{m-1}\choose j} \inner{\partial_x^j \partial_y^2\omega^s+\partial_x^j \partial_y^2\omega}\partial_x^{m-j} \omega+2\eps \sum_{j=0}^{m-1}{{m-1}\choose j} \inner{\partial_x^{j+1} \partial_y\omega}\partial_x^{m-j+1} u\\
	&&-2\sum_{j=0}^{m-1}{{m-1}\choose j} \inner{\partial_x^j \partial_y\omega^s+\partial_x^j \partial_y\omega}\partial_x^{m-j} \partial_y\omega-2\eps \sum_{j=0}^{m-1}{{m-1}\choose j} \inner{\partial_x^{j+1} \omega}\partial_x^{m-j+1} \omega.
	 \end{eqnarray*}
\end{lemma}

\begin{proof}

It follows from the equations for velocity and vorticity that   
\begin{eqnarray*}
	\Big(\partial_t    +\inner{u^{s}+u}\partial_x +v\partial_y  -\partial_y^2-\eps\partial_x^2   \Big) \partial_x u+\inner{\omega^s+\omega } \partial_x v
	=- \inner{\partial_x u } \partial_x u, \end{eqnarray*}
and
\begin{eqnarray*}
	\Big(\partial_t    +\inner{u^{s}+u}\partial_x +v\partial_y  -\partial_y^2-\eps\partial_x^2  \Big) \partial_x \omega+ \inner{\partial_y\omega^s+\partial_y \omega} \partial_x v
	=- \inner{\partial_x u} \partial_x \omega.
\end{eqnarray*}
Now we multiply the first equation above by $\partial_y\omega^s+\partial_y\omega$ and the second one by $\omega^s+\omega$, and then subtract one from the other
to have  
\begin{eqnarray*}
		&& \Big(\partial_t    + \inner{u^{s}+u} \partial_x +v\partial_y  -\partial_y^2-\eps\partial_x^2  \Big)
		 g_1\\
		&
		=&  -\inner{\partial_x u}   g_{1} -\Big[\partial_t   \partial_y\omega+ \inner{u^{s}+u}\partial_x \partial_y \omega +v\partial_y  \inner{\partial_y \omega^s+\partial_y \omega }-\partial_y^3 \omega-\eps\partial_x^2 \partial_y\omega\Big]\partial_x u\\
		&&+ 2\inner{\partial_y^2\omega^s+\partial_y^2\omega} \partial_x \omega +2\eps\inner{\partial_x \partial_y\omega} \partial_x^2u-2\inner{\partial_y\omega^s+\partial_y\omega} \partial_x \partial_y\omega-2\eps\inner{ \partial_x\omega} \partial_x^2 \omega
  \\
  &
		=&    2\inner{\partial_y^2\omega^s+\partial_y^2\omega} \partial_x \omega +2\eps\inner{\partial_x \partial_y\omega} \partial_x^2u-2\inner{\partial_y\omega^s+\partial_y\omega} \partial_x \partial_y\omega-2\eps\inner{ \partial_x\omega} \partial_x^2 \omega,
\end{eqnarray*} 
where in the last equality we have used the fact that 
\begin{eqnarray*}
	\partial_t   \partial_y\omega+ \inner{u^{s}+u}\partial_x \partial_y \omega +v\partial_y  \inner{\partial_y \omega^s+\partial_y \omega }-\partial_y^3 \omega-\eps\partial_x^2 \partial_y\omega =-  g_1.
\end{eqnarray*}
Then applying $\partial_x^{m-1}$ to the equation  yields the equation for $ g_m.$
\end{proof}

\bigskip
\noindent {\bf Acknowledgements.} Some part of the work was done when the first
author    visited  the City   University of Hong Kong, and he would like to thank their hospitality. 
The research of the first author was supported by NSF of China(11422106) and Fok Ying Tung Education Foundation (151001). And the research of the second
author was supported by the General Research Fund of Hong Kong, CityU
 No.11320016.

\end{document}